\documentclass[11pt]{article}

\usepackage{latexsym}
\usepackage{amssymb}
\usepackage{amsthm}
\usepackage{amscd}
\usepackage{amsmath}
\usepackage{mathrsfs}
\usepackage{graphicx}
\usepackage{shuffle}
\usepackage[colorlinks=true, pdfstartview=FitV, linkcolor=blue, citecolor=blue, urlcolor=blue]{hyperref}
\usepackage{mathtools}
\usepackage[bbgreekl]{mathbbol}
\usepackage{mathdots}
\usepackage{ytableau}
\usepackage[enableskew,vcentermath]{youngtab}
\usepackage{tikz-cd}

\usepackage{mathtools}
\usepackage{amssymb}

\usepackage[all]{xy}
\input xy \xyoption{frame}
\xyoption{dvips}

\usepackage[colorinlistoftodos]{todonotes}
\usetikzlibrary{chains,scopes}
\usetikzlibrary{shapes.geometric,positioning}

\usepackage{stmaryrd}

\DeclareSymbolFontAlphabet{\mathbb}{AMSb}
\DeclareSymbolFontAlphabet{\mathbbl}{bbold}

\newcommand{\U}{\mathsf{U}}
\def\Sp{\mathsf{Sp}}
\def\O{\mathsf{O}}

\definecolor{darkred}{rgb}{0.7,0,0} 
\newcommand{\defn}[1]{{\color{darkred}\emph{#1}}} 

\renewcommand{\mid}{:}

\numberwithin{equation}{section}
\allowdisplaybreaks[1]
\UseCrayolaColors

\theoremstyle{definition}
\newtheorem* {theorem*}{Theorem}
\newtheorem* {corollary*}{Corollary}
\newtheorem* {conjecture*}{Conjecture}
\newtheorem{theorem}{Theorem}[section]

\newtheorem{problem}[theorem]{Problem}
\theoremstyle{definition}

\newtheorem* {example*}{Example}

\newtheorem{lemma}[theorem]{Lemma}
\theoremstyle{definition}
\newtheorem{definition}[theorem]{Definition}
\theoremstyle{definition}

\newtheorem{conjecture}[theorem]{Conjecture}
\newtheorem{proposition}[theorem]{Proposition}
\newtheorem{corollary}[theorem]{Corollary}

\newtheorem{remark}[theorem]{Remark}
\newtheorem*{remark*}{Remark}
\theoremstyle{definition}
\newtheorem {example}[theorem]{Example}
\theoremstyle{definition}

\theoremstyle{definition}

\theoremstyle{definition}

\xyoption{dvips}

\def\modu{\ (\mathrm{mod}\ }

\def\({\left(}
\def\){\right)}

\newcommand{\CC}{\mathbb{C}}
\newcommand{\QQ}{\mathbb{Q}}

\newcommand{\cK}{\mathcal{K}}

\newcommand{\cR}{\mathcal{R}}

\newcommand{\cC}{\mathcal{C}}

\def\cX{\mathcal{X}}
\def\cY{\mathcal{Y}}

\def\CC{\mathbb{C}}
\def\RR{\mathbb{R}}
\def\ZZ{\mathbb{Z}}

\def\GL{\textsf{GL}}

\def\ch{\operatorname{ch}}

\newcommand{\g}{\mathfrak{g}}

\def\fk{\mathfrak}

\def\barr{\begin{array}}
\def\earr{\end{array}}
\def\ba{\begin{aligned}}
\def\ea{\end{aligned}}
\def\be{\begin{equation}}
\def\ee{\end{equation}}

\def\Cyc{\mathsf{Cyc}}

\def\qquand{\qquad\text{and}\qquad}
\def\quand{\quad\text{and}\quad}

\newcommand{\gl}{\mathfrak{gl}}

\def\hs{\hspace{0.5mm}}

\def\fkS{\fk S}

\def\ben{\begin{enumerate}}
\def\een{\end{enumerate}}

\def\bei{\begin{itemize}}
\def\eei{\end{itemize}}

\def\hs{\hspace{0.5mm}}

\def\fpf{{\textsf{fpf}}}

\def\D{\hat D}

\def\Des{\mathrm{Des}}

\def\Ifpf{I^{\fpf}}

\def\x{\textbf{x}}

\def\e{\textbf{e}}

\newcommand{\Fl}{\textsf{Fl}}

\newcommand{\cA}{\mathcal{A}}
\newcommand{\cB}{\mathcal{B}}

\def\arcstart{\ \xy<0cm,-.15cm>\xymatrix@R=.1cm@C=.3cm }
\newcommand{\arcstartc}[1]{\ \xy<0cm,-.15cm>\xymatrix@R=.1cm@C=#1cm}

\def\q{\mathfrak{q}}
\def\qq{\mathfrak{q}^+}

\def\Inv{\mathsf{Inv}}
\def\Marked{\mathsf{Marked}}

\newcommand{\row}{\operatorname{row}}

\def\simCK{\overset{\textsf{CK}}\sim}

\def\simOCK{\mathbin{\overset{\textsf{O}}\sim}}

\def\simFCK{\mathbin{\overset{\textsf{Sp}}\sim}}

\newcommand{\weight}{\operatorname{wt}}

\def\row{\textsf{row}}
\def\revrow{\textsf{revrow}}
\def\col{\textsf{col}}

\def\path{\mathsf{path}_{\textsf{OEG}}}

\def\whSym
{\mathfrak{m}\textsf{Sym}}

\def\whQSym
{\mathfrak{m}\textsf{QSym}}

\def\D{\textsf{D}}
\def\SD{\textsf{SD}}

\def\fkD{\fk D}

\def\PP{\mathbb{P}}
\def\NN{\mathbb{N}}

\def\EG{\textsf{EG}}

\def\PO{P_{\textsf{EG}}^\O}

\newcommand{\ytab}[1]{
\ytableausetup{boxsize = .4cm,aligntableaux=center}
{\small\begin{ytableau}  #1  \end{ytableau}}
}

\newcommand{\ytabsmall}[1]{
\ytableausetup{boxsize = .35cm,aligntableaux=center}
{\small\begin{ytableau}  #1  \end{ytableau}}
}

\def\pair{\textsf{pair}}

\def\fkSO{\fkS^{\O}}
\def\fkSS{\fkS^{\Sp}}

\def\T{\mathsf{T}}

\def\BB{\mathbb{B}}

\def\unprime{\mathsf{unprime}}

\newcommand{\abs}[1]{\lvert #1 \rvert}

\numberwithin{equation}{section}
\allowdisplaybreaks[1]
\UseCrayolaColors

\renewcommand{\ytab}[1]{
\ytableausetup{boxsize = .5cm,aligntableaux=center}
{\small\begin{ytableau}  #1  \end{ytableau}}
}

\def\Ifpf{I^{\mathsf{fpf}}}

\def\cA{\mathcal{A}_{\mathsf{inv}}}

\def\bar{\overline}

\def\half{\mathsf{half}_\leq}
\def\shalf{\mathsf{half}_<}
\def\cA{\mathcal{A}}

\usepackage{bbm}

\def\rf{\mathrm{RF}}
\def\brf{\mathrm{BRF}}

\def\upORF{{\mathrm{u}\rf}{}^\O_n}
\def\orf{\mathrm{RF}^\O}
\def\borf{\mathrm{BRF}^\O}

\def\sprf{\mathrm{RF}^\Sp}
\def\bsprf{\mathrm{BRF}^\Sp}

\def\RF{\rf_n}
\def\BRF{\brf_n}
\def\ORF{\orf_n}
\def\SpRF{\sprf_n}

\def\rev{\mathrm{rev}}

\renewcommand{\mid}{:}

\usepackage{listings}
\lstdefinelanguage{Sage}[]{Python}
{morekeywords={False,sage,True},sensitive=true}
\definecolor{dblackcolor}{rgb}{0.0,0.0,0.0}
\definecolor{dbluecolor}{rgb}{0.01,0.02,0.7}
\definecolor{dgreencolor}{rgb}{0.2,0.4,0.0}
\definecolor{dgraycolor}{rgb}{0.30,0.3,0.30}

\newcommand{\flagvar}{\Fl}  
\newcommand{\Schub}{\mathfrak{S}}
\newcommand{\key}{\kappa}

\newcommand{\SymO}{\mathsf{SyM}}

\newcommand{\iso}{\cong}
\newcommand{\xx}{\mathbf{x}}
\newcommand{\zero}{\mathbf{0}}
\newcommand{\bop}{{\color{blue}(}}
\newcommand{\bcp}{{\color{blue})}}

\newcommand{\genop}{\mathfrak{o}}  

\def\qsymbol{\mathsf{Q}}
\def\qkey{\kappa^\qsymbol}

\def\psymbol{\mathsf{P}}
\def\pkey{\kappa^\psymbol}

\def\fkR{\mathfrak{R}}

\def\BSpRF{\bsprf_n}
\def\BORF{\borf_n}

\def\reverse{\mathsf{reverse}}
\def\cY{\mathcal{Y}}

\def\gap{\hspace{0.5mm} /\hspace{0.5mm}}
\def\WD{\mathsf{WD}}

\definecolor{mycolor}{rgb}{0.95,0.95,0.95}
\tikzset{
bnode/.style={
  text width=15mm,
  align=center,
  draw=black,
  fill=white
  },
onode/.style={
  text width=15mm,
  align=center,
  fill=mycolor
  },
}
\usepackage{fullpage}
\usepackage[nomarkers,figuresonly,nofiglist]{endfloat}

\ytableausetup{centertableaux}

\usepackage[backend=biber,url=false,isbn=false]{biblatex}
\begin{document}
\title{Crystals for shifted key polynomials}
\author{
Eric Marberg \\ Department of Mathematics \\  Hong Kong University of Science and Technology \\ {\tt eric.marberg@gmail.com}
\and 
Travis Scrimshaw \\ Department of Mathematics \\ Hokkaido University \\ {\tt tcscrims@gmail.com}
}
\date{}

\maketitle

\begin{abstract}
This article continues our study of $P$- and $Q$-key polynomials, which are  (non-symmetric) ``partial'' Schur $P$- and $Q$-functions as well as ``shifted'' versions of key polynomials. Our main results provide a crystal interpretation of $P$- and $Q$-key polynomials, namely, as the characters of certain connected subcrystals of normal crystals associated to the queer Lie superalgebra $\mathfrak{q}_n$. In the $P$-key case, the ambient normal crystals are the $\mathfrak{q}_n$-crystals studied by Grantcharov et al., while in the $Q$-key case, these are replaced by the extended $\mathfrak{q}_n$-crystals recently introduced by the first author and Tong. Using these constructions, we propose a crystal-theoretic lift of several conjectures about the decomposition of involution Schubert polynomials into $P$- and $Q$-key polynomials. We verify these generalized conjectures in a few special cases. Along the way, we establish some miscellaneous results about normal $\mathfrak{q}_n$-crystals and Demazure $\mathfrak{gl}_n$-crystals.
\end{abstract}

\setcounter{tocdepth}{2}
\tableofcontents

\section{Introduction}

  \subsection{Classical background}
  
For a fixed positive integer $n$, denote by $\GL_n := \GL_n(\CC)$ the group of $n \times n$ invertible matrices, $B_n$ its standard Borel subgroup of upper triangular matrices, and $B_n^- := B_n^{\T}$ the opposite Borel subgroup of lower triangular matrices.
Let $\xx = (x_1, x_2, \ldots)$ be a sequence of commuting indeterminates.
Write $S_n$ for the symmetric group of  permutations of $\{1,2,\dots,n\}$ with length function $\ell \colon S_n \to \NN$ and longest element $w_0 := n \dotsm 21 \in S_n$. 
For $w \in S_n$, the \defn{Schubert variety} $X_w$ is the closure of the (right) $B_n$-orbit $B_n^- \backslash B_n^- w_0 wB_n$. This is an $\ell(w)$-dimensional affine space in the \defn{(complete) flag variety} $\flagvar_n := B_n^- \backslash \GL_n$.
Since Schubert varieties give a CW decomposition of $\flagvar_n$, they determine a basis of \defn{Schubert classes} for the cohomology ring $H^*(\flagvar_n; \ZZ)$ \cite[\S3.6]{Manivel}.

Schubert varieties also give rise to an important family of $B_n$-representations called \defn{Demazure modules} $V_w(\lambda)$~\cite{Demazure74}. For each $w \in S_n$ and partition $\lambda$ with at most $n$ parts, $V_w(\lambda)$ is the $B_n$-module of global sections of the principal line bundle on $X_w$ whose fiber is the $1$-dimensional $B_n$-representation of weight $-\lambda$ (restricting the \defn{Borel--Weil construction}; see~\cite[Ch.~23.3]{FultonHarris}).
The \defn{Bott--Samelson resolution} of $X_w$ leads to the \defn{Demazure character formula}~\cite{Andersen85,Demazure74II,Joseph85}
\[
\key_{w,\lambda}  = \ch V_w(\lambda) = \pi_w (x_1^{\lambda_1}  x_2^{\lambda_2} \dotsm x_n^{\lambda_n}),
\]
where
$\pi_w$ is the \defn{isobaric divided difference operator} indexed by $w$ (see Section~\ref{key-sect} for the precise definition). 

The characters $\key_{w,\lambda}$ are also known as \defn{key polynomials} from the joint work of Lascoux and Sch\"utzenberger~\cite{LS1,LS2}. They introduced the notion of the \defn{(right) key} for a semistandard tableau with entries in $\{1,2, \dotsc, n\}$ to give a combinatorial formula for $\key_{w,\lambda}$ using natural objects from the representation theory of $\GL_n$.
This encodes the fact that $\key_{w,\lambda}$ is a ``partial'' Schur function since for $w = w_0$, we have $X_{w_0} = \flagvar_n$
and $V_{w_0}(\lambda) = V(\lambda)$, and so the key polynomial
$
\key_{w_0,\lambda} = s_{\lambda}(x_1, x_2,\dotsc, x_n)
$
is obtained by restricting the symmetric power series $s_\lambda(\xx)$ to $n$ variables.

It is often more natural to index the key polynomial $\key_{w,\lambda} $ by the \defn{weak composition} $\alpha = w\lambda$
obtained by letting $w$ act on $\lambda$ by permuting entry indices. If we write $\key_\alpha$ in place of $\key_{w,\lambda}$, then
then the set $\{\key_{\alpha}\mid \alpha\text{ any weak composition} \}$ is a $\ZZ$-basis for $\ZZ [\xx]$~\cite[Thm.~17]{ReinerShimozono}.

We can encode the representation theory of $\GL_n$ with  combinatorial data by using the \defn{(Kashiwara) crystals}~\cite{Kashiwara90,Kashiwara91} associated to the Lie algebra 
 $\gl_n$ (or more precisely, its associated Drinfel'd--Jimbo quantum group $U_q(\gl_n)$; see also \cite{Lusztig1990, Lusztig1991}).
Work of Kashiwara \cite{Kashiwara93} shows that every finite-dimensional highest weight
 $\GL_n$-representation $V(\lambda)$ has a crystal basis $B(\lambda)$,
and that each Demazure module also has a crystal basis $B_w(\lambda)$ given by intersecting $B(\lambda)$ with $V_w(\lambda)$.
Moreover, it is known~\cite{BBBG21,ReinerShimozono} that the tableaux from Kashiwara's construction are exactly those given by Lascoux and Sch\"utzenberger in~\cite{LS1,LS2}.

In related work~\cite{LS82,LS82II}, Lascoux and Sch\"utzenberger  identified 
polynomial representatives for the cohomology classes of Schubert varieties with several nice properties. These representatives are the now well-known 
\defn{Schubert polynomials} $\Schub_w$ indexed by permutations $w \in S_n$.
One combinatorial definition of  $\Schub_w $ is given by the Billey--Jockusch--Stanley (BJS) formula~\cite{BJS}   using certain bounded decreasing factorizations of $w$, where a \defn{decreasing factorization} is a reduced expression partitioned into consecutive, possibly empty, decreasing subwords.
Taking the stable limit $F_w(\xx) = \lim_{k\to\infty} \Schub_{1^k \times w}$ produces the \defn{Stanley symmetric function}~\cite{Stanley} as a sum over all decreasing factorizations of $w$.

Morse and Schilling~\cite{MorseSchilling} constructed a natural   $\gl_n$-crystal $\RF(w)$
on these factorizations  that intertwines with the Edelman--Green correspondence from \cite{EG}. The existence of this crystal leads to another proof of the 
Schur-positivity property that
$
F_w(\xx) = \sum_{\lambda} c_w^{\lambda} s_{\lambda}(\xx)
$
for some nonnegative integers $c_w^{\lambda} \in \NN$.
By restricting (a twist of) this crystal structure to the factorizations in the BJS formula, Assaf and Schilling~\cite{AssafSchilling} gave a crystal-theoretic proof 
of the key-positivity property
$
\Schub_w = \sum_{\alpha} c_w^{\alpha} \key_{\alpha},
$
where the sum is over some set of weak compositions $\alpha$ with $c_w^{\alpha} \in \NN$.
This recovers a result of Lascoux and Sch\"utzenberger~\cite{LS1} which is restated as \cite[Thm.~4]{ReinerShimozono}.
Specifically, Assaf and Schilling constructed a $\gl_n$-crystal $\BRF(w)$ of bounded decreasing factorizations of $w$,
and proved that it is isomorphic to a direct sum of \defn{Demazure crystals} whose characters are key polynomials.

\subsection{Shifted constructions}

Now, instead of $B_n$-orbits in $\flagvar_n$, we consider the orbits for the symplectic group $\Sp_n := \Sp_n(\CC)$ when $n$ is even
(respectively, the orthogonal group $\O_n := \O_n(\CC)$ when $n$ is any positive integer), which are indexed by the set of fixed-point-free involutions (respectively, all involutions) $z \in S_n$.
The cohomology classes of the closures of these orbits have  polynomial representatives $\{ \Schub_z^{\Sp} \}_z$ and $\{ \Schub_z^{\O} \}_z$ computed by Wyser and Yong~\cite{WyserYong}, with many nice properties in common with Schubert polynomials.

Following \cite{HMP1}, we refer to these representatives as \defn{involution Schubert polynomials}.
These polynomials have stable limits $F_z^{\Sp}(\xx)$ and $F_z^{\O}(\xx)$  analogous to Stanley symmetric functions \cite{HMP1}. Hamaker, the first author, and Pawlowski  showed in~\cite{HMP5,HMP4} that $F_z^\Sp(\xx)$ expands positively in terms of the \defn{Schur $P$-functions} $P_{\lambda}(\xx)$,
while $F_z^\O(\xx)$ expands positively in terms of the \defn{Schur $Q$-functions} $Q_{\lambda}(\xx)$.

In our previous work~\cite{MS2023}, we introduced shifted analogues of key polynomials that we call \defn{$P$-key polynomials} $\pkey_{w,\lambda}$ 
and \defn{$Q$-key polynomials} $\qkey_{w,\lambda}$. These are formed by applying isobaric divided difference operators to certain \defn{dominant} involution Schubert polynomials.
The name was chosen as these polynomials are partial versions of Schur $P$- and $Q$-functions
in the sense that 
$
P_{\lambda}(x_1,x_2,\dots,x_n) = \pkey_{w_0,\lambda}
$
and
$
Q_{\lambda}(x_1,x_2,\dots,x_n) = \qkey_{w_0,\lambda}$
by \cite[Thm.~2.35]{MS2023}.
Part of our interest in $P$- and $Q$-key polynomials stems from the following conjecture.

\begin{conjecture}[{\cite[Conjs.~2.31 and 2.33]{MS2023}}]
\label{fkSS-conj}
  Each involution Schubert polynomial 
  $
  \fkSS_z
  $ (respectivley,   $
  \fkSO_z
  $)
is an $\NN$-linear combination of $P$-key polynomials (respectively, $Q$-key polynomials).
\end{conjecture}

The main purpose of this article is to outline a crystal-theoretic approach to proving this conjecture.
More precisely, we will describe how this statement lifts to a more general conjecture 
about how certain crystals related to involution Schubert polynomials decompose as direct sums of two different kinds of ``queer'' Demazure crystals.

Let us explain these ideas in more detail.
Schur $P$-functions are Schur positive~\cite{Stembridge89} and arise as the characters of  $\gl_n$-crystal structures on
certain sets of  \defn{shifted tableaux}~\cite{HPS}.
These objects are usually disconnected as $\gl_n$-crystals, and 
to turn them into connected crystals, we need to replace $\gl_n$ with the queer Lie superalgebra $\q_n$   introduced in \cite{Kac77}.
For example, Grantcharov \textit{et al.}~\cite{GJKKK15,GJKKK10} have shown that the Schur $P$-function $P_{\lambda}(x_1, \dotsc, x_n)$ is the
character of a connected $\q_n$-crystal on \defn{decomposition tableaux} corresponding to polynomial representation $V(\lambda)$ for $\q_n$.
There is an isomorphic $\q_n$-crystal structure on \defn{shifted (primed) tableaux} \cite{AO20,HPS,Hiroshima,Hiroshima2018}, which 
is closely related to the \defn{type $B$ and $C$ Stanley symmetric functions} studied in \cite{Lam95}.

The true character of $V(\lambda)$ is technically  $Q_{\lambda}(x_1, \dotsc, x_n)$ rather than
 $P_{\lambda}(x_1, \dotsc, x_n)$ (see~\cite{CW12} or \cite{Kac78}), but the difference is just a matter of rescaling since  $2^{\ell(\lambda)} P_{\lambda}(\xx) = Q_{\lambda}(\xx)$.
To encode the extra $2^{\ell(\lambda)}$ factor directly into the relevant character, the first author and Tong introduced a new family of objects called \defn{(extended) $\qq_n$-crystals} in~\cite{MT2021}.

Now, there are BJS-type formulas for $\Schub^{\Sp}_z$ and $\Schub^{\O}_z$ as sums over certain bounded decreasing factorizations~\cite{HMP1}.
Removing the boundedness condition transforms these formulas to generating series for the stable limits $F^{\Sp}_z(\xx)$ and $F^{\O}_z(\xx)$.
There are known constructions of a $\q_n$-crystal $\SpRF(z)$~\cite{Marberg2019b} and a $\qq_n$-crystal $\ORF(z)$ \cite{MT2021} on the sets of reduced factorizations appearing in these generating functions.
Adding back the boundedness condition identifies two distinguished subcrystals $\BSpRF(z)\subseteq \SpRF(z)$ and $\BORF(z)\subseteq \ORF(z)$ consisting of decreasing factorizations of certain reduced words.
Our main results are about these objects, and can be summarized as follows.
\bei
\item We prove some properties of the crystals $\BRF(w)\subseteq \RF(w)$, $\BSpRF(z)\subseteq \SpRF(z)$, and  $\BORF(z)\subseteq \ORF(z)$.
This includes several facts not explicitly stated in prior literature; 
see, for example, Theorems~\ref{reduced-tab-lem}, \ref{sp-reduced-tab-lem}, and \ref{o-reduced-tab-lem}
and Propositions~\ref{gl-highest-lem}, \ref{gl-lowest-lem}, \ref{sp-lowest-prop}, and \ref{o-lowest-prop}.

\item 
The $\q_n$-crystal $\BSpRF(z)$ usually has multiple components, which are classified by certain \defn{$\Sp$-reduced tableaux}; see Proposition~\ref{sp-subcrystal-prop}.
When $z$ is a \defn{dominant} fixed-point-free involution, however, we prove that $\BSpRF(z)$ is connected and its character is the $P$-key polynomial $\pkey_{1, \lambda(z)}$ (Theorem~\ref{sp-crystal-thm}), where $\lambda(z)$ is the  partition shape to $z$.
 We define a \defn{Demazure $\q_n$-crystal} to be any $\q_n$-crystal isomorphic to $\fkD_w \BSpRF(z)$, where $z$ is dominant and $\fkD_w$ is a \defn{crystal Demazure operator} as specified in~\eqref{cDo-eq}. Each $\fkD_w \BSpRF(z)$ is connected and its character is naturally $\pkey_{w,\lambda(z)}$.
In this way Demazure $\q_n$-crystals generalize $P$-key polynomials.

\item Similarly, the $\qq_n$-crystal $\BORF(z)$ usually has multiple components, which are classified by certain \defn{$\O$-reduced tableaux}; see Proposition~\ref{o-subcrystal-prop}.
When $z$ is a \defn{dominant} involution, we show $\BORF(z)$ is connected and its character is the $Q$-key polynomial $\qkey_{1,\lambda(z)}$ (Theorem~\ref{o-crystal-thm}). 
We define a \defn{Demazure $\qq_n$-crystal} to be any $\qq_n$-crystal isomorphic to $\fkD_w \BORF(z)$ for some $w \in S_n$ and dominant $z$.
Each $\fkD_w \BORF(z)$ is connected and its character is naturally $\qkey_{w,\lambda(z)}$,
 so we can view Demazure $\qq_n$-crystals as generalizations of $Q$-key polynomials.


\item We conjecture that $\BSpRF(z)$ is a direct sum of Demazure $\q_n$-crystals (Conjecture~\ref{sp-demazure-conj}) 
and that $\BORF(z)$ is a direct sum of Demazure $\qq_n$-crystals (Conjecture~\ref{o-demazure-conj}). 
Figures~\ref{sp-fig1}, \ref{sp-fig2}, and \ref{o-fig} show data supporting these conjectures, which we verify in a few special cases (Theorem~\ref{thm:single_row_o_tableau}).

\item Finally, we also discuss how Conjectures~\ref{sp-demazure-conj} and \ref{o-demazure-conj} imply a priori more general statements concerning crystals of ``flagged'' decreasing factorizations; compare Conjectures~\ref{sp-demazure-conj2} and \ref{o-demazure-conj2}
with Propositions~\ref{sp-equiv-prop} and \ref{o-equiv-prop}.
\eei

Unlike for $\gl_n$, we are not aware of any theory of Demazure modules or characters for $\q_n$ (or $\qq_n$-crystals) already appearing in the literature.
To explain some potential the representation-theoretic obstructions, recall that Demazure modules (for $\gl_n$) are indecomposable modules of the Borel subalgebra that give a filtration on highest weight representations parameterized by elements of the Weyl group.
For $\q_n$, it is not clear that such a construction would be independent of the choice of Borel subalgebra, and the Weyl group (which is $S_n$ from the even subalgebra $\gl_n$) does not seem to naturally encode the highest weight crystal or a filtration of the representation (such as coming from a BGG-type resolution or a Schubert variety decomposition for the Lie supergroup).
Despite this, from our results, we expect our crystals are good $\q_n$- and $\qq_n$-analogues for Demazure crystals and come from a global crystal basis of a yet-to-be-defined $U_q(\q_n)$ analogue of Demazure modules.

However, there are some differences between these new objects and the classical theory of Demazure $\gl_n$-crystals.
For example, although each Demazure $\q_n$- and $\qq_n$-crystal uniquely embeds into a \defn{normal} crystal of the same type just as in the classical $\gl_n$-case (compare Propositions~\ref{unique-embed-prop}, \ref{unique-embed-prop2}, and~\ref{unique-embed-prop3}), these subcrystals are no longer closed under all raising crystal operators $e_i$.
In particular, it is not straightforward to show directly that each Demazure $\q_n$- and $\qq_n$-crystal is connected.
Our proof of this fact instead relies on a difficult technical property shown in~\cite{GHPS} (see Lemma~\ref{ghps-lem}).
We expect that this reflects how the Borel subalgebras of $\q_n$ are not as well-behaved as those of $\gl_n$ and that our construction is not coming from representation theory.

\subsection{Outline}

This paper is organized as follows.
In Section~\ref{sec:prelims}, we set up our notation and give the precise definitions of (shifted) key polynomials
and (involution) Schubert polynomials used later on.
Section~\ref{sec:gl_crystals} is partially expository and surveys the main properties of normal and Demazure $\gl_n$-crystals.
Sections~\ref{sec:P_crystals} and \ref{sec:Q_crystals} contain our main new results. 
In Section~\ref{sec:Q_crystals} we develop a theory of Demazure $\q_n$-crystals whose characters are $P$-key polynomials. Section~\ref{sec:Q_crystals} then provides a set of complementary results about Demazure $\qq_n$-crystals for $Q$-key polynomials.

\subsection*{Acknowledgments}

The first author was partially supported by Hong Kong RGC grants 16306120 and 16304122.
The second author was partially supported by Grant-in-Aid for JSPS Fellows 21F51028 and for Scientific Research for Early-Career Scientists 23K12983.
We thank Jon Brundan, Takeshi Ikeda, Brendan Pawlowski, Anne Schilling, and Alexander Yong for useful discussions.
The second author also thanks Tomoo Matsumura for the discussion and references regarding Hessenberg varieties and degeneracy loci, and for International Christian University's hospitality during his visits.

\section{Preliminaries}
\label{sec:prelims}
 
In this section we  fix our notation and review some relevant background material.
 Specifically, we will discuss the algebraic definitions of \defn{key polynomials}, \defn{Schubert polynomials},
 and \defn{crystals}.

\subsection{Notation}\label{not-section}

Throughout, $n$ is a positive integer, $[n] := \{1,2,\dotsc,n\}$, $\NN := \{0,1,2,\dots\}$, and $\PP := \{1,2,3,\dots\}$.
Continue to let $\xx := (x_1,x_2,x_3,\ldots)$ be an infinite sequence of commuting variables.
We write $\e_i$ for $i \in \PP$ to denote the $i$-th standard basis element of $\ZZ^n$.

A \defn{weak composition} is an infinite sequence $\alpha = (\alpha_1,\alpha_2,\ldots)$ of nonnegative integers with finite sum.
Given such a sequence, we set $\xx^\alpha := \prod_{i} x_i^{\alpha_i}$ and $\abs{\alpha} := \sum_i \alpha_i$
and say that $\alpha$ is a weak composition of $\abs{\alpha}$.
The \defn{length} $\ell(\alpha)$ of a weak composition $\alpha$ is either the largest index $\ell$ such that $\alpha_{\ell} > 0$,
or $0$ when $\alpha=\emptyset := (0,0,0,\ldots)$ is the unique empty composition. 
In examples, we often write weak compositions $(\alpha_1,\alpha_2,\dotsc,\alpha_{\ell},0,0,\ldots)$
as finite words $\alpha_1\alpha_2\cdots \alpha_{\ell}$.

\begin{remark}
Throughout this article, we identify $\NN^n$ with the set of weak compositions of length at most $n$ and we write $\alpha \in \NN^n$ to indicate that $\alpha$ is such a composition.
In other words, we treat $\NN^n$ as the set of infinite sequences $\alpha = (\alpha_1,\alpha_2,\ldots)$ with $\alpha_i = 0$ for all $i > n$ and \emph{not} as the set of finite tuples $\alpha = (\alpha_1,\alpha_2,\ldots,\alpha_n)$.
This convention makes it sensible to write that $\NN^n \subset \NN^{n+1}$.
\end{remark}

A \defn{partition} is a weak composition that is weakly decreasing.
The \defn{Young diagram} of a partition $\lambda$ is  the set of pairs $\D_{\lambda}:=\{(i, j) \in \PP\times \PP \mid j \leq \lambda_i \}$.
We draw Young diagrams in English convention, viewing its elements as positions in a matrix.
Given a weak composition $\alpha$, let $\lambda(\alpha)$ be the unique partition that can be formed by rearranging its parts.

Let $s_i := (i \: i+1)$ for $i \in \ZZ$ be the permutation of $\ZZ$ that interchanges $i$ and $i+1$ while fixing all other integers.
Define groups
$
 S_\ZZ := \langle s_i \mid i \in \ZZ\rangle \supset
 S_\infty := \langle s_i \mid i \in \PP\rangle\supset
 S_n := \langle s_i \mid i \in [n-1]\rangle.
$
A \defn{reduced word} for $w \in S_\ZZ$ a minimal length sequence of integers $i_1i_2\cdots i_\ell$ with $w=s_{i_1}s_{i_2}\cdots s_{i_{\ell}}$.
The \defn{length} of $w$ is the (finite) length $\ell(w) = \ell$ of any of its reduced words.
For each $w \in S_\ZZ$, we write $\cR(w)$ for its set of reduced words.
An integer $i \in \ZZ$ is a \defn{descent} of $w \in S_\ZZ$ if $w(i) > w(i+1)$.

The group $S_{\infty}$ acts on $\ZZ [\xx]$  by permuting variables, 
and on the set of all weak compositions by permuting positions.
Under these actions, the simple transposition $s_i$ for $i \in \PP$ maps  
\[\ba f = f(\ldots, x_i, x_{i+1}, \ldots) &\mapsto  f(\ldots, x_{i+1}, x_i, \ldots) =: s_if\quad\text{and}\\
\alpha = (\ldots, \alpha_i, \alpha_{i+1}, \ldots) &\mapsto   (\ldots, \alpha_{i+1}, \alpha_i, \ldots) =: s_i\alpha.
\ea\]

Write $\circ$ for the \defn{Demazure product} on $S_\ZZ$, which is the unique associative operation $S_\ZZ \times S_\ZZ \to S_\ZZ$ with $v\circ w =vw$ if $\ell(vw)=\ell(v) + \ell(w)$ and $s_i \circ s_i = s_i$ for all $i \in \ZZ$;
see~\cite[Thm.~7.1]{Humphreys}.
This makes $(S_\ZZ,\circ)$ into a monoid, which contains $(S_\infty,\circ)$ and $(S_n,\circ)$ as submonoids.

There is a unique action of $(S_{\infty}, \circ)$ on weak compositions, also denoted $\circ$, in which the simple transpositions operate as follows.
For $i \in \PP$ and $\alpha$ a weak composition, define $s_i \circ \alpha$ to be $\alpha$ if $\alpha_i  \leq \alpha_{i+1}$, and to be $s_i\alpha$ 
otherwise.
Since we view elements of $\NN^n$ as $0$-padded infinite sequences, 
the notation $w \circ \alpha$ is well-defined for any $w \in S_\infty$ and $\alpha \in \NN^n$;
however, when $\alpha \in \NN^n$ but $w \notin S_n$ we might have $w \circ \alpha \notin \NN^n$.

For each weak composition $\alpha$, there exists a unique minimal-length permutation $u(\alpha)\in S_\infty$ with $\alpha = u(\alpha) \circ \lambda(\alpha) = u(\alpha)   \lambda(\alpha)$. 
This can be computed inductively via the identity $u(s_i \circ \alpha) = s_i \circ u(\alpha)$ for any $i \in \PP$, with base case $u(\lambda(\alpha)) = 1$.

\subsection{Key polynomials}\label{key-sect}

Suppose we have operators $\genop_i$ indexed by $i \in [n-1]$ (or $i \in \PP$) acting on a set $\cY$.
We say that these operators
 \defn{satisfy the braid relations} on a subset $\cX\subseteq \cY$
 if 
$
\genop_i \genop_{i+1} \genop_i \cX = \genop_{i+1} \genop_i \genop_{i+1} \cX$ for all $i$
and
$\genop_i \genop_j \cX = \genop_j \genop_i \cX$
whenever $\abs{i - j} > 1$.
When these conditions hold, for any $w \in S_{n}$ (or $w \in S_\infty$) we can define $\genop_w := \genop_{i_1} \genop_{i_2} \dotsm \genop_{i_{\ell}}$ using any reduced word 
$i_1i_2\cdots i_\ell \in \cR(w)$.

For each $i \in \PP$, let $\partial_i $ and $\pi_i$ be the \defn{divided difference operator} and \defn{isobaric divided difference operators}, respectively,
on $\ZZ[\xx]$ defined by
$
\partial_i f := \frac{f-s_if}{x_i-x_{i+1}}
$
and
$
\pi_i f := \partial_i (x_i f) = \frac{x_i f - x_{i+1} s_i f}{x_i-x_{i+1}}.
$  
The operators $\partial_i$ and $\pi_i$ satisfy the braid relations on $\ZZ[\xx]$ and 
have $\partial_i ^2 =0 $ and $ \pi_i^2=\pi_i$ for all $i \in \PP$.
The following definition originates in work of Demazure \cite{Demazure74II} and Lascoux--Sch\"utzenberger \cite{LS1,LS2},
 using slightly different terminology. Our conventions follow \cite{ReinerShimozono}.

\begin{definition}
The \defn{key polynomial} of a weak composition $\alpha$
is the unique element $\kappa_\alpha \in \ZZ[\xx]$
satisfying $\kappa_\alpha = \pi_w \xx^\lambda$ for all $w \in S_\infty$ and partitions $\lambda$ such that $\alpha = w\circ \lambda$.
\end{definition}

A partition $\lambda$ is \defn{symmetric} if $\lambda=\lambda^\T$.
A weak composition $\alpha$ is \defn{symmetric} if $\lambda(\alpha) = \lambda(\alpha)^\T$.
 The objects in the following definition are the same as the polynomials
 $\pkey_{w,\lambda}$ and $\qkey_{w,\lambda}$ discussed in the introduced, but 
 are now presented through a more streamlined notation following \cite{MS2023}.

\begin{definition}\label{spo-key-def}
Let $\alpha$ be a symmetric weak composition with $u=u(\alpha)$ and $\lambda = \lambda(\alpha)$.
Define
\[ 
\ba
\pkey_{\alpha} := \pi_u \(\prod_{\substack{(i,j) \in \D_{\lambda} \\ i \geq j} } (x_i + x_{j})\)
\quand
\qkey_{\alpha} := \pi_u \(\prod_{\substack{(i,j) \in \D_{\lambda} \\ i >j} } (x_i + x_{j})\).
\ea
\]
We refer to these functions as \defn{$P$- and $Q$-key polynomials}, respectively.
\end{definition}

By~\cite[Prop.~2.15]{MS2023}, we have $\pkey_\alpha = \pi_w \pkey_\lambda$ and $\qkey_\alpha = \pi_w \qkey_\lambda$ for any $w \in S_\infty$ with $\alpha = w\circ \lambda$.
Moreover, the coefficients of $\qkey_{w\lambda}$ are all divisible by $2^{\ell(\lambda)}$.
For example, we have
\[\ba \pkey_{3143} &= \pi_2\pi_1\pi_3 \bigl( (x_1+x_2)(x_1+x_3)(x_1+x_4)(x_2+x_3) \bigr)
 = \kappa_{0022} + \kappa_{0031} + \kappa_{0112},
\\
 \qkey_{2031} &= \pi_2\pi_1\pi_3 \bigl( 4x_1x_2(x_1+x_2)(x_1+x_3) \bigr)
 = 4 \kappa_{103} + 4 \kappa_{202} + 4 \kappa_{1021}.
\ea\]
The positive expansion into key polynomials exhibited here is typical:
if $\alpha$ is any symmetric weak composition,
then $\pkey_{\alpha}$ and  $ \qkey_{\alpha}$ are nonzero linear combinations of key polynomials $\kappa_{\beta}$ with nonnegative integer coefficients~\cite[Thm.~2.9]{MS2023}.

 A symmetric partition $\lambda$ is \defn{skew-symmetric} if for the maximal $i$ such that $(i, i) \in \D_\lambda $, neither $\D_\lambda \cup \{(i,i+1)\}$
nor $\D_\lambda \setminus \{(i,i+1)\}$ is the diagram of a partition distinct from $\lambda$.
Both $\lambda=\emptyset$ and $\lambda=(2,2)$ are skew-symmetric, but
neither $\lambda =(1)$ nor $\lambda=(2,1)$ is skew-symmetric.
A weak composition $\alpha$ is \defn{skew-symmetric} if  $\lambda(\alpha)$ 
is skew-symmetric.
The partition $ \lambda = (4,3,3,1)$  is skew-symmetric,
and $\alpha = (3,0,1,4,0,0,3)$ is a skew-symmetric weak composition with $\lambda(\alpha) = \lambda$.

If $\lambda$ is a symmetric partition that is not skew-symmetric, then there is a unique diagonal box that be added or removed from $\D_\lambda$
to obtain the diagram of a skew-symmetric partition $\mu$, and for this partition it holds that $\pkey_\lambda = \pkey_\mu$.
This implies that if $\alpha$ is a symmetric weak composition that is not skew-symmetric then we can add or subtract one from a single part of $\alpha$ to obtain a skew-symmetric weak composition $\beta$ with $\pkey_\alpha = \pkey_\beta$ (see~\cite[Lem.~2.16]{MS2023}).

For this reason, we usually consider $P$-key polynomials to be indexed by skew-symmetric weak compositions.
This is still not a unique indexing set; 
see~\cite[\S2.2]{MS2023}.
On the other hand, we do not know of distinct symmetric compositions $\alpha\neq\beta$ with 
$\qkey_\alpha=\qkey_\beta$; see~\cite[Conj.~2.18]{MS2023}.

\subsection{Schubert polynomials}\label{schub-sect}

 The \defn{Rothe diagram} of $w \in S_\ZZ$ is the set $D(w) := \{ (i, w(j)) \mid i,j \in \ZZ,\ i<j,\ w(i)>w(j)\}$.
 This finite set has $D(w) \subset \PP\times \PP$ if $w \in S_\infty$ and $D(w) \subset \{ (i,j) \mid i+j \in [n]\}$
 if $w \in S_n$.
For each partition $\lambda$, 
there is a unique $w_\lambda \in S_\infty$ with $D(w_\lambda) = \D_\lambda$,
called the \defn{dominant} permutation of shape $\lambda$; see, e.g.,~\cite[Prop.~4.7]{HMP6}.
As $\lambda$ ranges over all partitions in $\NN^n$ (that is, with at most $n$ nonzero parts), $w_\lambda$ ranges over all $132$-avoiding elements of $S_n$~\cite[Ex.~2.2.2]{Manivel}.
The following definition is well-known  \cite[\S2]{Manivel}:

\begin{definition} 
The \defn{Schubert polynomials} $\fkS_w$ for $w \in S_\infty$ are the unique elements of $\ZZ[\xx]$ 
 such that $\fkS_w = \xx^\lambda$ if $w$ is dominant of shape $\lambda$ and $\partial_i \fkS_{w} =   \fkS_{ws_i}$ if $i \in \PP$ has $w(i)>w(i+1)$.
 \end{definition}

Schubert polynomials are cohomology representatives of Borel orbit closures in the complete flag variety.
  There are analogous \defn{involution Schubert polynomials} introduced in \cite{WyserYong}
  that represent cohomology classes of orbit closures in the same space for the orthogonal and symplectic groups.
To define these polynomials, let $I_\ZZ$ be the set of involutions in $S_\ZZ$, and let $I_\infty$ and $I_n$ be the subsets of elements of $I_\ZZ$ preserving $\PP$ and $[n]$, respectively.
If $\lambda$ is a symmetric partition then the dominant permutation $w_\lambda \in S_\infty$ must be an involution since $D(w)^\T = D(w^{-1})$.

For $z \in I_\ZZ$ let $\cA^\O(z)$ be the set of minimal-length permutations $w \in S_\ZZ$ with  $z = w^{-1}\circ  w$.
The operation $w \mapsto w^{-1} \circ w$ is a surjective map $S_\ZZ \to I_\ZZ$~\cite[\S6.1]{HMP2}, so $\cA^\O(z)$ is always nonempty.
Write $\Cyc(z) $ for the finite set of pairs $(a,b) \in \ZZ\times \ZZ$ with $a < b = z(a)$.
  
\begin{definition}[\cite{WyserYong}]\label{ischub-def}
The \defn{involution Schubert polynomials of orthogonal type} for $z \in I_\infty$ are 
\be\label{atom-eq-o}
\fkSO_z := 2^{\abs{\Cyc(z)}}\sum_{w \in \cA^\O(z)} \fkS_w\in\ZZ[x_1,x_2,\ldots].
\ee
Equivalently, these are the unique polynomials indexed by $z \in I_\infty$ satisfying both
  \be\label{ischub-eq1}
  \partial_i \fkSO_{z} = \begin{cases} 0&\text{if } z(i) < z(i+1), \\ 
    2 \fkSO_{zs_i} &\text{if } z(i) = i+1,\\
    \fkSO_{s_i z s_i} &\text{otherwise,}
    \end{cases}
  \ee
  for all $i \in \PP$ and $\fkSO_z = \qkey_\lambda$ if $z$ is dominant of shape $\lambda$.
\end{definition}

Next let $\Ifpf_\ZZ$ (respectively, $\Ifpf_\infty$) be the $S_\ZZ$-orbit (respectively, $S_{\infty}$-obit) under conjugation of the permutation $1_\fpf$ mapping $i \mapsto i - (-1)^i$ for all $i \in \ZZ$, which is given in cycle notation as
\[
1_\fpf = \cdots (-1 \: 0)(1 \: 2)(3 \: 4)(5 \: 6)\cdots.
\]
Note that $\Ifpf_{\ZZ}$ (respectively $\Ifpf_{\infty}$) is a subset of the group of \emph{all} bijections of $\ZZ$ (respectively $\PP$), also known as the extended permutation group.
In particular, they are not subgroups of $S_{\ZZ}$ or $S_{\infty}$ as their elements do not have a finite number of fixed points.

The Rothe diagram for $z \in \Ifpf_\infty$ is defined in the same way as for elements of $S_\ZZ$.
If $\lambda$ is a skew-symmetric partition, then there is a unique $z^\fpf_\lambda \in \Ifpf_\infty$ with $\{ (i,j) \in D(z) \mid i \neq j\} = \{ (i,j) \in \D_\lambda \mid i\neq j\}$~\cite[Prop.~4.31]{HMP6}, which we call the \defn{dominant} element of $ \Ifpf_\infty$ of shape $\lambda$.
For $z \in \Ifpf_\infty$ let $\cA^\Sp(z)$ be the (nonempty) set of minimal-length permutations $w \in S_\ZZ$ with $z = w^{-1} 1_\fpf w$.   

\begin{definition}[\cite{WyserYong}]\label{fschub-def}
The \defn{involution Schubert polynomials of symplectic type} for $z \in \Ifpf_\infty$
are   \be\label{atom-eq-sp}
\fkSS_z := \sum_{w \in \cA^\Sp(z)} \fkS_w\in\ZZ[\xx].\ee
Equivalently, these are the unique polynomials indexed by $z \in\Ifpf_\infty$ satisfying both
\be\label{fschub-eq1}
  \partial_i \fkSS_{z} = \begin{cases} 0&\text{if }z(i) < z(i+1) \text{ or } z(i)=i+1, \\ 
  \fkSS_{s_i z s_i} &\text{otherwise,}
  \end{cases}
    \ee
    for all $i \in \PP$ and  
 $\fkSS_z = \pkey_\lambda$ if $z$ is dominant of shape $\lambda$.
\end{definition}

Wyser and Yong~\cite{WyserYong} originally constructed $\fkSO_z$ and $\fkSS_z$ just using formulas 
\eqref{ischub-eq1} and~\eqref{fschub-eq1}.
To derive~\eqref{atom-eq-o} and~\eqref{atom-eq-sp} from this, see~\cite{HMP5,HMP1}.
One needs~\cite[Thm.~1.3]{HMP1} and~\cite[Thm.~4.2]{HMP5} to justify the formulas $\fkSO_z = \qkey_\lambda$  and $\fkSS_z = \pkey_\lambda$ for all dominant $z$.  
For other geometric interpretations of these polynomials, see \cite{HMP6,Pawlowski2019}.
For $K$-theoretic generalizations, see \cite{MP2019a,MP2019b,WyserYong}.

\subsection{Normal crystals}
\label{crystal-sect}

This section introduces a non-standard definition of $\g$-crystals for certain Lie (super)algebras $\g$.  
This material will help streamline the presentation of our main results, which concern highest weight crystals and their subcrystals.
As a technical matter, our definition of abstract $\g$-crystals roughly follows Kashiwara for type $\g=\gl_n$~\cite[Def.~1.2.1]{Kashiwara93} and Grancharov \textit{et al.}\ for type $\g=\q_n$~\cite[Def.~1.9]{GJKKK15}.

Throughout, $\g$ be will be one of three specific types $\gl_n$, $\q_n$, or $\qq_n$, each depending on a positive integer parameter $n$.
When discussing types $\q_n$ and $\qq_n$, we require $n\geq 2$.
For each choice of $\g$ there will be an associated index set $I$, which will always contain $[n-1]=\{1,2,\dotsc,n-1\}$ as a subset. The exact description of the index set and other data associated to each $\g$
will be presented in Sections~\ref{gl-sect},~\ref{q-sect}, and~\ref{qq-sect}.

We first explain the definition of a category of \defn{abstract $\g$-crystals} whose objects are nonempty sets $\cB$ with several associated maps.
Each abstract $\g$-crystal $\cB$ is equipped with a family of
\defn{crystal operators} $e_i, f_i \colon \cB \to \cB \sqcup \{\zero\}$ (where $\zero\notin \cB$ is an auxiliary element),
a family of statistics $\varepsilon_i, \varphi_i \colon \cB \to \NN$, 
and a \defn{weight function} $\weight \colon \cB \to \NN^n$.
The crystal operators and statistics are indexed by $i \in I$.
We require these maps to satisfy the following axioms:
\begin{enumerate}
\item[(1)] $\weight(e_i b) = \weight(b) + \e_i - \e_{i+1}$ for all $i \in [n-1]$ and $b \in \cB$ such that $e_i b \neq \zero$;
\item[(2)] $\weight(b)_i - \weight(b)_{i+1} = \varphi_i(b) - \varepsilon_i(b)$ for all $i \in [n-1]$ and $b \in \cB$;
\item[(3)] $\varepsilon_i(e_i b) = \varepsilon_i(b) - 1$ for all $i \in I$ and $b \in \cB$ such that $e_i b \neq \zero$;
\item[(4)] $\varphi_i(e_i b) = \varphi_i(b) + 1$ for all $i \in I$ and $b \in \cB$ such that $e_i b \neq \zero$;
\item[(5)] $e_i b = c$ if and only if $b = f_i c$ for all $i \in I$ and $b,c \in \cB$.
\end{enumerate}
We will often refer to the set $\cB$ as an abstract $\g$-crystal when the crystal operators, statistics $\varepsilon_i$ and $ \varphi_i$, and weight function are clear from context.

For two abstract $\g$-crystals $\cB$ and $\cC$,   a \defn{crystal morphism} $\psi \colon \cB \to \cC$ is a set-theoretic map $\psi \colon \cB \to \cC \sqcup \{\zero\}$ 
such that for all $i \in I$ the following holds:
\begin{enumerate}
\item[(a)] if $b,e_i b \in \cB$ and $\psi(b), \psi(e_i b) \in \cC$ then $\psi(e_i b) = e_i \psi(b)$, and similarly for $f_i$; and
\item[(b)] if $b \in \cB$ and $\psi(b) \in \cC$ then $\varepsilon_i$, $\varphi_i$, and $\weight$ each take the same values on $b$ and $\psi(b)$.
\end{enumerate}
A crystal morphism is \defn{strict} if it commutes with $e_i$ and $f_i$ for all $i \in I$, where it is understood that $e_i \zero = f_i \zero = \zero$.
The morphism $\psi$ is an \defn{embedding} if it is an injective map of sets $\cB \to \cC$ (so that, in particular, $\zero$ is not in the image).
An \defn{isomorphism} $\cB \to \cC$ is a strict morphism that defines a bijective map of sets $\cB \to \cC$.
When $\cB$ is isomorphic to $\cC$, we write $\cB \iso \cC$.

An abstract $\g$-crystal $\cB$ is a \defn{subcrystal} of another abstract $\g$-crystal  $\cC$ if $\cB$ is a subset of $\cC$ and the inclusion map is a crystal embedding. We indicate this situation by writing  $\cB \subseteq \cC$. 
We say $\cB$ is a \defn{full subcrystal} if $\cB \subseteq \cC$ and the embedding $\psi$ is strict.
An abstract $\g$-crystal $\cB$ is \defn{connected} if the only full subcrystal is $\cB$ itself. 
The disjoint union of two abstract $\g$-crystals $\cB$ and $ \cC$ naturally forms a larger abstract $\g$-crystal, which we denote by $\cB \oplus \cC$.

Next, we explain how 
to construct two full subcategories of abstract $\g$-crystals, whose objects will be called \defn{$\g$-crystals} and \defn{normal $\g$-crystals}.
This will involve two pieces of additional data associated with $\g$.
First, for each $\g$ there will be a distinguished abstract $\g$-crystal $\BB$ with $|\BB|<\infty$ called the \defn{standard crystal}.
Second, there will be a tensor product rule that describes an abstract $\g$-crystal structure on the set $\BB^{\otimes m}$ for each $m \in \NN$.
If $m>0$ then the elements of $\BB^{\otimes m}$ are the formal tensors $b_1\otimes b_2\otimes \cdots \otimes b_m$, where each $b_i \in \BB$.
When $m=0$ we interpret $\BB^{\otimes m}$ as the single-element set $\{ \mathbf{1} \}$, where $\weight(\mathbf{1}) = 0 \in \NN^n$.

The way that the crystal operators $e_i $ and $f_i$ are defined on $\BB^{\otimes m}$ will vary for each choice of $\g$.
However, in all cases, the weight function on $\BB^{\otimes m}$ will be given by 
\be\label{weight-eqs}
\weight(b_1 \otimes b_2 \otimes \dotsm \otimes b_m) = \weight(b_1)+\weight(b_2) + \cdots + \weight(b_m) \in \NN^n
\ee
and the statistics $\varepsilon_i, \varphi_i \colon \BB^{\otimes m} \to \NN$ are defined in terms of the crystal operators by
\be\label{string-eqs}
\varepsilon_i(b) := \max \left\{ k\in \NN \mid e_i^kb \neq \zero \right\}
\quand
\varphi_i(b) := \max \left\{ k \in \NN \mid f_i^kb \neq \zero \right\}.
\ee
We can now specify our two full subcategories of interest.

\begin{definition}
For each type $\g$ (with an associated index set $I$, standard crystal $\BB$, and tensor product crystals $\BB^{\otimes m}$ for $m \in \NN$),
the category of \defn{$\g$-crystals} (respectively, \defn{normal $\g$-crystals}) is the smallest full subcategory of abstract $\g$-crystals that is closed under finite direct sums and contains every object isomorphic to a subcrystal (respectively, full subcrystal) of $\BB^{\otimes m}$ for some $m \in \NN$.
\end{definition}

If $\cB \subseteq \BB^{\otimes p}$ and $\cC\subseteq \BB^{\otimes q}$ are nonempty sets for $p,q \in \NN$, then we can identify the set of formal tensors $\cB \otimes \cC := \{ b\otimes c \mid b \in \cB\text{ and }c\in \cC\}$ with a nonempty subset of $\BB^{\otimes (p+q)}$.
The abstract $\g$-crystal structure on  $\BB^{\otimes (p+q)}$ induces an abstract $\g$-crystal structure  on $\cB \otimes \cC$  in which the weight function and statistics $\varepsilon_i$ and $ \varphi_i$ take the values in \eqref{weight-eqs} and \eqref{string-eqs}, but the crystal operators are modified to act as zero whenever they leave  $\cB \otimes\cC$. 
Our notion of $\g$-crystals forms a tensor category in this way.
The tensor product rules for our choices of $\g$ will always have the property that if $\cB \subseteq \BB^{\otimes p}$ and $\cC\subseteq \BB^{\otimes q}$ are full subcrystals, then $\cB \otimes \cC$ is a union of full subcrystals of $\BB^{\otimes (p+q)}$.
We can therefore also view normal $\g$-crystals as a tensor category.

The \defn{character} of a finite $\g$-crystal $\cB$ is the polynomial
\[
\ch (\cB) := \sum_{b \in \cB} \xx^{\weight(b)} \in \NN[x_1, x_2,\dotsc, x_n].
\]
The character behaves functorially in the sense $\ch(\cB \otimes \cC) = \ch(\cB)  \ch(\cC)$, $\ch(\cB \oplus \cC) = \ch(\cB) + \ch(\cC)$, and if $\cB\cong \cC$ then $\ch(\cB) = \ch(\cC)$.

Let $\cB$ be a $\g$-crystal.
For $J \subseteq I$, an element $b \in \cB$ is called \defn{$J$-highest weight} (respectively \defn{$J$-lowest weight}) if $e_i b = \zero$ (respectively $f_i b = \zero$) for all $i \in J$.
We refer to an $I$-highest (respectively $I$-lowest) weight element as \defn{highest weight} (respectively \defn{lowest weight}).
If for all $i \in I$ and $b \in \cB$ it holds that $e_i b = \zero$ (respectively $f_i b = \zero$) if and only if $\varepsilon_i(b) = 0$ (respectively $\varphi_i(b)$), then $\cB$ is \defn{upper seminormal} (respectively \defn{lower seminormal}).
The crystal $\cB$ is \defn{seminormal} if it is upper and lower seminormal.
For $i \in I$, the \defn{$i$-string} containing $b \in \cB$ 
is the set of nonzero elements in the sequence
\[
\ldots, \quad e_i^3 b,\quad e_i^2 b, \quad e_ib, \quad b, \quad f_ib,\quad f_i^2b,\quad f_i^3b,\quad \ldots,
\]
which we view as a connected (but typically not full) subcrystal.
When $\cB$ is seminormal, $\varepsilon_i(b)$ (respectively $\varphi_i(b)$) measures how far $b \in \cB$ is from the start (respectively end) of the $i$-string containing $b$ (which is necessarily unique for any fixed $i \in I$).

\begin{remark}
Our crystal axioms are a slight variation of~\cite[Def.~1.2.1]{Kashiwara93}, and are useful for simplifying the description of $\q_n$-crystals~\cite{GJKKK15} and $\qq_n$-crystals~\cite{MT2021} in the sequel.
Our way of defining the term \defn{normal} is nonstandard, but the resulting property will match the usual definition for what we consider here.
In particular, normal crystals will always be seminormal.

We are also breaking with standard terminology as our $\g$-crystals do not have to correspond to the crystal basis of a (Drinfeld--Jimbo) quantum group module when $\g = \qq_n$. 
However, in the cases when $\g = \gl_n$ and $\g= \q_n$, our normal $\g$-crystals will correspond to quantum group modules.
\end{remark}

It is often useful to visualize an abstract $\g$-crystal $\cB$ by drawing its  \defn{crystal graph}, which is the (weighted) directed graph with vertex set $\cB$ and $I$-labeled edges of the form $b \xrightarrow[\hspace{15pt}]{i} c$
for each $b,c\in \cB$ and $i \in I$ with $f _i b = c$.
The crystal graph completely encodes the crystal operators, but does not uniquely specify the weight function or the statistics $\varepsilon_i$ or $ \varphi_i$.

Finally, we mention how some terminology can be reformulated in terms of the crystal graph.
An $i$-string corresponds to a directed path of maximal length whose edges are all labeled by $i$ in the crystal graph.
A highest (respectively lowest) weight element is a source (respectively sink) vertex.
An abstract $\g$-crystal is connected if its crystal graph is weakly connected as a directed graph.
Relative to the crystal graph, subcrystals correspond to (not necessarily induced) subgraphs while full subcrystals correspond to (weakly) connected components.

\section{Crystals for the general linear Lie algebra}
\label{sec:gl_crystals}

We start this section by reviewing some facts about highest weight $\gl_n$-crystals, which originate in the classical semistandard tableaux description of $\GL_n$-modules.
We then construct a normal $\gl_n$-crystal on the set of ``primed decreasing factorizations'' of reduced words for a fixed permutation.
This slightly generalizes the construction in~\cite{MorseSchilling}. 
Most other crystals in this paper will be derived from these objects, though often with additional crystal operators.

\subsection{Standard crystals and tensor products}\label{gl-sect}

\ytableausetup{boxsize=1.2em}%
Section~\ref{crystal-sect} gave a definition of $\g$-crystals involving an index set $I$, a standard crystal $\BB$, and a crystal structure 
on each tensor power $\BB^{\otimes m}$ for $m \in \NN$.
Here we explain the input data that goes along with this construction 
when $\g=\gl_n$ is   the complex general linear Lie algebra.

For this type, we have $I = [n-1]$.
The \defn{standard $\gl_n$-crystal} is the set $\BB = \left\{ \ytableaushort{1},\ytableaushort{2},\dotsc,\ytableaushort{n}\right\}$ with weight function $\weight(\ytableaushort{i}):=\e_i$ and crystal graph
\[
    \begin{tikzpicture}[xscale=2.4, baseline=-5, yscale=1,>=latex]
      \node at (0,0) (T0) {$\ytableaushort{1}$};
      \node at (1,0) (T1) {$\ytableaushort{2}$};
      \node at (2,0) (T2) {$\ytableaushort{3}$};
      \node at (3,0) (T3) {${\cdots}$};
      \node at (4,0) (T4) {$\ytableaushort{n}$};
      \draw[->,thick]  (T0) -- (T1) node[midway,above,scale=0.75] {$1$};
      \draw[->,thick]  (T1) -- (T2) node[midway,above,scale=0.75] {$2$};
      \draw[->,thick]  (T2) -- (T3) node[midway,above,scale=0.75] {$3$};
      \draw[->,thick]  (T3) -- (T4) node[midway,above,scale=0.75] {$n-1$};
     \end{tikzpicture}.
\]
This means that  $f_i \ytableaushort{i} = \ytableaushort{j}$ if $j = i + 1$ but $f_i \ytableaushort{k} = \zero$ for all $k \neq i$.
The statistics $\varepsilon_i,\varphi_i \colon \BB\to \NN$ for $i \in I$ are defined such that $\varphi_i(\ytableaushort{j}) =1$ if $i=j \in [n-1]$ and otherwise  $\varphi_i(\ytableaushort{j}) =0$,
while $ \varepsilon_i(\ytableaushort{j}) = 1$ if $i=j-1 \in [n-1]$ and otherwise $ \varepsilon_i(\ytableaushort{j}) =0$.
The unique highest and lowest weight elements of $\BB$ are $\ytableaushort{1}$ and $\ytableaushort{n}$, respectively.

It remains to describe the crystal structure on $\BB^{\otimes m}$ for $m\geq 2$.
Recall that the weight function and statistics $\varepsilon_i, \varphi_i$ are given by~\eqref{weight-eqs} and~\eqref{string-eqs}.
The crystal operators are defined inductively by 
the following tensor product rule: if $b \in \BB$ and $c \in \BB^{\otimes(m-1)}$ then
\be
\label{eq:gl_tensor_product}
e_i(b\otimes c) := \begin{cases}
b \otimes (e_i c) &\text{if }\varepsilon_i(b) \leq \varphi_i(c), \\
(e_i b) \otimes c &\text{if }\varepsilon_i(b) > \varphi_i(c),
\end{cases}
\quand
f_i(b\otimes c) := \begin{cases}
b \otimes (f_i c) &\text{if }\varepsilon_i(b) < \varphi_i(c), \\
(f_i b) \otimes c &\text{if }\varepsilon_i(b) \geq \varphi_i(c),
\end{cases}
\ee
where it is understood that $b\otimes \zero = \zero \otimes c = \zero$.
This tensor product uses the left-right convention of Bump--Schilling~\cite{BumpSchilling}, which is the opposite of Kashiwara~\cite{Kashiwara91}.

\begin{remark}
An alternative way to encode the crystal operators on $\BB^{\otimes m}$ is by using the \defn{signature rule}.
Fix some $i \in I$, and consider $b := b_1 \otimes \cdots \otimes b_m \in \BB^{\otimes n}$.
For each box $\ytableaushort{i}$ (respectively, $\ytableaushort{j}$ for $j = i + 1$) we write a right parenthesis ``$\bcp$'' (respectively, a left parenthesis ``$\bop$''), and ignore all other factors.
We then cancel matching pairs ``$\bop \bcp$'' 
until we reach a reduced $i$-signature of the form
\be\label{bcp-eq}
\underbrace{\bcp \bcp \cdots \bcp}_{p \text{ terms}} \, \underbrace{\bop \bop \cdots \bop}_{q\text{ terms}} .
\ee
If $q=0$ (respectively, $p=0$), then $e_ib$ (respectively, $f_ib$) is defined to be $\zero$.
Otherwise, $e_i b$ (respectively $f_i b$) is formed by changing the box corresponding to the leftmost (respectively rightmost) unpaired ``$\bop$'' (respectively ``$\bcp$'') to $\ytableaushort{i}$ (respectively $\ytableaushort{j}$).
If this box is in factor $k$ then $e_i b = b_1 \otimes \cdots \otimes (e_i b_k) \otimes \cdots \otimes b_m$ 
and $f_i b = b_1 \otimes \cdots \otimes (f_i b_k) \otimes \cdots \otimes b_m$.
It follows 
from~\eqref{string-eqs}
that the values of $p$ and $q$ in~\eqref{bcp-eq} are actually $p=\varphi_i(b)$ and $q=\varepsilon_i(b)$.
See \cite[\S2.4]{BumpSchilling} for more details.
\end{remark}

As explained in Section~\ref{crystal-sect}, the above data gives rise to categories of \defn{$\gl_n$-crystals} and \defn{normal $\gl_n$-crystals} 
that are closed under tensor products~\cite{Kashiwara91} (see also~\cite[\S2.3]{BumpSchilling}).
We mention some useful properties of normal $\gl_n$-crystals.

Whether an abstract $\gl_n$-crystal is normal can be detected using local conditions called the \defn{Stembridge axioms}~\cite{Stembridge2003}.
The isomorphism class of any given connected normal $\gl_n$-crystal is determined by the weight $\lambda$ of its unique highest weight element~\cite{Kashiwara91}.
Choose a connected normal $\gl_n$-crystal $\cB(\lambda)$ in the isomorphism class with highest weight $\lambda$. 
The value of $\lambda$ is always a partition in $\NN^n$ and any such partition can occur as the highest weight of some connected normal $\gl_n$-crystal.
The crystal $\cB(\lambda)$ can be identified with the \defn{crystal basis} of a highest weight $U_q(\gl_n)$-module $V(\lambda)$ and its character is the \defn{Schur polynomial}
$
s_{\lambda}(x_1, \dotsc, x_n) := \ch\bigl( \cB(\lambda) \bigr)$~\cite{Kashiwara90}.

Suppose $\cB$ is a tensor power $\BB^{\otimes m}$ of the standard crystal, or more generally any normal $\gl_n$-crystal.
There is an action of $S_n$ on $\cB$ given as follows.
Let $\sigma_i \colon \cB \to \cB$ for $i \in I$ be the map
\be
\label{sigma-def}
\sigma_i(b) 
=
\begin{cases} e_i^{-k}(b) &\text{if }k\leq 0, \\
f_i^k(b) &\text{if }k\geq 0,\end{cases}
\quad\text{where }k := \varphi_i(b) - \varepsilon_i(b) = \weight(b)_i - \weight(b)_{i+1}.
\ee
In terms of the crystal graph, this operator reverses the order of the $i$-strings.
We adopt the convention that $\sigma_w \zero = \zero$ for all $w \in S_n$.

Letting the simple transposition $s_i \in S_n$ act as $\sigma_i$ on $\cB$ extends to a group action~\cite[Thm.~11.14]{BumpSchilling}, and
so for any $w \in S_n$, we can define an operator $\sigma_w := \sigma_{i_1}\sigma_{i_2}\cdots \sigma_{i_{\ell}}$ where $i_1 i_2 \cdots i_{\ell} \in \cR(w)$ is any reduced word.
Applying $\sigma_i$ to $b \in \cB$ affects the weight by interchanging $\weight(b)_i$ and $\weight(b)_{i+1}$; that is, we have $\weight(b) = s_i \weight(b)$.
This recovers the well-known fact that $\ch(\cB)$ is a symmetric polynomial.

Let $u_\lambda$ be the unique highest weight element in $\cB(\lambda)$.
Then $\cB(\lambda)$ has precisely one element of weight $w \lambda$, given by $\sigma_w u_{\lambda}$, for any $w \in S_n$.
In particular, if $w_0 = n\cdots 321 \in S_n$ is the reverse permutation then $\sigma_{w_0} u_{\lambda}$ is the unique element of weight $w_0 \lambda = (\lambda_n,\dotsc,\lambda_2,\lambda_1)$; this element is also the unique lowest weight element of $\cB(\lambda)$.

\subsection{Primed words primer}

Let $\ZZ' := \ZZ - \frac{1}{2}$ and define $x' := x-\frac{1}{2}$ for $x \in \ZZ$, so that 
we have $\lceil x'\rceil = x$ and  $1' < 1 < 2' < 2 < \cdots$ under the natural ordering of $\QQ$.
A number $x \in \ZZ\sqcup\ZZ'$ is \defn{primed} if $x \in \ZZ'$.
\defn{Reversing the prime} on $x \in \ZZ\sqcup \ZZ'$ means to subtract $\frac{1}{2}$ if $x \in \ZZ$ and to add $\frac{1}{2}$ if $x \in \ZZ'$; in symbols, this is the operation
\be
x \mapsto x - \tfrac{1}{2}  + \lceil x \rceil - \lfloor x \rfloor.
\ee
\defn{Swapping the primes} on two elements $x,y \in \ZZ\sqcup\ZZ'$ means to reverse the primes on both numbers if exactly one is in $\ZZ$ and the other in $\ZZ'$ and to leave the numbers unchanged otherwise.
We refer to finite sequences $i_1i_2\cdots i_\ell$ with $i_j \in \ZZ\sqcup \ZZ'$ as \defn{primed words}.
When $i=i_1i_2\cdots i_\ell$ is a primed word, let $\lceil i\rceil := \lceil i_1 \rceil \lceil i_2 \rceil \cdots \lceil i_\ell \rceil$.

Fix an element $w \in S_\ZZ$.
 For $j \in \ZZ$ let $s_{j'} := s_j \in S_\ZZ$.
 Define $\cR^+(w)$ to be the set of primed words $i_1i_2\cdots i_\ell$ with $\lceil i_1i _2\cdots i_\ell\rceil \in \cR(w)$.
 For each primed word $i=i_1i_2\cdots i_\ell \in \cR^+(w)$, let 
  \be\label{abj-eq}
  a_j := s_{i_\ell}s_{i_{\ell-1}}\cdots s_{i_{j+1}}(\lceil i_j\rceil)
  \quand b_j := s_{i_\ell}s_{i_{\ell-1}}\cdots s_{i_{j+1}}(\lceil i_j\rceil + 1)
\quad\text{for each $j \in [\ell]$}\ee and then  define $\Marked(i) := \{ (a_j,b_j) \mid j \in [\ell]\text{ with } i_j \in \ZZ'\}$.
We refer to the elements of $\Marked(i)$ as \defn{marked inversions}. 
 This is reasonable as $\Marked(i) $ is a subset of the usual inversion set 
 \be
 \Inv(w) := \{ (a,b) \in \ZZ\times\ZZ \mid a<b\text{ and }w(a)>w(b)\},
\ee
 which is equal to $\{(a_1,b_1),(a_2,b_2),\dotsc,(a_\ell,b_\ell)\}$ by~\cite[Lem.~2.1.4]{Manivel}.

 \begin{example}\label{marked-ex}
If $w=231$ and $i = 1'2 \in \cR^+(w)$ then 
\[
(a_1,b_1) = (s_2(1), s_2(2)) = (1,3)\quand (a_2,b_2) = (2,3),
\]
so we have $\Marked(i) = \{(1,3)\}$.
If $w=3412$ and $i=23'12' \in \cR(w)$, then 
\[
(a_1,b_1)=(1,4), \quad (a_2,b_2)=(2,4), \quad (a_3,b_3)=(1,3), \quand (a_4,b_4)= (2, 3),
\]
so we have $\Marked(i) = \{ (2,4),(2,3)\}$.
 \end{example}
 
The set of marked inversions can be easily read off from the wiring diagram of $i \in \cR^+(w)$.
The \defn{wiring diagram} of an (unprimed) word $i_1i_2\cdots i_\ell$ is given as follows.
Start by drawing horizontal line segments at height $h$ for each $h \in \ZZ$.
Retain the height $h$ as the label of the right endpoint of each line, and 
divide each line into $l$ segments ordered from left to right.
Then for each $j \in[\ell]$, replace the parallel wires in segment $j$ at heights $i_j$ and $1+i_{j}$ by crossing wires.
To convert this picture into a finite diagram, we omit the flat wires at all heights $h \in \ZZ$
that do not satisfy $i_j \leq h \leq i_k$ for any $j,k \in [\ell]$.
For example, the wiring diagrams of the words $12$ and $2321$ are 
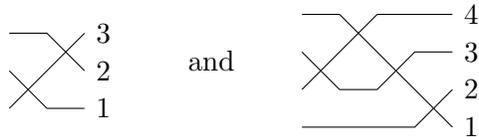
\begin{center}
\begin{tikzpicture}[baseline=(dummynode.base), scale=0.5]
\draw (0,0) -- (1,1) -- (2,2);
\draw (0,1) -- (1,0) -- (2,0);
\draw (0,2) -- (1,2) -- (2,1);
\node (dummynode) at (0, 1.0) {};
\node at (2.5,0) {1};
\node at (2.5,1) {2};
\node at (2.5,2) {3};
\end{tikzpicture}
\qquand
\begin{tikzpicture}[baseline=(dummynode.base), scale=0.5]
\draw (0,0) -- (1,0) -- (2,0) -- (3,0) -- (4,1);
\draw (0,1) -- (1,2) -- (2,3) -- (3,3) -- (4,3);
\draw (0,2) -- (1,1) -- (2,1) -- (3,2) -- (4,2);
\draw (0,3) -- (1,3) -- (2,2) -- (3,1) -- (4,0);
\node (dummynode) at (0, 1.5) {};
\node at (4.5,0) {1};
\node at (4.5,1) {2};
\node at (4.5,2) {3};
\node at (4.5,3) {4};
\end{tikzpicture}
\end{center}
respectively.
The wiring diagram of a primed word $i$ is obtained from that of $\lceil i \rceil$
by marking the unique crossing in segment $j$ when $i_j \in \ZZ'$.
The wiring diagrams of $1'2 $ and $23'12' $ are
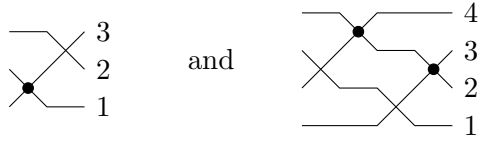
\begin{center}
\begin{tikzpicture}[baseline=(dummynode.base), scale=0.5]
\draw (0,0) -- (1,1) -- (2,2);
\draw (0,1) -- (1,0) -- (2,0);
\draw (0,2) -- (1,2) -- (2,1);
\node (dummynode) at (0, 1.0) {};
\node at (2.5,0) {1};
\node at (2.5,1) {2};
\node at (2.5,2) {3};
\filldraw [black,fill=black] (0.5,0.5) circle (4pt);
\end{tikzpicture}
\qquand
\begin{tikzpicture}[baseline=(dummynode.base), scale=0.5]
\draw (0,0) -- (1,0) -- (2,0) -- (3,1) -- (4,2);
\draw (0,1) -- (1,2) -- (2,3) -- (3,3) -- (4,3);
\draw (0,2) -- (1,1) -- (2,1) -- (3,0) -- (4,0);
\draw (0,3) -- (1,3) -- (2,2) -- (3,2) -- (4,1);
\node (dummynode) at (0, 1.5) {};
\node at (4.5,0) {1};
\node at (4.5,1) {2};
\node at (4.5,2) {3};
\node at (4.5,3) {4};
\filldraw [black,fill=black] (1.5,2.5) circle (4pt);
\filldraw [black,fill=black] (3.5,1.5) circle (4pt);
\end{tikzpicture}\end{center}
for example.
As we see by comparing these pictures with Example~\ref{marked-ex}, if $i \in \cR^+(w)$, then
a pair $(a,b) \in \ZZ\times \ZZ$ belongs to $ \Marked(i)$ if and only if $a<b$ and wires $a$ and $b$ intersect at a marked crossing
in the wiring diagram of $i$.

Recall that \defn{Coxeter--Knuth equivalence} is the transitive closure of the symmetric relation on (unprimed) words with
$
 \cdots acb \cdots \sim \cdots cab \cdots 
$
and
$
\cdots bca \cdots \sim \cdots bac \cdots 
$
for all integers $a<b<c$, along with
$
\cdots aba \cdots \sim \cdots  bab \cdots 
$
for all integers $a$ and $b$ with $\abs{a-b}=1$.
In these expressions, the corresponding ellipses on either side are required to mask identical subwords.
We define \defn{primed Coxeter--Knuth equivalence} to be the transitive closure of 
the symmetric relations
 \be\label{ock2-eq}
  \cdots ACB \cdots \sim \cdots CAB \cdots 
\quand
\cdots BCA \cdots \sim \cdots BAC \cdots 
\ee
for all $A,B,C \in \frac{1}{2}\ZZ$ with $\lceil A \rceil < \lceil B\rceil  < \lceil C  \rceil$, together with the symmetric relations
 \be\label{ock3-eq}
 \ba \cdots aba\cdots \sim  \cdots bab\cdots, \\
 \cdots a'b'a'\cdots \sim  \cdots b'a'b'\cdots,
  \ea
  \quad
  \ba  
\cdots a'ba\cdots &\sim \cdots bab'\cdots,
  \\
    \cdots ab'a\cdots &\sim  \cdots ba'b\cdots,
    \ea\quad
    \ba
      \cdots a'b'a\cdots &\sim  \cdots ba'b'\cdots,
          \\
      \cdots a'ba'\cdots &\sim  \cdots b'ab'\cdots,
      \ea
  \ee
  for all $a,b \in \ZZ$ with $|a-b|=1$. We denote primed Coxeter--Knuth equivalence by $\simCK$.
Two words equivalent under $\simCK$ have the same number of primed letters,
 so primed Coxeter--Knuth equivalence restricts to ordinary Coxeter--Knuth equivalence on unprimed words.
 
 \begin{lemma}
Suppose $w \in S_\ZZ$ and $i,j\in \cR^+(w)$. If $i \simCK j$ then $\Marked(i) = \Marked(j)$.
 \end{lemma}
 
 \begin{proof}
 If $i$ and $j$ differ by a relation in~\eqref{ock2-eq},
then the wiring diagram of $j$ is obtained from that of $i$  by switching the order of the two crossings
in the adjacent positions of $A$ and $C$. These crossings involve disjoint pairs of wires since $|\lceil A \rceil - \lceil C\rceil| >1$,
so $\Marked(i) = \Marked(j)$.
 If $i$ and $j$ differ by a relation in~\eqref{ock3-eq}, then the wiring diagram of $j$ is obtained from that of $i$ 
 by applying  
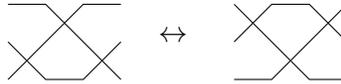
\[
\begin{tikzpicture}[baseline=(dummynode.base), scale=0.5]
\draw (0,0) -- (1,1) -- (2,2) -- (3,2);
\draw (0,1) -- (1,0) -- (2,0) -- (3,1);
\draw (0,2) -- (1,2) -- (2,1) -- (3,0);
\node (dummynode) at (0, 1.0) {};
\end{tikzpicture}
\quad\leftrightarrow
\quad
\begin{tikzpicture}[baseline=(dummynode.base), scale=0.5]
\draw (0,0) -- (1,0) -- (2,1) -- (3,2);
\draw (0,1) -- (1,2) -- (2,2) -- (3,1);
\draw (0,2) -- (1,1) -- (2,0) -- (3,0);
\node (dummynode) at (0, 1.0) {};
\end{tikzpicture}
\]
to three consecutive crossings and then reversing the order of the associated markings, so that the first/middle/last crossing is marked in 
$j$ if and only if the last/middle/first crossing is respectively marked in $i$. 
The wires that intersect at the  first/middle/last crossing in $j$ are the same as  the wires that 
respectively intersect at the last/middle/first crossing in $i$, so again $\Marked(i) = \Marked(j)$.
\end{proof}

\subsection{Crystal structure on reduced factorizations}
\label{rf-sect}

Continue to fix $w \in S_\ZZ$.
Let $\rf(w)$ (respectively, $\rf^+(w)$) be the set of sequences $a=(a^1,a^2,\cdots )$ of 
 decreasing words with concatenation $a^1a^2\cdots\in \cR(w)$
 (respectively, $a^1a^2\cdots\in \cR^+(w)$).
 We refer to the elements of $\rf^+(w)$ as  \defn{(decreasing) reduced factorizations} of $w$.

 For $A \subseteq \Inv(w)$, let $\cR^+(w,A)$ be the set of primed words $i\in \cR^+(w)$ with $\Marked(i)=A$,
 and let $\rf^+(w,A)$ be the set of $a= (a^1,a^2,\ldots) \in \rf^+(w)$ with $a^1a^2\cdots \in \cR^+(w,A)$.
Then we have 
\[
\cR^+(w) = \bigsqcup_{A} \cR^+(w,A),
\quad
\rf^+(w) = \bigsqcup_{A} \rf^+(w,A),
\quand
\rf(w) = \rf^+(w,\emptyset),
\]
where the unions are over all subsets $A\subseteq \Inv(w)$.
Define  $\unprime \colon \rf^+(w) \to \rf(w)$ by
 \be
 \unprime(a) := (\lceil a^1\rceil, \lceil a^2\rceil ,\ldots).
 \ee 

Let $\rf_n^+(w)$ denote the set of  tuples $a=(a^1,a^2,\cdots) \in \rf^+(w)$ with $a^i$ empty for all $i>n$.
Define $\RF(w)\subseteq \rf(w)$ and $\RF^+(w,A) \subseteq \rf^+(w,A)$ 
similarly.
When convenient, we identify $a=(a^1,a^2,\ldots) \in \RF^+(w)$ with the $n$-tuple $(a^1,a^2,\dotsc,a^n)$.
The \defn{weight} of $a   \in \RF^+(w)$ is 
\be\label{weight-eq}\weight(a) := (\ell(a^1),\ell(a^2),\dotsc,\ell(a^n))\in \NN^n.\ee
The set  $\RF^+(w,A)$ can be empty if $n$ is too small. 
Specifically, recall that the \defn{(Lehmer) code} of $w \in S_\ZZ$
consists of the numbers
$ c_i(w) := \abs{ \{ j \in \ZZ \mid i < j \text{ and } w(i) > w(j)\} }$ for $i \in \ZZ$.

\begin{proposition}\label{rf-nonempty-prop}
The set $\RF^+(w,A)$ is nonempty if and only if $c_i(w^{-1}) \leq n$ for all $i \in\ZZ$.
\end{proposition}

\begin{proof}
We have $|\RF^+(w,A)|=|\RF(w)|$.
The map $a \mapsto (\reverse(a^n), \dots, \reverse(a^2), \reverse(a^1))$
is a bijection from $\RF(w)$ to the set of $n$-tuples of strictly increasing words with concatenation in $\cR(w^{-1})$.
The latter set is nonempty if and only if $c_i(w^{-1}) \leq n$ for all $i \in\ZZ$ by~\cite[Rem.~3.4]{Marberg2019b}.
\end{proof}

The goal of this section is to describe a normal $\gl_n$-crystal structure on $\RF^+(w,A)$.
This will involve the following pairing procedure based on a similar construction in~\cite[\S3.2]{MorseSchilling}:

\begin{definition}\label{pair-def}
Suppose $i=i_1i_2\cdots i_M$ and $j=j_1j_2\cdots j_N$ are strictly decreasing primed words.
Form a set of paired letters $\pair(i,j)$ by iterating over the letters in $j$ from right to left; at each iteration,
the current letter $j_q$ is paired with the largest unpaired letter $i_p$ with $\lceil i_p\rceil <\lceil j_q\rceil $ (if such a letter exists)
and then $(i_p,j_q)$ is added to $\pair(i,j)$.
\end{definition}

\begin{example}
If $i=12',11,10,9,6,4,3$ and $j=12,8',5,2',1$ then $\pair(i,j) = \{ (4,5), (6,8'),(11,12)\}$.
\end{example}

Our crystal operators $e_i$ and $f_i$ will be given by the following formulas:

\begin{definition}
\label{ef-def}
Suppose $i \in \PP$ and $a = (a^1,a^2,\ldots) \in \rf^+(w)$. 
Construct $e_i a$ and $f_i a$ from $a$ by the following procedures.
\begin{enumerate}
\item[\defn{$e_i$}:] Set $e_i a := \zero$ if every letter in $a^{i+1}$ is the last term of some $(b,c)\in\pair(a^i,a^{i+1})$.
 
 Otherwise, let $x \in \frac{1}{2}\ZZ$ be the largest unpaired letter in $a^{i+1}$.

Then let $q \in \NN$ be minimal with $\lceil x\rceil +q$ not in $ \lceil a^{i}\rceil $.

Finally, form $e_i(a)$  by removing $x$ from $a^{i+1}$, adding $x+q$ to $a^{i}$ in the position of a decreasing word, and 
 swapping the primes on each pair $(b,c) \in \pair(a^i,a^{i+1})$ with $\lceil x\rceil +q \geq \lceil c\rceil >\lceil x\rceil$.

\item[\defn{$f_i$}:]  Set  $f_i a := \zero$ if every letter in $a^i$ is the first term of some $(b,c)\in\pair(a^i,a^{i+1})$.
 
 Otherwise, let $y \in \frac{1}{2}\ZZ$ be the smallest unpaired letter in $a^i$.

Then let $q \in \NN$ be minimal with $\lceil y\rceil -q$ not in $ \lceil a^{i+1}\rceil $.

Finally, form $f_i(a)$ by removing $y$ from $a^i$, adding $y-q$ to $a^{i+1}$ in the position of a decreasing word, and 
  swapping the primes on each pair $(b,c) \in \pair(a^i,a^{i+1})$ with $\lceil y\rceil > \lceil b\rceil \geq \lceil y\rceil- q$.
  \end{enumerate}
\end{definition}

\begin{example}
We have $e_1(97'651,763') = f_1(7' 651,9763') = \zero$, while
 \[
 (9{\color{red}7'}651,763')  \xmapsto{\ f_1\ }
  (9651,76{\color{red}5'}3')
  \quand
 (9{\color{red}7}6{\color{red}5}41,7{\color{red}6'}52)  \xmapsto{\ f_1\ } 
 (96{\color{red}5'}41,7{\color{red}6}5{\color{red}4}2),
 \]
 with $e_1$ acting in reverse.
Note that  $\weight\bigl( (97'651,763') \bigr) = (5,3)$ and $\weight\bigl( (9651,765'3') \bigr) = (4,4)$.
\end{example}

\begin{remark}
\label{rem:Lusztig_twist}
We can already deduce from known results that the set $\RF(w)$ is a normal $\gl_n$-crystal (when is it is nonempty)
relative to the given crystal operators and weight map.
The \defn{Lusztig involution} is an operation on normal $\gl_n$-crystals that swaps the highest and lowest weight elements of each connected component while interchanging the crystal operators $e_i \leftrightarrow f_{n-i}$; see~\cite{Lenart07} or~\cite[Ex. 5.2]{BumpSchilling}.
Our prospective crystal structure on $\RF(w)$ is given by twisting
the normal $\gl_n$-crystal structure in~\cite[\S3]{MorseSchilling} (and later used in \cite[\S4.3]{AssafSchilling}) by the Lusztig involution.
\end{remark}

To show that $\RF^+(w)$ is a normal $\gl_n$-crystal, 
we will leverage the preceding observations with the following lemma,
which  is clear from the definitions:

\begin{lemma}
\label{lemma:prime_reduction}
The operators $e_i$ and $f_i$ commute with $\unprime$ under the convention $\unprime(\zero) = \zero$.
\end{lemma}

An important property that is not clear from the definitions is that $e_i$ and $f_i$ actually define maps $\rf^+(w) \to \rf^+(w)\sqcup\{\zero\}$.
By  Lemma~\ref{lemma:prime_reduction}, this is equivalent to the claim that $e_i $ and $f_i$ are maps $\rf(w) \to \rf(w)\sqcup\{\zero\}$, which 
can be deduced from results in~\cite{MorseSchilling} via Remark~\ref{rem:Lusztig_twist}.

We want to prove the stronger statement that the crystal operators define maps $\rf^+(w, A) \to \rf^+(w, A) \sqcup\{\zero\}$ for any fixed $A \subseteq \Inv(w)$.
To check this, we need one technical lemma from~\cite{MorseSchilling}.
We include a short proof to explain how our version of this result follows from prior work.

 \begin{lemma}[{See~\cite[\S3.3]{MorseSchilling}}]
 \label{cons-lem}
 Suppose $i \in \PP$ and $a=(a^1,a^2,\cdots ) \in \rf^+(w)$.  
 \ben 
 \item[(a)] If $e_i a \neq \zero$ and $x$ and $q$ are as in Definition~\ref{ef-def},
then
 \bei
\item  $\lceil a^i\rceil $ contains 
 $\lceil x \rceil+q-1,\dotsc, \lceil x \rceil+2, \lceil x \rceil+1,\lceil x \rceil$
but not  $\lceil x \rceil - 1$ or $\lceil x \rceil+q$, and
 \item   $\lceil a^{i+1}\rceil $ contains
 $\lceil x \rceil+q,\dotsc, \lceil x \rceil+2, \lceil x \rceil+1,\lceil x \rceil$
 but not $\lceil x \rceil+q+1$.
 \eei
 
 \item[(b)]
 If $f_ib \neq \zero$ and $y$ and $q$ are as in Definition~\ref{ef-def},
then 
\bei
\item $\lceil a^i\rceil $ contains
 $\lceil y \rceil , \lceil y \rceil-1, \lceil y \rceil-2,\dotsc,\lceil y \rceil-q$
but not $\lceil y \rceil-q-1$, and

  \item $\lceil a^{i+1}\rceil $ contains
$\lceil y \rceil , \lceil y \rceil-1, \lceil y \rceil-2,\dotsc,\lceil y \rceil-q+1$
 but not $\lceil y\rceil +1$ or  $\lceil y \rceil-q$. 
\eei
\een
 \end{lemma}
 
\begin{proof}
No word in $\cR^+(w)$ contains $x$ and $x'$ as adjacent letters,
so no decreasing factor $a^i$ can contain both $x$ and $x'$.
As our pairing ignores all primes, it suffices to demonstrate this lemma when $a \in \rf(w)$, 
and then the desired claims are equivalent to~\cite[Lems.~3.8 and 3.9]{MorseSchilling}.
 \end{proof}

 \begin{lemma}\label{up-lem}
Suppose $A \subseteq \Inv(w)$.
For each $i \in \PP$, 
 both $e_i$ and $f_i$ restrict to maps $\rf^+(w,A) \to \rf^+(w,A)\sqcup\{\zero\}$.
 \end{lemma}
 
  \begin{proof}
First suppose  $a \in \rf^+(w,A)$
and $b := e_i a \neq \zero$.
Let $x$ and $q$ be as in Definition~\ref{ef-def}, and set $y := x+q$.
Consider the section of the wiring diagram of $a^ia^{i+1}$ containing the wires between heights $\lceil x\rceil -1$ and $\lceil y \rceil +2$
and the crossings between the positions of $\lceil y\rceil-1 \in \lceil a^i\rceil$
and $\lceil x\rceil \in \lceil a^{i+1}\rceil$.
We denote this by $\WD(a)$ and define $\WD(b)$ to be the part of the wiring diagram of $b^ib^{i+1}$ in the same location.
Label the wires in both diagrams according to the height of their right endpoints.

The operator $e_i$ acts on $a$  by transforming $\WD(a)$ 
to $\WD(b)$, and if we ignore all primes then this transformation looks like the following picture (where $q=4$):
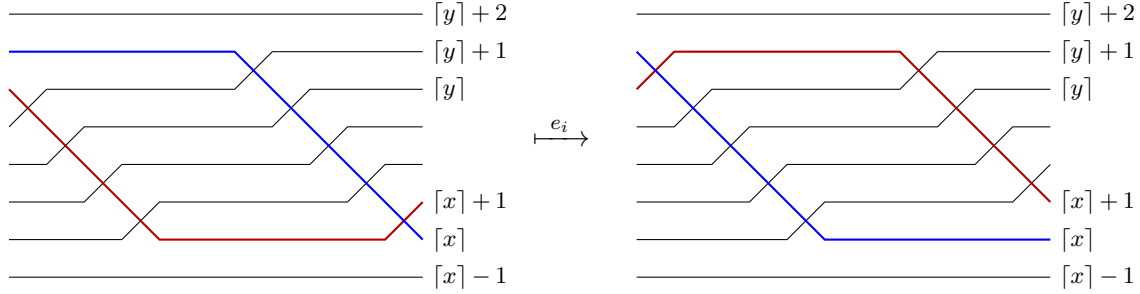
\[
\begin{tikzpicture}[baseline=(dummynode.base), scale=0.5]
\draw (0,0) -- (1,0) -- (2,0) -- (3,0) -- (4,0) -- (5,0) -- (6,0) -- (7,0) -- (8,0) -- (9,0) -- (10,0) -- (11,0);
\draw (0,1) -- (1,1) -- (2,1) -- (3,1) -- (4,2) -- (5,2) -- (6,2) -- (7,2) -- (8,2) -- (9,2) -- (10,3) -- (11,3);
\draw (0,2) -- (1,2) -- (2,2) -- (3,3) -- (4,3) -- (5,3) -- (6,3) -- (7,3) -- (8,3) -- (9,4) -- (10,4) -- (11,4);
\draw (0,3) -- (1,3) -- (2,4) -- (3,4) -- (4,4) -- (5,4) -- (6,4) -- (7,4) -- (8,5) -- (9,5) -- (10,5) -- (11,5);
\draw (0,4) -- (1,5) -- (2,5) -- (3,5) -- (4,5) -- (5,5) -- (6,5) -- (7,6) -- (8,6) -- (9,6) -- (10,6) -- (11,6);
\draw[darkred,thick] (0,5) -- (1,4) -- (2,3) -- (3,2) -- (4,1) -- (5,1) -- (6,1) -- (7,1) -- (8,1) -- (9,1) -- (10,1) -- (11,2);
\draw[blue,thick] (0,6) -- (1,6) -- (2,6) -- (3,6) -- (4,6) -- (5,6) -- (6,6) -- (7,5) -- (8,4) -- (9,3) -- (10,2) -- (11,1);
\draw (0,7) -- (1,7) -- (2,7) -- (3,7) -- (4,7) -- (5,7) -- (6,7) -- (7,7) -- (8,7) -- (9,7) -- (10,7) -- (11,7);
\node (dummynode) at (0, 3.5) {};
\node[text width=10mm,align=left] at (12.25,0) {\footnotesize$\lceil x\rceil-1$};
\node[text width=10mm,align=left] at (12.25,1) {\footnotesize$\lceil x\rceil$};
\node[text width=10mm,align=left]  at (12.25,2) {\footnotesize$\lceil x\rceil+1$};
\node at (12.5,3) {};
\node at (12.5,4) {};
\node[text width=10mm,align=left]  at (12.25,5) {\footnotesize$\lceil y\rceil$};
\node[text width=10mm,align=left]  at (12.25,6) {\footnotesize$\lceil y\rceil+1$};
\node[text width=10mm,align=left]  at (12.25,7) {\footnotesize$\lceil y\rceil+2$};
\end{tikzpicture}
\; \xmapsto[\hspace{15pt}]{e_i} \quad
\begin{tikzpicture}[baseline=(dummynode.base), scale=0.5]
\draw (0,0) -- (1,0) -- (2,0) -- (3,0) -- (4,0) -- (5,0) -- (6,0) -- (7,0) -- (8,0) -- (9,0) -- (10,0) -- (11,0);
\draw (0,1) -- (1,1) -- (2,1) -- (3,1) -- (4,1) -- (5,2) -- (6,2) -- (7,2) -- (8,2) -- (9,2) -- (10,2) -- (11,3);
\draw (0,2) -- (1,2) -- (2,2) -- (3,2) -- (4,3) -- (5,3) -- (6,3) -- (7,3) -- (8,3) -- (9,3) -- (10,4) -- (11,4);
\draw (0,3) -- (1,3) -- (2,3) -- (3,4) -- (4,4) -- (5,4) -- (6,4) -- (7,4) -- (8,4) -- (9,5) -- (10,5) -- (11,5);
\draw (0,4) -- (1,4) -- (2,5) -- (3,5) -- (4,5) -- (5,5) -- (6,5) -- (7,5) -- (8,6) -- (9,6) -- (10,6) -- (11,6);
\draw[darkred,thick] (0,5) -- (1,6) -- (2,6) -- (3,6) -- (4,6) -- (5,6) -- (6,6) -- (7,6) -- (8,5) -- (9,4) -- (10,3) -- (11,2);
\draw[blue,thick] (0,6) -- (1,5) -- (2,4) -- (3,3) -- (4,2) -- (5,1) -- (6,1) -- (7,1) -- (8,1) -- (9,1) -- (10,1) -- (11,1);
\draw (0,7) -- (1,7) -- (2,7) -- (3,7) -- (4,7) -- (5,7) -- (6,7) -- (7,7) -- (8,7) -- (9,7) -- (10,7) -- (11,7);
\node (dummynode) at (0, 3.5) {};
\node[text width=10mm,align=left] at (12.25,0) {\footnotesize$\lceil x\rceil-1$};
\node[text width=10mm,align=left] at (12.25,1) {\footnotesize$\lceil x\rceil$};
\node[text width=10mm,align=left]  at (12.25,2) {\footnotesize$\lceil x\rceil+1$};
\node at (12.5,3) {};
\node at (12.5,4) {};
\node[text width=10mm,align=left]  at (12.25,5) {\footnotesize$\lceil y\rceil$};
\node[text width=10mm,align=left]  at (12.25,6) {\footnotesize$\lceil y\rceil+1$};
\node[text width=10mm,align=left]  at (12.25,7) {\footnotesize$\lceil y\rceil+2$};
\end{tikzpicture}
\]
The features of this example hold in general by 
Lemma~\ref{cons-lem}, which implies that in $\WD(a)$:
 \bei
 \item wires $\lceil x\rceil -1$ and $\lceil y \rceil +2$  are flat;
\item  the first $q$ crossings from the left
 are between wire $\lceil x\rceil +1$ and wires $\lceil y \rceil+ 1$, $\lceil y \rceil$, \dots, $\lceil x \rceil+ 2$;
\item the next $q$ crossings are between wire $\lceil x\rceil$ and 
 wires $\lceil y \rceil+ 1$, $\lceil y \rceil$, \dots, $\lceil x \rceil+ 2$; and
\item the final crossing is between wires $\lceil x\rceil$ and $\lceil x\rceil +1$.
\eei
The same lemma plus the definition of $e_i$ implies that in $\WD(b)$:
  \bei
 \item wires $\lceil x\rceil -1$ and $\lceil y \rceil +2$  are again flat;
\item the first crossing from the left is between wires $\lceil x\rceil$ and $\lceil x\rceil +1$;
\item  the next $q$ crossings 
 are between wire $\lceil x\rceil$ and wires $\lceil y \rceil+ 1$, $\lceil y \rceil$, \dots, $\lceil x \rceil+ 2$; and
\item the last $q$ crossings are between wire $\lceil x\rceil+1$ and 
 wires $\lceil y \rceil+ 1$, $\lceil y \rceil$, \dots, $\lceil x \rceil+ 2$.
\eei
Both wiring diagrams involve the same crossings.
The last crossing in $\WD(a)$ is marked if and only if the first crossing in $\WD(b)$ is marked,
which occurs if and only if $x \in \ZZ'$ (equivalently, $y \in \ZZ'$). 
On the other hand, the first $q$ crossings in the $a$-diagram are successively paired with the next $q$ crossings in $\pair(a^i,a^{i+1})$,
and the primes on these pairs are swapped when constructing $e_i a$.
Given the preceding observations and our description of marked inversions via wiring diagrams,
we deduce that $b \in \rf^+(w,A)$.

If instead $b \in \rf^+(w,A)$ and $a := f_i b \neq \zero$,
then  $a \in \rf^+(w,B)$ for some subset $B \subseteq \Inv(w)$ by the remarks before Lemma~\ref{cons-lem}.
But then $b = e_i a\neq \zero$ so $b \in \rf^+(w,B)$ and $A=B$.
 \end{proof}
 
We can now obtain the result promised earlier.

 \begin{theorem}
 \label{up-thm}
 Let $w \in S_\ZZ$ and $A \subseteq \Inv(w)$. 
When nonempty, the sets
 $\RF^+(w,A)$ and  $\RF^+(w)$ are normal $\gl_n$-crystals for the given weight function and with $e_i$ and $f_i$ as in Definition~\ref{ef-def},
and the map $\unprime \colon \RF^+(w,A) \to \RF(w)$ is a $\gl_n$-crystal isomorphism.
\end{theorem}

  \begin{proof}
  As $\RF(w)$ is a normal $\gl_n$-crystal by Remark~\ref{rem:Lusztig_twist},
  this holds by
 Lemmas~\ref{lemma:prime_reduction} and~\ref{up-lem}.
\end{proof}

See Figure~\ref{gl3-fig} for an example of the normal $\gl_n$-crystal $\RF^+(w,A)$.
The character of this crystal is the \defn{Stanley symmetric polynomial}
$F_w(x_1,x_2,\dotsc,x_n)$ introduced in~\cite{Stanley}.

\begin{figure}
\begin{center}
\input{gl3-decr-crystal.tex}
\end{center}
\caption{The $\gl_3$-crystal graph of $\rf^+_3(w,A)$ for $w=21543\in S_5$ and $A = \{ (1,2),(3,4),(3,5)\}$.
}
\label{gl3-fig}
\end{figure}

\begin{remark}
\label{rem:negation}
Results in~\cite{MT2021,Marberg2019b} 
describe crystal operators on tuples $a=(a^1,a^2,\ldots)$, where each $a^i$ is a (primed) word that is strictly \emph{increasing} rather than decreasing.
In addition to the Lusztig involution, there is a second operation that translates these $\gl_n$-crystals to the ones here.

For any bijection (that is, a permutation) $w \colon \ZZ\to\ZZ$, let $w^\ast$ be the permutation mapping $i \mapsto 1 - w(1-i)$ for all $i \in \ZZ$.
For  $x \in \frac{1}{2}\ZZ$, let $x^\ast := -x - (\lceil x \rceil - \lfloor x \rfloor)$,
so that if $x \in \ZZ$ then $x^\ast = -x$ and $(x')^\ast= (-x)' = -(x') -1$.
Finally, for primed words $i=i_1i_2\cdots i_p$, let $i^\ast := (i_1)^\ast (i_2)^\ast \cdots (i_p)^\ast$
and for tuples of primed words $a=(a^1, a^2, \ldots)$, let 
$ a^\ast := ((a^1)^*, (a^2)^*, \ldots).$

The operation $\ast $ defines inverse bijections $\cR^+(w) \leftrightarrow \cR^+(w^\ast)$.
In turn, $\ast$ gives a bijection from the set of
 tuples $a=(a^1,a^2,\ldots)$ of increasing primed words with $a^1a^2\cdots\in \cR^+(w)$ to $\RF^+(w^\ast)$.
Further restricted to unprimed factorizations, the latter map is an isomorphism from the $\gl_n$-crystal denoted  $\cR_n(w)$ 
in~\cite[\S3.1]{Marberg2019b} 
to the $\gl_n$-crystal $\RF(w^\ast)$.
\end{remark}

We will need a description of the highest and lowest weight elements in  $\rf_n(w)$. 
Recall that the \defn{(Young) diagram} of a partition $\lambda$ is the set $\D_\lambda := \{ (i,j) \in \PP\times \PP \mid 1 \leq j \leq \lambda_i\}$.
A \defn{tableau} of partition shape $\lambda$ is a map $T \colon \D_\lambda \to \ZZ$, written $(i,j) \mapsto T_{ij}$.
The \defn{row reading word} of a tableau $T$ is the sequence $\row(T)$ formed by reading the rows left to right, but starting with the last row.
The \defn{reverse row reading word} of $T$ is the reverse of this sequence, denoted $\revrow(T)$.
The \defn{column reading word} of $T$ is the sequence $\col(T)$ formed by reading the columns from bottom to top, starting with the first row.
For example, $ T=\ytabsmall{ 1 & 2 & 3 \\ 2 & 3}$ has $ \row(T) = 23123$, $\revrow(T) = 32132$, and $\col(T) = 21323$.
A tableau is \defn{increasing}/\defn{decreasing}  if its rows and columns are strictly increasing/decreasing.

\begin{theorem}  \label{reduced-tab-lem}
Fix a Coxeter--Knuth equivalence class $\cK\subseteq \cR(w)$ for $w \in S_\ZZ$. 
There is a unique increasing tableau $U$ with $\revrow(U) \in \cK$ and a unique decreasing tableau $V$ with $\row(V) \in \cK$.
Additionally, both $U$ and $V$ have the same shape.
\end{theorem}

\begin{proof}
Choose $i=i_1i_2\cdots i_k \in \cK$. Let $i^\rev := i_k\cdots i_2i_1$, and recall that $i^\ast = (-i_1)(-i_2)\cdots(-i_k)$.
Form $T^\ast$  from a given tableau $T$ by applying $\ast$ to all of its entries.
The \defn{Edelman--Greene insertion algorithm} from~\cite{EG} gives an increasing tableau $P_\EG(i)$ with row reading word in $\cK$~\cite[Lem.~6.23]{EG} and with $P_\EG(i) = P_\EG(j)$ when $j \in \cR(w)$ if and only if $ j \in \cK$~\cite[Thm.~6.24]{EG}.
It is straightforward to deduce from properties of EG insertion (see~\cite[Lem.~6.22, Cor.~7.21, Cor.~7.22]{EG}) that the desired tableaux are $U = P_\EG(i^\rev) = P_\EG(i)^\T$ and $V = P_\EG(i^\ast)^\ast$, which have the same shape.
%
%
%
\end{proof}

We define a \defn{reduced tableau} for $w \in S_\ZZ$ to be a tableau $T$ that is increasing with $\revrow(T) \in \cR(w)$ or decreasing with $\row(T) \in \cR(w)$.
This terminology is unambiguous since $T$ can only be both increasing and decreasing when $\revrow(T) = \row(T)$. 
Given a reduced tableau $T$ for $w$, let $\rf_n(T) \subseteq \rf_n(w)$ be the subset of decreasing reduced factorizations $a=(a^1,a^2,\ldots)$ that have $a^1a^2\cdots \simCK \revrow(T)$ if $T$ is increasing and $a^1a^2\cdots \simCK \row(T)$ if $T$ is decreasing.

\begin{example} \label{ex:reduced_tableaux}
The increasing/decreasing reduced tableaux for $w=21543 \in S_5$ are 
\[
\ytabsmall{ 1 & 3 \\ 3 & 4}
\qquad \ytabsmall{1 & 3 & 4 \\ 4 }
\qquad \ytabsmall{ 1 & 4 \\ 3 \\ 4}
\qquand
 \ytabsmall{ 4 & 3 \\ 3 & 1}
\qquad \ytabsmall{ 4 & 3 & 1  \\ 3}
\qquad \ytabsmall{4 & 3 \\ 3 \\ 1}.
\]
It is useful to compare these with Figure~\ref{gl3-fig} (ignoring primes).
\end{example}

\begin{proposition}[{See~\cite[\S4]{MorseSchilling}}]
\label{gl-subcrystal-prop}
The map $T \mapsto \rf_n(T)$ is a  bijection from increasing (equivalently, decreasing) reduced tableaux for $w \in S_\ZZ$ with at most $n$ rows 
to full subcrystals of $\rf_n(w)$.
If $T$ is a reduced tableau for $w \in S_\ZZ$ then $\rf_n(T)$ is empty if and only if $T$ has more than $n$ rows.
\end{proposition}

\begin{proof}
It follows from~\cite[Thm.~4.11]{MorseSchilling} (via the properties of the \defn{Edelman--Greene correspondence} in~\cite[\S6]{EG} and 
Remark~\ref{rem:negation}; see~\cite[Thms. 3.27 and 3.28]{Marberg2019b}) that each subset $\rf_n(T)\subseteq \rf_n(w)$ is a full subcrystal or empty, and every full subcrystal arises as $\rf_n(T)$ for some reduced tableau $T$.

If $T$ is a reduced tableau with at most $n$ rows, then its row reading word gives rise to an element of the set $\rf_n(T)$; see Propositions~\ref{gl-highest-lem} and~\ref{gl-lowest-lem}.
It is implicit in the literature that $\rf_n(T)$ is empty if $T$ has more than $n$ rows, but we include another argument here for completeness.

Let $a \in \rf_n(T)$. 
Then the semistandard version of the {Edelman--Greene correspondence} in~\cite[Thm.~4.11]{MorseSchilling} 
maps $a^\ast$ to a pair of tableaux $(P,Q)$ of the same shape,
where $P = P_{\EG}\bigl( (a^1a^2\cdots a^n)^{\ast} \bigr)$ and where $Q$ is semistandard with all entries in $[n]$.
From the proof of Theorem~\ref{reduced-tab-lem}, we see that $P$ and $Q$ have the same shape as $T$.
As there are no semistandard tableaux $Q$ that have more than $n$ rows and all entries in $[n]$, $\rf_n(T)$ must be empty if $T$ has more than $n$ rows.
\end{proof}

Indexing the full subcrystals of $\rf_n(w)$ using increasing reduced tableaux makes it easy to write down all highest weight elements, while using decreasing tableaux gives the lowest weight elements.
For examples of the following these statements, consider the reduced tableaux in Example~\ref{ex:reduced_tableaux}.

\begin{proposition}\label{gl-highest-lem}
If $T$ is an increasing reduced tableau for $w \in S_\ZZ$ with at most $n$ rows, then
the unique highest weight element  of $\rf_n(T)$ is
$a=(a^1,\dotsc,a^n)$,
where $a^i$ is the reversal of row $i$ of $T$.
 \end{proposition}
 
\begin{proof}
Since $\rf_n(T)$ is a connected normal $\gl_n$-crystal, we know that it has a unique highest weight element.
We just need to check that the given factorization $a=(a^1,\dotsc,a^n)$
 has $e_i(a)=0$ for all $i \in [n-1]$.
This holds since every letter in $a^{i+1}$ is the second term in some $(b,c) \in \pair(a^i,a^{i+1})$.
\end{proof}

\begin{proposition}\label{gl-lowest-lem}
 If $T$ is a decreasing reduced tableau for $w \in S_\ZZ$ with at most $n$ rows, then 
the
  unique lowest weight element of $\rf_n(T)$ is $a=(a^1,a^2,\dotsc,a^n)$,
where $a^i$ is row $n + 1 - i$ of $T$.
\end{proposition}

\begin{proof}
Similar to proof of Proposition~\ref{gl-highest-lem}, it is enough to check that
every letter in $a^{i}$ as described is the first term in some $(b,c) \in \pair(a^i,a^{i+1})$, and this is a straightforward exercise.
\end{proof}

Recall the definition of the dominant permutation $w_\lambda \in S_\infty$ from Section~\ref{schub-sect}.

\begin{proposition}\label{dom-prop}
Suppose $\lambda$ is a partition and $T_\lambda$ is the tableau of shape $\lambda$
with entry $i+j-1$ in each position $(i,j) \in \D_\lambda$.
Then  $\cR(w_\lambda)$ is a single Coxeter--Knuth equivalence class
and $T_\lambda$ is the unique increasing reduced tableau for $w_\lambda$.
\end{proposition}

\begin{proof}
The claim that  $\cR(w_\lambda)$ is a single equivalence class
 can be shown using~\cite[Thm.~8.1]{EG}.
 As we have $\revrow(T_\lambda) \in \cR(w_\lambda)$ by~\cite[Rem.~2.1.9]{Manivel},
the result follows by Theorem~\ref{reduced-tab-lem}.
\end{proof}

For each partition $\lambda$,
define $\rf_n(\lambda) := \rf_n(w_\lambda)$.
The following well-known properties are immediate consequences of 
Proposition~\ref{gl-subcrystal-prop} given the preceding result and Proposition~\ref{gl-highest-lem}.

\begin{corollary}
Let $\lambda$ be a partition. The set $\RF(\lambda)$ is nonempty if and only if  $\lambda$ has at most $n$ parts, 
in which case it is a connected normal $\gl_n$-crystal with unique highest weight $\lambda$.
Consequently, each connected normal $\gl_n$-crystal
is isomorphic to $ \RF(\lambda)$ for a unique partition $\lambda \in \NN^n$.
\end{corollary}

Proposition~\ref{rf-nonempty-prop} characterizes when the set $\rf_n(w)$ is nonempty for general $w \in S_\ZZ$.
Determining the smallest $n$ such that $\rf_n(w)$ includes all highest weight elements of $\rf(w)$ (for $\gl_N$ with $N \gg 1$) is a more subtle question.
For example, $\rf_1(s_1 s_3)$ contains only one highest weight element but $\rf_n(s_1 s_3)$ has two highest weight elements for all $n \geq 2$.

\subsection{Demazure crystals}

Given a subset $\cX$ of a $\gl_n$-crystal $\cB$ and $i \in [n-1]$, define the \defn{crystal Demazure operator}
\be\label{cDo-eq}
\fkD^{\cB}_i \cX := \left\{ b \in \cB \mid e_i^k b \in \cX \text{ for some }k\in \NN\right\},
\ee
which we view as a subcrystal of $\cB$ by restricting the crystal structure.
It always holds that $\fkD^\cB_i\fkD^\cB_i = \fkD^\cB_i$.
When $\cB$ is clear, we will sometimes write $\fkD_i := \fkD_i^{\cB}$.

\begin{definition}\label{dem-def}
A \defn{Demazure $\gl_n$-crystal} is a $\gl_n$-crystal isomorphic to $\fkD^\cB_{i_1}\fkD^\cB_{i_2} \dotsm \fkD^\cB_{i_k} \{ u \}$ for a highest weight element $u$ in a 
normal $\gl_n$-crystal $\cB$ and any sequence of indices $i_1,i_2,\dotsc,i_{k} \in [n-1]$. 
\end{definition}

Like connected normal $\gl_n$-crystals, every Demazure $\gl_n$-crystal has a unique highest weight element $u$,
which is also the unique element at which $\epsilon_i(u) = 0$ for all $i \in I$.
Most Demazure crystals that we discuss are constructed as subsets of specific connected normal crystals.
It is useful to note that the latter objects are uniquely determined up to isomorphism:

\begin{proposition}\label{unique-embed-prop}
Suppose $\cB$ is a connected normal $\gl_n$-crystal with highest weight element $b$.
Let $\cX$ be a Demazure $\gl_n$-crystal with highest weight element $u$.
\ben
\item[(a)] There is a unique embedding $\cX \to \cB$ if $\weight(b) = \weight(u)$.
\item[(b)] There are no embeddings $\cX \to \cB$ if $\weight(b) \neq \weight(u)$.
\item[(c)] If $\cX\subseteq \cB$ then $u=b$ and $\cX = \fkD^\cB_{i_1} \fkD^\cB_{i_2}\cdots \fkD^\cB_{i_k} \{b\} $ for some $i_1,i_2,\dotsc,i_k \in [n-1]$.


\een
 \end{proposition}
 
We omit the proof of this proposition, which follows as a straightforward exercise from Definition~\ref{dem-def}.
This result is also implicit in~\cite{Kashiwara93}.

Assume $\cB$ is a normal $\gl_n$-crystal. Fix a partition $\lambda \in \NN^n$
and suppose $u_\lambda \in \cB$ is a highest weight element with $\weight(u_\lambda) = \lambda$.
The operators $\fkD^\cB_i$ satisfy the braid relations for $S_n$ when applied to $\cX = \{u_\lambda\}$ by results in~\cite{Kashiwara93} (see also~\cite[Thm.~13.5]{BumpSchilling}).
We can therefore define $\fkD^\cB_w \{u_\lambda\}$ for any $w \in S_n$.
By~\cite[Thm.~13.7]{BumpSchilling}, the character of this Demazure $\gl_n$-crystal is the key polynomial
\be\label{fkD-eq2}
\ch(\fkD^\cB_w \{u_\lambda\}) = \kappa_{w\circ \lambda} \in \NN[x_1,x_2,\dotsc,x_n]
\quad\text{for any }w \in S_n.
\ee
Every Demazure $\gl_n$-crystal is isomorphic to some $\fkD^\cB_w \{u_\lambda\}$, so the characters of Demazure $\gl_n$-crystals are the precisely the key polynomials $\kappa_\alpha$ indexed by $\alpha \in \NN^n$.
In fact, these polynomials are a $\ZZ$-basis for $\ZZ[x_1,x_2,\dotsc,x_n]$~\cite[Cor.~7]{ReinerShimozono} and two Demazure $\gl_n$-crystals are isomorphic if and only if they have the same character (see, e.g.,~\cite{Kashiwara93}).

\begin{example}
As $\alpha=(1,0,2,1) = w \circ \lambda$ for $w =  s_2 s_1 s_3$ and $\lambda = (2,1,1,0)$, we have 
\[
\kappa_{1021} = \pi_2\pi_1\pi_3 (\xx^{2110})
 = \xx^{2110} + \xx^{1210} + \xx^{1120} + \xx^{2101} + \xx^{2011} + \xx^{1201} + \xx^{1111} + \xx^{1021}.
\]
\ytableausetup{boxsize=1.2em}%
This is the character of the Demazure $\gl_4$-crystal $\fkD_2\fkD_1\fkD_3\left\{ \ytableaushort{3} \otimes \ytableaushort{2} \otimes \ytableaushort{1} \otimes \ytableaushort{1}\right\}\subset  \BB_4^{\otimes 4}$ shown as
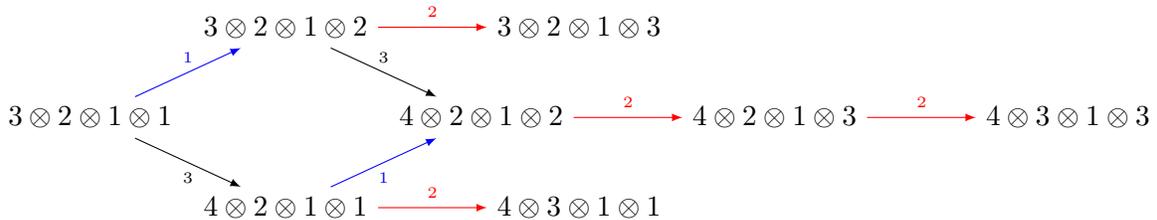
\[
\begin{tikzpicture}[xscale=1.3,yscale=1.2,>=latex]
\node (hw) at (0,0) {$3 \otimes 2 \otimes 1 \otimes 1$};
\node (1) at (2, 1) {$3 \otimes 2 \otimes 1 \otimes 2$};
\node (3) at (2, -1) {$4 \otimes 2 \otimes 1 \otimes 1$};
\node (13) at (4, 0) {$4 \otimes 2 \otimes 1 \otimes 2$};
\node (213) at (7, 0) {$4 \otimes 2 \otimes 1 \otimes 3$};
\node (2213) at (10, 0) {$4 \otimes 3 \otimes 1 \otimes 3$};
\node (21) at (5, 1) {$3 \otimes 2 \otimes 1 \otimes 3$};
\node (23) at (5, -1) {$4 \otimes 3 \otimes 1 \otimes 1$};
\draw[->,color=blue] (hw) -- node[midway,above] {\tiny $1$} (1);
\draw[->,color=black] (hw) -- node[midway,below] {\tiny $3$} (3);
\draw[->,color=blue] (3) -- node[midway,below] {\tiny $1$} (13);
\draw[->,color=black] (1) -- node[midway,above] {\tiny $3$} (13);
\draw[->,color=red] (1) -- node[midway,above] {\tiny $2$} (21);
\draw[->,color=red] (3) -- node[midway,above] {\tiny $2$} (23);
\draw[->,color=red] (13) -- node[midway,above] {\tiny $2$} (213);
\draw[->,color=red] (213) -- node[midway,above] {\tiny $2$} (2213);
\end{tikzpicture}
\]
where we have written $a \otimes b \otimes c \otimes d$ for $\ytableaushort{a} \otimes \ytableaushort{b} \otimes \ytableaushort{c} \otimes \ytableaushort{d}$.
\end{example}

Later, we will use the following trivial extension of~\eqref{fkD-eq2}.

\begin{lemma}\label{sp-ij-lem}
Suppose $\cX$ is a direct sum of Demazure $\gl_n$-crystals that is a subcrystal of a normal $\gl_n$-crystal $\cB$.
Then the operators $\fkD^{\cB}_i$ for $i \in [n-1]$ satisfy the braid relations on $\cX$,
and each set $\fkD^{\cB}_i \cX$ is itself a direct sum of Demazure $\gl_n$-crystals with character $\ch(\fkD^{\cB}_i \cX) = \pi_i \ch(\cX)$.
\end{lemma}

Now we examine certain $\gl_n$-crystals of unprimed factorizations that decompose as direct sums of Demazure crystals. These crystals are
closely related to the geometry of the complete flag variety.
For the rest of this section, we fix an element $w \in S_\infty$ rather than in $S_\ZZ$.

A \defn{flag}
$\phi = (\phi_1 \leq \phi_2 \leq \cdots)$ is a weakly increasing sequence with $i \leq \phi_i \in \PP$ for all $i $.
A sequence of words $a=(a^1,a^2,\ldots)$ is \defn{bounded} by $\phi$ if  any of the following equivalent conditions holds:
\begin{enumerate}
\item[(1)] for each $i \in \PP$, every letter $m$ appearing in $a^i$ is a positive integer with $i \leq \phi_m$;
\item[(2)] each $m \in \ZZ$  appears in $a^i$ only if $1\leq i \leq \phi_m$ where we set $\phi_m := 0$ for $m \leq 0$;
\item[(3)]  the ``conjugate'' sequence $\psi_i := \min\{ m \in \PP \mid i \leq \phi_m \}$ has $\min a^i \geq \psi_i$ whenever $a^i \neq \emptyset$.
\end{enumerate}
Given a permutation $w \in S_\infty$ and a flag $\phi$, let  
\[
\brf(w, \phi) \subseteq \rf(w)\quand \brf_n(w,\phi) \subseteq \rf_n(w)
\]
be the subsets of reduced factorizations that are bounded by $\phi$. 
We view $\brf_n(w, \phi)$ as a $\gl_n$-subcrystal of $\rf_n(w)$.
The \defn{standard flag} $\phi^S$ has $\phi_i^S = i$ for all $i$, and we let
\[ \brf(w) := \brf(w, \phi^S)\quand \brf_n(w) := \brf_n(w, \phi^S).\]
We refer to elements of these sets as \defn{bounded reduced factorizations} of $w$.
If  $\phi$ and $\psi$ are flags with $\phi_i \leq \psi_i$ for all $i$ then $\brf(w,\phi)\subseteq \brf(w,\psi)$,
 and so $\brf(w) \subseteq \brf(w,\phi)$ always holds.
 
 \begin{example}\label{brf-ex}
The crystal graph of $\brf_3(w)$ for $w=21543\in S_5$ (compare with Figure~\ref{gl3-fig}) is 
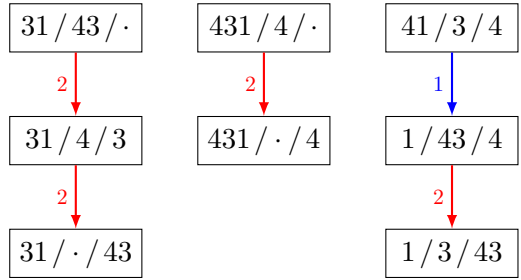
\[
\begin{tikzpicture}[xscale=2.5,yscale=1.5,>=latex]
\node [bnode] (a) at (0, 0) {$31 \gap 43 \gap \cdot$};
\node [bnode] (b) at (0, -1) {$31 \gap 4 \gap 3$};
\node [bnode] (c) at (0, -2) {$31 \gap \cdot \gap 43$};
\draw[->,thick,color=red] (a) -- node[midway,left,scale=0.75] {  $2$} (b);
\draw[->,thick,color=red] (b) -- node[midway,left,scale=0.75] {  $2$} (c);
\node [bnode] (d) at (1, 0) {$431 \gap 4 \gap \cdot$};
\node [bnode] (e) at (1, -1) {$431 \gap \cdot \gap 4$};
\draw[->,thick,color=red] (d) -- node[midway,left,scale=0.75] {  $2$} (e);
\node [bnode] (f) at (2, 0) {$41 \gap 3 \gap 4$};
\node [bnode] (g) at (2, -1) {$1 \gap 43 \gap 4$};
\node [bnode] (h) at (2, -2) {$1 \gap 3 \gap 43$};
\draw[->,thick,color=blue] (f) -- node[midway,left,scale=0.75] {  $1$} (g);
\draw[->,thick,color=red] (g) -- node[midway,left,scale=0.75] {  $2$} (h);
\end{tikzpicture}
\]
\end{example}

These crystals are of interest as their characters are Schubert polynomials.
 It is known that $\fkS_w \in \ZZ[x_1,x_2,\dotsc,x_n]$ if and only if $w(i) < w(i+1)$ for all $i >n$~\cite[\S5.2]{Manivel}.
The \defn{Billey--Jockusch--Stanley formula}~\cite[Thm.~1.1]{BJS} is equivalent 
to the identity $\fkS_w  = \sum_{a \in \brf(w)} \xx^{\weight(a)}$;
see~\cite[Prop.~5.5]{AssafSchilling} for a proof this equivalence.
 It follows that if $w \in S_\infty $ has no descents greater than $n$ then  
 \be\label{ch-brf-eq}
 \brf_n(w) = \brf(w)\quand 
 \ch\bigl( \brf_n(w) \bigr) = \fkS_w.
 \ee
 However, if $w $ has at least one descent greater than $n$, then $ \brf_n(w) $ is a proper subset of $ \brf(w)$
 and $\ch\bigl( \brf_n(w) \bigr)$ is the polynomial obtained from $\fkS_w$ by setting $x_{n+1}=x_{n+2}=\cdots=0$.

\begin{example}
The Schubert polynomial of $w=21543 \in S_5$ is
 \[\ba
 \fkS_{21543} &=
 \xx^{1021} + 
\xx^{1111} + 
\xx^{1120} + 
\xx^{1201} + 
\xx^{1210} + 
\xx^{2011} 
\\&
\quad+ 
\xx^{2020} + 
\xx^{2101} + 
2\xx^{2110} +  
\xx^{2200} + 
\xx^{3001} +
\xx^{3010} +
\xx^{3100} = \ch\bigl( \brf_4(21543) \bigr)
\ea\] 
while $\ch\bigl( \brf_3(21543) \bigr) = 
\xx^{112} + 
\xx^{121} + 
\xx^{202} + 
2\xx^{211} +  
\xx^{220} + 
\xx^{301} +
\xx^{310}$. 
\end{example}

Suppose $\lambda  \in \NN^n$ is a partition and $w_\lambda \in S_\infty$ is the dominant element of shape $\lambda$.
 Then $\brf_n(w_\lambda)$ contains only the unique highest weight element of $\rf_n(w_\lambda)=\rf_n(\lambda)$ by~\cite[Thm.~2.5.1 and Prop.~2.6.7]{Manivel}.
If $\alpha \in \NN^n$ is a weak composition with $\lambda=\lambda(\alpha)$, then $u(\alpha) \in S_n$ so we can define 
\be
\brf_n(\alpha) := \fkD_{i_1} \fkD_{i_2}\cdots \fkD_{i_k}\brf_n(w_\lambda),
\ee
where $\fkD_i := \fkD_i^{\RF(\lambda)}$ and $i_1i_2\cdots i_k$ is any reduced word for $u(\alpha)$.
It follows from~\eqref{fkD-eq2} that $\brf_n(\alpha)$ is a Demazure $\gl_n$-crystal with character $\kappa_{\alpha}$.
Therefore the following holds by the observations after Proposition~\ref{unique-embed-prop}:

\begin{proposition}\label{unique-embed-prop(d)}
If $\cX$ is a Demazure $\gl_n$-crystal with highest weight element $u$,
then there is a unique $\alpha \in \NN^n$ with $  \brf_n(\alpha) \cong \cX$, and for this $\alpha$ it holds that
 $\key_\alpha=\ch(\cX)$ and $\lambda(\alpha) = \weight(u)$.
\end{proposition}

\begin{example}\label{brf-alpha-ex}
We have $w_{(2,2)} = s_2s_1s_3s_2$,  
$w_{(3,1)} = s_3s_2s_1s_2$, and 
$w_{(2,1,1)} = s_2s_1s_2s_3$. 
The crystal graphs of $\brf_3(\alpha)$ for  $\alpha=(2,0,2)$, $(3,0,1)$, $(1,1,2)$ are respectively
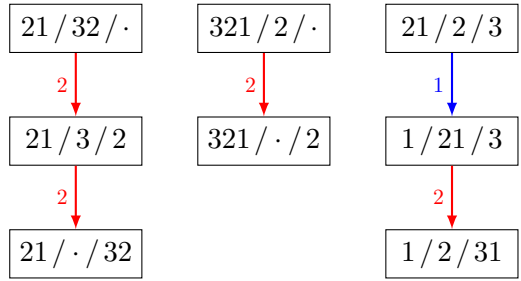
\[
\begin{tikzpicture}[xscale=2.5,yscale=1.5,>=latex]
\node [bnode] (a) at (0, 0) {$21 \gap 32 \gap \cdot$};
\node [bnode] (b) at (0, -1) {$21 \gap 3 \gap 2$};
\node [bnode] (c) at (0, -2) {$21 \gap \cdot \gap 32$};
\draw[->,thick,color=red] (a) -- node[midway,left,scale=0.75] {  $2$} (b);
\draw[->,thick,color=red] (b) -- node[midway,left,scale=0.75] {  $2$} (c);
\node [bnode] (d) at (1, 0) {$321 \gap 2 \gap \cdot$};
\node [bnode] (e) at (1, -1) {$321 \gap \cdot \gap 2$};
\draw[->,thick,color=red] (d) -- node[midway,left,scale=0.75] {  $2$} (e);
\node [bnode] (f) at (2, 0) {$21 \gap 2 \gap 3$};
\node [bnode] (g) at (2, -1) {$1 \gap 21 \gap 3$};
\node [bnode] (h) at (2, -2) {$1 \gap 2 \gap 31$};
\draw[->,thick,color=blue] (f) -- node[midway,left,scale=0.75] {  $1$} (g);
\draw[->,thick,color=red] (g) -- node[midway,left,scale=0.75] {  $2$} (h);
\end{tikzpicture}
\]
Observe that $\brf_3(21543) \iso \brf_3((2,0,2)) \oplus\brf_3((3,0,1))\oplus \brf_3((1,1,2))$.
\end{example}


Boundedness excludes some elements from $\RF(w)$, but never the highest weight elements.

\begin{lemma}\label{each-highest-lem}
If $w \in S_\infty$ then each highest weight element in $\rf_n(w)$ is in $\brf_n(w)$.
\end{lemma}

\begin{proof}
The highest weight factorizations in Proposition~\ref{gl-highest-lem} are bounded since the first column of an increasing reduced tableau for $w \in S_\infty$ is an increasing sequence of positive numbers.
 \end{proof}
 
 It is also useful to note the following property:

\begin{lemma}\label{brf-e-lem}
Suppose  $a \in \brf_n(w, \phi)$ and $i \in [n-1]$.
Then $e_ia \in \brf_n(w, \phi) \sqcup \{ \zero \}$.
\end{lemma}

\begin{proof}
If $e_i a \neq \zero$  then $e_{ i}a$ is formed from $a=(a^1,a^2,\ldots)$ by removing a letter $x$ from $a^{i+1}$ and adding a letter $y \geq x$ to $a^i$.
This gives another element of $ \brf_n(w,\phi)$ since $\phi_y \geq \phi_x \geq i+1$.
\end{proof}

We abbreviate by writing $\brf_n(T,\phi) := \rf_n(T) \cap \brf_n(w,\phi)$
and
$\brf_n(T) := \brf_n(T,\phi^S)$.

\begin{corollary}\label{dem-sub-cor}
The map $T \mapsto \brf_n(T,\phi)$ is a  bijection from increasing (equivalently, decreasing) reduced tableaux for $w \in S_\infty$ with at most $n$ rows 
to full subcrystals of $\brf_n(w,\phi)$.
\end{corollary}

\begin{proof}
In view of Proposition~\ref{gl-subcrystal-prop}, 
it suffices to show that   $\brf_n(T,\phi) $ is nonempty and connected for every reduced tableau $T$ with at most $n$ rows.
This holds by Lemmas~\ref{each-highest-lem} and~\ref{brf-e-lem}.
\end{proof}

Generalizing Examples~\ref{brf-ex} and~\ref{brf-alpha-ex},
 Assaf and Schilling~\cite{AssafSchilling} have shown that if $w \in S_\infty$ then $\brf_n(w)$  is always a direct sum of Demazure $\gl_n$-crystals.\footnote{%
  This follows from~\cite[Thm.~5.11]{AssafSchilling} after twisting by the Lusztig involution as in Remark~\ref{rem:Lusztig_twist}.
  One should note, however, that the proof of~\cite[Thm.~5.11]{AssafSchilling} relies on results from a preprint of Assaf that has been superseded by~\cite{Assaf2021}, which updates some of the definitions in~\cite[\S5]{AssafSchilling}.
  In particular, the definition of the \defn{lift} operation just before~\cite[Def.~5.6]{AssafSchilling} should be changed to follow~\cite[Def.~4.22]{Assaf2021}, but this does not affect the proof of~\cite[Thm.~5.11]{AssafSchilling}.} 
Taking characters 
recovers the result of \cite{LS2}  that each $\fkS_w$ is an $\NN$-linear combination of key polynomials.
These properties extend to the $\phi$-bounded setting via the following theorem.

\begin{theorem}[{\cite[Thm.~4.5]{Wen2023}}] \label{flag-thm}
Let $w \in S_\infty$.
Suppose $\phi$ is a  flag and $i $ is minimal with $i < \phi_i$.
\ben
\item[(a)] If $i \geq n$ or $\phi_i >n$ then $\BRF(w,\phi) = \BRF(w,\phi- \e_i)$. 
\item[(b)] Otherwise $\BRF(w,\phi) = \fkD_j^{\RF(w)} \BRF(w,\phi- \e_i)$ for $j=\phi_i-1$.
\een
\end{theorem}

As $\BRF(w)$ is a direct sum of Demazure crystals, it follows from the preceding theorem and Lemma~\ref{sp-ij-lem} 
that $\brf_n(w,\phi)$ is a direct sum of Demazure $\gl_n$-crystals for any flag $\phi$.
Thus, each full subcrystal of $\brf_n(w,\phi)$ is  a Demazure crystal, and Corollary~\ref{dem-sub-cor} implies the following.

\begin{corollary}\label{is-dem-cor}
If $\phi$ is a flag and $T$ is a reduced tableau for $w \in S_\infty$ with at most $n$ rows,
then $\BRF(T,\phi)$ is a Demazure $\gl_n$-crystal and $\BRF(w,\phi)$ is a direct sum of Demazure $\gl_n$-crystals.
\end{corollary}

Suppose $\phi$ is a flag and $T$ is a reduced tableau for some $w \in S_\infty$.
If $T$ has more than $n$ rows then the set $\BRF(T,\phi)$ is empty and $\ch\bigl( \BRF(T,\phi) \bigr) = 0$.
If $T$ has at most $n$ rows, however, then the preceding corollary implies that 
there is a unique weak composition $\alpha_n(T,\phi) \in \NN^n$ with
\be
\BRF(T,\phi) \iso \BRF\bigl( \alpha_n(T,\phi) \bigr) \quand
\ch\bigl( \BRF(T,\phi) \bigr) = \kappa_{\alpha_n(T,\phi)}.
\ee
As $n$ increases, the sequence of weak compositions $\alpha_n(T,\phi)$ is eventually constant
since the set $\BRF(T,\phi)$ is equal to $\brf(T,\phi)$ if $n$ is sufficiently large.
Define \be\alpha(T) := \lim_{n\to\infty} \alpha_n(T,\phi^S)\quad\text{where }\phi^S = (1<2<3<\dots).\ee
 An algorithm is known for computing $\alpha(T)$: if $T$ is an increasing reduced tableau for an element of $S_\infty$  then   $\alpha(T)$ is the content of the \defn{left nil-key} of $T$, as defined in~\cite[Thm.~5(1)]{ReinerShimozono}.

We now explain how to express $\alpha_n(T,\phi)$ for any $n$ and $\phi$ in terms of $\alpha(T)$.
For positive integers $a\leq b$ define $s_{b\searrow a} := s_{b-1}s_{b-2}\cdots s_{a+1} s_{a}$
so that $s_{a\searrow a} = 1$.
Then for each flag $\phi$ let 
\be\label{Delta-eq}
\Delta_n(\phi) := s_{\max\{n,\phi_1\}\searrow 1} \circ s_{\max\{n,\phi_2\}\searrow 2} \circ  \cdots \circ s_{\max\{n,\phi_n\}\searrow n}  \in S_n.
\ee
There is a monoid homomorphism $(S_\infty,\circ) \to (S_n,\circ)$ that sends $s_i \mapsto s_i$ for all $i \in [n-1]$ and $s_i\mapsto 1$ for all $i \notin[n-1]$.
For any weak composition $\alpha$, let $u_n(\alpha) \in S_n$ be the image under this homomorphism of 
the shortest permutation $u(\alpha) \in S_\infty$ with $\alpha = u(\alpha)\circ \lambda(\alpha)$.
This permutation can be computed by taking any expression for $u(\alpha) = s_{i_1} \circ s_{i_2}\circ \cdots \circ s_{i_k}$
as a Demazure product of simple transpositions and then omitting all factors $s_{i_j}$ with $i_j \notin [n-1]$.

\begin{lemma}\label{fkR-lem}
Let $\cX \subseteq \cB$ be a subset of a $\gl_n$-crystal where $ 0\leq k \leq n$, and define
\[
\fkR_k \cX := \{ b \in \cX \mid  \weight(b)_{i} = 0\text{ for all $k<i\leq n$}\}.
\]
Then for any index $i \in [n-1]$ it holds that
$
\fkR_k \fkD^\cB_i \cX = \begin{cases}   
\fkD^\cB_i \fkR_k\cX & \text{if }i<k
\\
\fkR_k \cX &\text{if } i \geq k
.\end{cases}
$
\end{lemma}

\begin{proof}
This follows as $\weight(\cX) \subset \NN^n$ and $\weight(e_ib) = \weight(b) + \e_i - \e_{i+1}$ for any $b \in \cB$ with $e_ib\neq \zero$.
\end{proof}

\begin{proposition}\label{alpha-prop}
 Let $T$ be a reduced tableau of shape $\lambda \in \NN^n$ for an element of $ S_\infty$. Then
 \[
 \alpha_n(T,\phi) = \Delta_n(\phi) \circ u_n\bigl( \alpha(T) \bigr) \circ \lambda.
 \]
\end{proposition}

\begin{proof}
In view of Theorem~\ref{reduced-tab-lem} and Proposition~\ref{gl-highest-lem}, Proposition~\ref{unique-embed-prop(d)} implies that  $\lambda(\alpha(T)) = \lambda$.
Choose a reduced word $i_1i_2\cdots i_k$ for $u(\alpha(T)) \in S_\infty$.
If $N\gg n$ is sufficiently large, then there is an isomorphism $\brf_N(T) \iso \brf_N(\alpha(T))$ of $\gl_N$-crystals,
which means that we also have
\[
\brf_n(T) = \fkR_n \brf_N(T) \iso \fkR_n \brf_N(\alpha(T)) =  \fkR_n\fkD_{i_1}\fkD_{i_2}\cdots \fkD_{i_k} \brf_N(\lambda),
\]
where the isomorphism is as $\gl_n$-crystals and where $\fkD_i := \fkD_i^{\rf_N(\lambda)}$.
It follows from Lemma~\ref{fkR-lem} that the  last expression is exactly $\brf_n(u_n(\alpha(T)) \circ \lambda)$,
so $\alpha_n(T,\phi^S)=u_n(\alpha(T)) \circ \lambda$.
To extend this identity from the standard flag to arbitrary flags, inductively apply Theorem~\ref{flag-thm}.
\end{proof}

 \begin{example}
If $T=\ytabsmall{ 1 & 3 \\ 3 & 4}$ then
 $
 \alpha_2(T) = (2,2) \neq \alpha_3(T) =(2,0,2)= \alpha(T)
 $
  and for the flag $\phi=(2\leq 2 < 4 \leq 4)$ we have 
$
   \alpha_2(T,\phi) = (2,2)
   \neq 
  \alpha_3(T,\phi) = (0,2,2)
  \neq 
    \alpha_4(T,\phi) = (0,2,0,2) .
    $
 \end{example}

\section{Demazure crystals for \texorpdfstring{$P$}{P}-key polynomials}
\label{sec:P_crystals}

In this section, we discuss analogues of the crystals $\RF(w)$
that are associated to fixed-point-free involutions in the symmetric group. These constructions will
lead to a ``symplectic'' version of Demazure crystals, whose characters recover the family of $P$-key polynomials.
Our main results here consist of Theorem~\ref{sp-crystal-thm} and the associated (equivalent) Conjectures~\ref{sp-demazure-conj} and~\ref{sp-demazure-conj2}.

\subsection{Queer crystals}\label{q-sect}

The crystals relevant to this section are for the \defn{queer Lie superalgebra} $\q_n$ studied in~\cite{GJKKK15, GJKKK14, GJKKK10} rather than $\gl_n$.
Such $\q_n$-crystals are defined via Section~\ref{crystal-sect} by the following concrete data.

The index set for this type is the disjoint union of three sets  $I \sqcup \overline{I} \sqcup \underline{I}$, where $I := [n-1]$ as for $\gl_n$-crystals, $\overline{I} := \{\overline{1}, \overline{2} \dotsc, \overline{n-1}\}$, and $\underline{I} := \{\underline{1}, \underline{2},\dotsc, \underline{n-1}\}$.
\ytableausetup{boxsize=1.2em}%
The \defn{standard $\q_n$-crystal} is the set $\BB= \left\{ \ytableaushort{1},\ytableaushort{2},\dotsc,\ytableaushort{n}\right\}$ with weight function $\weight(\ytableaushort{i}):=\e_i$, crystal graph
\[
    \begin{tikzpicture}[xscale=2.4, yscale=1,>=latex,baseline=0]
      \node at (0,0) (T0) {$\ytableaushort{1}$};
      \node at (1,0) (T1) {$\ytableaushort{2}$};
      \node at (2,0) (T2) {$\ytableaushort{3}$};
      \node at (3,0) (T3) {${\cdots}$};
      \node at (4.6,0) (T4) {$\ytableaushort{n}$};
      \draw[->,thick]  (T0) -- (T1) node[midway,above,scale=0.75] {$1$, $\overline 1$, $\underline 1$};
      \draw[->,thick]  (T1) -- (T2) node[midway,above,scale=0.75] {$2$, $\overline 2$, $\underline 2$};
      \draw[->,thick]  (T2) -- (T3) node[midway,above,scale=0.75] {$3$, $\overline 3$, $\underline 3$};
      \draw[->,thick]  (T3) -- (T4) node[midway,above,scale=0.75] {${n-1}$, $ \overline{n-1}$, $\underline {n-1}$};
     \end{tikzpicture}
\]
and statistics $\varepsilon_i, \varphi_i$ defined by~\eqref{string-eqs}.
This construction is identical to the standard $\gl_n$-crystal but with additional crystal operators indexed by $\overline{I} \sqcup \underline{I}$.

If $ i \in I$ then the operators $e_i$, $e_{\overline{\imath}}$, $e_{\underline{i}}$ (respectively, $f_i$, $f_{\overline{\imath}}$, $f_{\underline{i}}$)
all act in the same way on the elements of $\BB$.
This property does not hold for the $\q_n$-crystal structure on tensor powers $\BB^{\otimes m}$ for $m\geq 2$, which is given as follows.
We just need to specify the crystal operators on tensors. 
For all $i \in I$, the definitions of $e_i$ and $f_i$ on $\BB^{\otimes m}$ are identical to the $\gl_n$-case
and given inductively by~\eqref{eq:gl_tensor_product}. In particular, the category of normal $\q_n$-crystals has a natural forgetful functor to the category of normal $\gl_n$-crystals.
The operators $e_{\overline{\imath}}$ are defined by the inductive formulas~\cite[Lemma~2.2]{GHPS}
\be\label{q-e-eq}
  e_{\overline 1}(b\otimes c) := \begin{cases} 
 b \otimes (e_{\overline{1}} c) & \text{if }\weight(b)_1 = \weight(b)_2 = 0,
 \\
  (e_{\overline{1}}b) \otimes c
&\text{otherwise,}
 \end{cases}
 \quand 
 e_{\overline{\imath}} := \sigma_{i-1} \sigma_i e_{\overline{\imath-1}} \sigma_i \sigma_{i-1}
\ee
for $b \in \BB$, $c \in \BB^{\otimes(m-1)}$, and $i \in \{2,3,\dotsc,n-1\}$.
The $\sigma_i$ operators here are the same as in~\eqref{sigma-def}, and give a well-defined $S_n$-action on tensor powers of $\BB$.
The operators $f_{\overline{\imath}}$ are defined by the repeating the formulas in~\eqref{q-e-eq} with every ``$e$'' replaced by ``$f$''. 
Lastly, for all $i \in I$ we define
\be
e_{\underline{i}} := \sigma_{w_0} f_{\overline{n-\imath}} \sigma_{w_0}
\quand
f_{\underline{i}} := \sigma_{w_0} e_{\overline{n-\imath}} \sigma_{w_0}.
\ee
Observe from these definitions that the entire set of crystal operators on $\BB^{\otimes m}$
is completely determined by just the operators $e_i$ and $f_i$ indexed by $i \in \{\overline{1}, 1,2,\dotsc,n-1\}$.

As explained in Section~\ref{crystal-sect}, this data gives rise to categories of \defn{$\q_n$-crystals} and \defn{normal $\q_n$-crystals} 
that are closed under tensor products~\cite[Thm.~1.8]{GJKKK14}. 
Just like $\BB^{\otimes m}$, the crystal operators on any normal $\q_n$-crystal are completely determined by just the operators $e_i$ and $f_i$ indexed by $i \in \{\overline{1}, 1,2,\dotsc,n-1\}$. When drawing  the graphs of normal $\q_n$-crystals, we often include only the arrows of these indices (see, for example, Figure~\ref{fig:Sp_crystal}).

A partition is \defn{strict} if its nonzero parts are all distinct.
Analogous to $\gl_n$-crystals, each connected normal $\q_n$-crystal has a unique highest weight element whose weight $\lambda$ uniquely determines its isomorphism class; these weights $\lambda$ range over all strict partitions in $\NN^n$, that is, with at most $n$ parts~\cite{GJKKK15} (see also~\cite[Thm.~1.14]{GJKKK14}). 
For each strict partition $\lambda \in \NN^n$, we can choose a connected normal $\q_n$-crystal  $\cB(\lambda)$ with highest weight element $u_{\lambda}$ of weight $\lambda$. Then $\cB(\lambda)$ also has a unique lowest weight element of weight $w_0 \lambda = (\lambda_n,\dotsc,\lambda_2,\lambda_1)$, given by $\sigma_{w_0} u_{\lambda}$.
The crystal $\cB(\lambda)$ can be identified with the crystal basis of a polynomial $U_q(\q_n)$-module
and its character is the \defn{Schur $P$-polynomial}
$
P_{\lambda}(x_1, \dotsc, x_n) 
$~\cite{GJKKK15,GJKKK14,GJKKK10}.

\subsection{Symplectic reduced factorizations}

Recall the definitions of $\Ifpf_\infty\subset \Ifpf_\ZZ$ from Section~\ref{schub-sect}.
For each $z \in \Ifpf_\ZZ$
 define 
\be\label{cRsp-eq}
\cR^\Sp(z) := \bigsqcup_{w \in \cA^\Sp(z)} \cR(w)
\ee
where $\cA^\Sp(z)$ is the set of minimal-length permutations $w \in S_\ZZ$ with $z = w^{-1} 1_\fpf w$ as in Definition~\ref{fschub-def}.
Following~\cite{HMP5,HMP1,Hiroshima}, we refer to
 elements of  $\cR^\Sp(z)$ as \defn{fpf-involution words}.

Whereas reduced words for permutations may be identified with maximal chains in the weak
order on the symmetric group, fpf-involution words correspond to maximal
chains in an analogous weak order on the closures of the orbits of the symplectic group 
acting on the complete flag variety; see~\cite{Brion2001,RichSpring}. This accounts for the $\Sp$ superscript in some of our notation.

\begin{remark}\label{fpf-rmk}
The easiest way  to construct  the set $\cR^\Sp(z) $ is as follows.
Consider a word $i=i_1i_2\cdots i_\ell$ with each $i_j \in \ZZ$.
Let $z_1 := 1_\fpf $ and for $j \in [\ell]$ let $z_{j+1} := s_{i_j}   z_j   s_{i_j}$. Then $i \in \cR^\Sp(z)$ if and only if $i_j$ is an ascent of $z_j$ for each $j \in [\ell]$ by~\cite[Lem.~3.12]{Marberg2019b}.
It follows when $i \in \cR^\Sp(z)$ that 
\begin{itemize}
\item[(a)] $i_1$ must be even and $i_1\neq i_3$;
\item[(b)] if $i_2$ is odd then  $i_2 = i_1 \pm 1$ and $i_1(i_1 \mp 1) i_3\cdots i_\ell  $ is also in $ \cR^\Sp(y)$;
\item[(c)] at most one of $i_2$ or $i_3$ can belong to $\{i_1-1,i_1+1\}$.
\end{itemize}
The set $\cR^\Sp(z) $ is automatically preserved by the usual braid relations for the symmetric group.
The set is spanned by these relations plus the single additional relation that interchanges words for the form
$i_1(i_1 + 1) i_3\cdots i_\ell \leftrightarrow i_1(i_1 - 1) i_3\cdots i_\ell$ \cite[Thm.~6.22]{HMP2}.
\end{remark}

For $m \in \NN$ 
let $\Ifpf_m$ be the subset of elements $z \in \Ifpf_\infty$ with $z([m]) = [m]$ and $z(i) = 1_\fpf(i)$ for all $i > m$.
This set is empty if $m$ is odd, and is in bijection with the fixed-point-free involutions in $S_m$ when $m$ is even.
Each fixed-point-free involution $z\in S_{m}$ uniquely extends to an element of $\Ifpf_{m}$
mapping $2i \mapsto 2i-1$ for all $i\notin[m]$, and in examples we implicitly identity $z$ with this extension.

\begin{example}
Let $z=(1\: 5)(2 \: 4)(3 \: 6) \in \Ifpf_6$.
Then $i=2143 \in \cR^\Sp(z)$ since
\[
z_1 =1_\fpf, \quad
z_2 = (1\: 3)(2 \: 4)(5\: 6),\quad
z_3 = (1\: 4)(2 \: 3)(5\: 6),\quad
z_4 = (1 \: 5)(2\: 3)(4\: 6),\quad
z_5 = z.
\]
More generally, we have
 $\cR^\Sp(z) = \{2143, 2343, 2413, 2431, 2434, 4213, 4231, 4234\}.$
\end{example}

Let $\SpRF(z)$ be the set of  all tuples of decreasing words $a=(a^1,a^2,\ldots)$
with $a^i $ empty for all $i>n$ and with $a^1a^2\cdots \in \cR^\Sp(z)$.
Define $\weight(a)$ for $a \in \SpRF(z)$ by~\eqref{weight-eq}.
We may also write
\be
\SpRF(z) :=\bigsqcup_{w \in \cA^\Sp(z)} \RF(w).
\ee
We refer to the elements of this union as \defn{symplectic reduced factorizations}.

Like $\RF(w)$, 
the set $\SpRF(z)$ can be empty if $n$ is too small; see Theorem~\ref{sprf-thm} below.
When nonempty,  $\SpRF(z)$ is a disjoint union of normal $\gl_n$-crystals so is itself
a normal $\gl_n$-crystal. 
Results in~\cite{Marberg2019b} (see also~\cite{Hiroshima}) extend this structure to a $\q_n$-crystal, as we now explain.

\begin{definition}
\label{def:sp_1bar}
For $a=(a^1,a^2,\ldots) \in \SpRF(z)$, define $e_{\overline{1}}a$ and $f_{\overline{1}}a$ as follows:
\begin{enumerate}
\item[\defn{$e_{\overline{1}}$}:]
Set $e_{\overline{1}}a := \zero$ if $a^2$ is empty or $\max(a^2) \leq \max(a^1)$.

Otherwise, let $y:=\max(a^2)$ and form $e_{\overline{1}}a$ from $a$ by doing one of the following:
\begin{itemize}
\item if $y$ is even then remove $y$ from $a^2$ and add $y$ to the start of $a^1$;
\item if $y$ is odd then remove $y$ from $a^2$ and add $y-2$ directly after the first letter of $a^1$.
\end{itemize}

\item[\defn{$f_{\overline{1}}$}:]
Set $f_{\overline{1}}a := \zero$ if $a^1$ is empty or $\max(a^1) \leq \max(a^2)$.

Otherwise, let $x:=\max(a^1)$ and form $f_{\overline{1}}a$ from $a$ by doing one of the following:
\begin{itemize}
\item if $x-1 \in a^1$ then remove $x-1$ from $a^1$ and add $x+1$ to the start of $a^2$;
\item if $x-1 \notin a^1$ then remove $x$ from $a^1$ and add $x$ to the start of $a^2$.
\end{itemize}
\end{enumerate}
\end{definition}

\begin{remark}
Let $a=(a^1,a^2,\ldots) \in \SpRF(z)$.
The observations in Remark~\ref{fpf-rmk} imply that: 
\begin{enumerate}
\item[(1)] if $a^1$ and $a^2$ are nonempty then $\max(a^1) \neq \max(a^2)$ and $\max(a^1)$ is even, and
\item[(2)] if $\max(a^1) < \max(a^2)$ and $\max(a^2)$ is odd, then $\max(a^2) = \max(a^1)+1$ and $\max(a^1)-1\notin a^1$.
\end{enumerate}
These properties make it clear that the factors of $e_{\overline 1}a$ are still strictly decreasing.
\end{remark}

\begin{example}
The $f_{\overline 1}$ operator sends 
$
(6421,53)  \xmapsto{\ f_{\overline 1}\ } 
(421,653)
$ and
$(4321,32)  \xmapsto{\ f_{\overline 1}\ }  
(421,532)
$
with $e_{\overline 1}$ acting in the reverse direction.
On the other hand, $e_{\overline 1}(6421,53) = f_{\overline 1}(421,532)=\zero$.
\end{example}

We claim that the operators $e_{\overline 1}$ and $f_{\overline 1}$ make the $\gl_n$-crystal $\SpRF(z)$ into a normal $\q_n$-crystal. This follows from~\cite{Marberg2019b}; however, the relevant statements in~\cite{Marberg2019b} concern a crystal structure on \emph{increasing} rather than decreasing factorizations of
fpf-involution words. 

To translate things into our current setup, recall the definition of $\ast$ for (primed) words, permutations, and tuples of words from
 Remark~\ref{rem:negation}.
The $\ast$ operation fixes $1_\fpf$ and defines a bijection $\Ifpf_\ZZ \to \Ifpf_\ZZ$.
It follows from Remark~\ref{fpf-rmk} that for any given 
$z \in \Ifpf_\ZZ$, the map $i\mapsto i^\ast$ is a bijection $\cR^\Sp(z) \to \cR^\Sp(z^\ast)$. 
On tuples of words, the $\ast$ operation converts the $\q_n$-crystal
  in~\cite[\S3.3]{Marberg2019b}
 into what is described in the following theorem.
In this statement, fix $z \in \Ifpf_\ZZ$ and let 
\be
c^\Sp_i(z) := \abs{ \{ j \in \ZZ \mid i < j \text{ and }\min\{i, z(i)\} > z(j)\} }
\ee
for each $i \in \ZZ$.
 These numbers make up what is called the \defn{(fpf-)involution code} of $z$ in~\cite{HMP5,HMP1}.

\begin{theorem}
\label{sprf-thm}
The set $\SpRF(z)$ is nonempty if and only if $c^\Sp_i(z) \leq n$ for all $i \in\ZZ$. When this  holds, 
 $\SpRF(z)$ 
has a normal $\q_n$-crystal structure with crystal operators
$e_i$ and $f_i$ for $i \in \{1,2,\dotsc,n-1\}\sqcup \{\overline 1\}$ given 
as in Definitions~\ref{ef-def} and~\ref{def:sp_1bar}.
\end{theorem}

\begin{proof}
The set $\SpRF(z)$ is a normal $\q_n$-crystal when it is nonempty 
by~\cite[Cor.~3.37]{Marberg2019b} via the preceding discussion.
By~\cite[Rem.~3.16]{Marberg2019b}, we have $\ch\bigl( \SpRF(z^\ast) \bigr) \neq 0$  if and only if  $c^\Sp_i(z) \leq n$ for all $i \in\ZZ$.
However, comparing~\cite[Rem.~3.16]{Marberg2019b} with~\cite[Cor.~5.10]{Marberg2019a} shows that $\ch\bigl( \SpRF(z^\ast) \bigr) = \ch\bigl( \SpRF(z) \bigr)$,
so $\SpRF(z)$ is nonempty if and only if $c^\Sp_i(z) \leq n$ for all $i$.
\end{proof}

Figure~\ref{fig:Sp_crystal} shows an example of the normal $\q_n$-crystal $\SpRF(z)$.
The character of $\SpRF(z)$ is the \defn{fpf-involution Stanley symmetric polynomial}
$P_z(x_1,x_2,\dotsc,x_n)$ studied in~\cite{HMP5}; see~\cite[Rem.~3.16]{Marberg2019b}.
Our next result is a ``symplectic'' analogue of Theorem~\ref{reduced-tab-lem}.
 
 \begin{figure}[h]
\begin{center}
\input{q3-decr-crystal.tex}
\end{center}
\caption{The $\q_4$-crystal graph of $\sprf_4(z)$ for the dominant fpf-involution $z = (1\: 5)(2\: 3)(4\: 6) \in \Ifpf_6$ of shape $\lambda=(4,1,1,1)$. 
The boxed elements make up the set of bounded factorizations $\bsprf_4(z)$.
Solid blue, red, and green arrows indicate $1$-, $2$-, and $3$-edges, respectively, while dashed blue arrows are $\overline 1$-edges.
The indices of the solid arrows can also be inferred from the difference in weight between the source and target vertices.
To make this picture easier to view, we have not drawn the $\overline 2$-, $\overline 3$-, $\underline 1$-, $\underline 2$-, or $\underline 3$-edges,
as these are determined by the displayed arrows.
}
\label{fig:Sp_crystal}
\end{figure}

Define \defn{symplectic Coxeter--Knuth equivalence} to be the transitive closure $\simFCK$
of Coxeter--Knuth equivalence $\simCK$ plus the symmetric relation on words that has 
$a(a-1) \cdots \simFCK  a(a+1) \cdots$ for all $a\in\ZZ$
and 
 $ ab \cdots \simFCK ba \cdots $
 for all $a,b\in\ZZ$ with $a\equiv b \modu 2)$.
These extra relations can only change the two letters at the start of a word.
As noted in Remark~\ref{fpf-rmk}, if $z \in \Ifpf_\ZZ$ then $\cR^\Sp(z)$ is a disjoint union of symplectic Coxeter--Knuth equivalence classes.

The shifted diagram of a strict partition $\lambda = (\lambda_1 > \lambda_2 > \cdots > \lambda_k \geq 0)$
is the set of positions $\SD_\lambda := \{ (i,i+j-1) \mid (i,j) \in \D_\lambda\}$. 
A \defn{shifted tableau} of   shape $\lambda$ is
a filling of $\SD_\lambda$ by elements of $\ZZ\sqcup\ZZ'$.
The \defn{row and column reading words} of a shifted tableau are defined just as for ordinary tableaux,
as are the notions of  \defn{increasing} and \defn{decreasing}.

\begin{theorem} \label{sp-reduced-tab-lem}
Fix a $\simFCK$-equivalence class $\cK\subseteq\cR^\Sp(z)$ for  $z \in \Ifpf_\ZZ$.
Then $\cK$ contains the row reading words of a unique increasing shifted tableau $U$ and a
unique decreasing shifted tableau $V$, which have the same shape. 
\end{theorem}

\begin{proof}
Choose a word $i=i_1i_2\cdots i_k \in \cK$. 
The \defn{symplectic Edelman--Greene insertion algorithm} from~\cite{Marberg2019a} gives an increasing shifted tableau 
$P^\Sp_\EG(i)$ with $i \simFCK \row(P^\Sp_\EG(i))$ by~\cite[Cor.~3.22]{Marberg2019a}. 
This is the unique increasing shifted tableau with row reading word in $\cK$,
since if $T$ is such a tableau then
\ben
\item[(a)] $\cK$ also contains the \defn{column reading word} $\col(T)$ by~\cite[Lem.~2.7]{Marberg2019a};
\item[(b)] it is clear from the definition of $P^\Sp_\EG$ (see~\cite[Def.~3.23]{Marberg2019b}) that $T = P^\Sp_\EG(\col(T))$; and
\item[(c)] we have $P^\Sp_\EG(\col(T))=P^\Sp_\EG(i)$  
by~\cite[Thm.~4.4]{Hiroshima}
(see also~\cite[Thm.~3.35]{Marberg2019b}). 
\een
It follows that the desired shifted tableaux are
 $U := P^\Sp_\EG(i)$ and $V := P^\Sp_\EG(i^\ast)^\ast$, where $\ast$ 
 is the operation on words from Remark~\ref{rem:negation}, which we extend to tableaux
 by applying $\ast$ to each entry individually.

 Let $\mu$ and $\nu$ be the shapes of $U$ and $V$, which are also the shapes of $P^\Sp_\EG(i)$ and $P^\Sp_\EG(i^\ast)$. To show that $\mu=\nu$ we need some more notation.
For any subset $S \subseteq [k-1]$, let $F_{S,k}(\x)$ be the \defn{fundamental quasisymmetric function} given by summing all monomials $x_{a_1}x_{a_2}\cdots x_{a_k}$
with $1\leq a_1 \leq a_2 \leq \cdots \leq a_k$ and $a_j<a_{j+1}$ whenever $j \in S$.
Also let $\Des(i) := \{ j \in [k-1] \mid i_j > i_{j+1}\}$. 
Then  $\sum_{i \in \cK} F_{\Des(i),k}(\x)$ and $\sum_{i \in \cK^\ast} F_{\Des(i),k}(\x)$ are the characters of the full $\q_n$-subcrystals 
in~\cite[Thm.~3.32(a)]{Marberg2019b} corresponding to $P^\Sp_\EG(i)$ and $P^\Sp_\EG(i^\ast)$, or more precisely the  limits of these characters in the sense of formal power series as $n\to\infty$.
By~\cite[Rem.~2.13 and Thm.~3.32(b)]{Marberg2019b},
these limits respectively coincide with the Schur $P$-functions $P_\mu(\x)$ and $ P_\nu(\x)$.

Since 
$\Des(i^\ast) = [k - 1] \setminus \Des(i)$,
we have $\sum_{i \in \cK^\ast} F_{\Des(i),k}(\x) = \sum_{i \in \cK} F_{[k-1]\setminus\Des(i),k}(\x)$.
The $\ZZ$-linear map $\psi$ sending $F_{S,k}(\x) \mapsto F_{[k-1]\setminus S,k}(\x)$
is an automorphism of the ring of quasisymmetric functions~\cite[\S3.6]{KMW}
that fixes all Schur $P$-functions~\cite[\S{III}.8, Ex. 3(a)]{Macdonald}.
Therefore
{\small
\be\label{numu-eq}
P_\nu(\x) = \sum_{i \in \cK^\ast} F_{\Des(i),k}(\x) = \sum_{i \in \cK} F_{[k-1]\setminus\Des(i),k}(\x) = \psi\( \sum_{i \in \cK} F_{\Des(i),k}(\x)\) = \psi(P_\mu(\x)) = P_\mu(\x),
\ee}%
which can only hold if $\mu = \nu$ as Schur $P$-functions are linearly independent.
\end{proof}

We define a \defn{$\Sp$-reduced tableau} for $z \in \Ifpf_\infty$ to be a shifted tableau $T$ that is increasing or decreasing with $\row(T) \in \cR^\Sp(z)$.
Given such a tableau $T$, let
$\sprf_n(T) \subseteq \sprf_n(z)$ be the subsets of factorizations $a=(a^1,a^2,\ldots)$ with $a^1a^2\cdots \simFCK \row(T)$.

\begin{proposition} \label{sp-subcrystal-prop}
The map $T \mapsto \sprf_n(T)$ is a  bijection from increasing (equivalently, decreasing) $\Sp$-reduced tableaux for $z \in \Ifpf_\ZZ$ with at most $n$ rows 
to full $\q_n$-subcrystals of $\sprf_n(z)$.
 If $T$ is a $\Sp$-reduced tableau then $\sprf_n(T)$ is empty if and only if $T$ has more than $n$ rows.
\end{proposition}

\begin{proof}
Fix $z \in \Ifpf_\ZZ$.
As in the proof of Proposition~\ref{gl-subcrystal-prop}, the operation $\ast$ preserves $\simFCK$-equivalence.
Therefore~\cite[Thms. 3.35 and 3.36]{Marberg2019b} imply that each subset $\sprf_n(T)\subseteq \sprf_n(z)$ is a full subcrystal or empty, and by Theorem~\ref{sp-reduced-tab-lem} every full subcrystal arises as $\sprf_n(T)$ for some reduced tableau $T$.
 
If there exists a decreasing $\Sp$-reduced tableau $T$ with at most $n$ rows, then $\SpRF(T)$ is nonempty since it contains the sequence $a=(a^1,a^2,\dotsc,a^n)$ in which $a^i$ is row $n + 1 - i$ of $ T$.
It remains to check that $\SpRF(T)$ is empty if $T$ has more than $n$ rows.

For this, suppose $a \in \SpRF(T)$. 
Then the \defn{symplectic EG correspondence} described in~\cite[Def.~3.23, Thm.~3.26]{Marberg2019b} maps $a^\ast$ to a pair of shifted tableaux $(P,Q)$ of the same shape, where $P = P^\Sp_\EG((a^1a^2\cdots a^n)^\ast)$ and
where $Q$ is 
semistandard with all entries in $\{1'<1<\cdots<n'<n\}$.
From the proof of Theorem~\ref{sp-reduced-tab-lem}, we see that $P$ and $Q$ have the same shape as $T$.
Yet there are no semistandard shifted tableaux $Q$ with more than $n$ rows and all entries in $\{1'<1<\cdots<n'<n\}$, so $\SpRF(T)$ must be empty if $T$ has more than $n$ rows.
\end{proof}

We can characterize the $\q_n$-lowest weight elements of $\SpRF(z)$.

\begin{proposition}\label{sp-lowest-prop}
If $T$ is a decreasing $\Sp$-reduced tableau for $z \in \Ifpf_\ZZ$ with at most $n$ rows, then the unique $\q_n$-lowest weight element of $\sprf_n(T)$ is $a=(a^1, \dotsc,a^n)$, where $a^i$ is row $n + 1 - i$ of~$T$.
\end{proposition}

\begin{proof}
Suppose $T$ has shape $\mu$ and recall that $T = P^{\Sp}_{\EG}\bigl( (a^1a^2\cdots a^n)^\ast \bigr)^\ast$.
Composing $\ast $ with the map in~\cite[Thm.~3.36(b)]{Marberg2019b}
gives an isomorphism from $\sprf_n(T)$ to a connected normal $\q_n$-crystal of semistandard shifted tableaux of shape $\mu$,
whose unique $\q_n$-lowest weight element is also its unique element of weight $(\mu_n,\dotsc,\mu_2,\mu_1)$ by~\cite[Thm.~3.3]{Hiroshima2018}. 
Since the indicated factorization $a$ also has weight $(\mu_n,\dotsc,\mu_2,\mu_1)$,
it must be the unique $\q_n$-lowest weight element.
\end{proof}

\begin{corollary}\label{sp-highest-cor}
If $T$ is a $\Sp$-reduced tableau with at most $n$ rows, then the unique highest weight of the connected normal $\q_n$-crystal $\sprf_n(T)$ is the partition shape of $T$.
\end{corollary}

\begin{proof}
By Theorem~\ref{sp-reduced-tab-lem} we may assume that $T$ is decreasing. 
If this holds  and $T$ has shape $\mu \in \NN^n$, then
Lemma~\ref{sp-lowest-prop} shows that the unique lowest weight of  $\sprf_n(T)$ is $(\mu_n,\dotsc,\mu_2,\mu_1)$.
Reversing this tuple gives the unique highest weight by the discussion in Section~\ref{q-sect}.
\end{proof}

\begin{example}
The unique increasing and decreasing $\Sp$-reduced tableaux for the involution $z=(1\: 5)(2\: 3)(4\: 6) \in\Ifpf_6$ are $\ytabsmall{2 & 3 & 4}$ and $\ytabsmall{4&2&1}$.
Compare these with Figure~\ref{fig:Sp_crystal} and Proposition~\ref{sp-lowest-prop}.
\end{example}

Unlike the $\gl_n$-case, the increasing $\Sp$-reduced tableaux for $z \in \Ifpf_\ZZ$
do not identify the highest weight elements of $\sprf_n(z)$ in any simple way.
The following problem is open:

\begin{problem}
Describe the highest weight elements of the $\q_n$-crystals $\sprf_n(z)$ for $z \in \Ifpf_\ZZ$.
\end{problem}

Recall our definition of skew-symmetric partitions
from Section~\ref{key-sect}.
If $\lambda$ is any partition then let $\shalf(\lambda)$ be the strict partition
whose nonzero parts are $\{ \lambda_i - i \mid i \in \PP\} \cap \PP$.
For example, if
\[ \lambda = (4,3,3,1)= \ytabsmall{ \ & \ & \ & \ \\ \ & \ & \ \\ \ & \ & \ \\ \ }
\quad\text{then}\quad  \shalf(\lambda)=(3,1).\] 
The map $\lambda \mapsto \shalf(\lambda)$ is a bijection from skew-symmetric partitions
 to strict partitions.

 \begin{proposition}\label{sp-dom-prop}
Suppose $\lambda$ is a skew-symmetric partition and $\mu=\shalf(\lambda)$.
Let $z=z^\fpf_\lambda$ and define $T^\Sp_\lambda$ to be the shifted tableau of shape $\mu$
with entry $i+j$ in each position $(i,j) \in \SD_\mu$.
Then  $\cR^\Sp(z)$ is a single $\simFCK$-equivalence class
and $T^\Sp_\lambda$ is the unique increasing $\Sp$-reduced tableau for $z$.
\end{proposition}

\begin{proof}
Define $F_{S,k}(\x)$ as in the proof of Theorem~\ref{sp-reduced-tab-lem} and
let $\widehat F^\fpf_z(\x) := \sum_{i \in \cR^\Sp(z)} F_{[k-1]\setminus\Des(i),k}(\x)$. This
is the \defn{fpf-involution Stanley symmetric function} of $z$ studied in~\cite{HMP1,HMP5}.
In view of~\eqref{numu-eq}, 
to prove that $\cR^\Sp(z^\fpf_\lambda)$ is a single $\simFCK$-equivalence class it suffices to show that
 $\widehat F^\fpf_z(\x) =P_{\mu}(\x)$.
 
 This identity follows from~\cite[Thm 1.4 and Cor.~7.9]{HMP5}.
 The first result asserts that $\widehat F^\fpf_z(\x) = P_\mu(\x) + \text{(lower order Schur $P$-terms in dominance order)}$,
 and the second result asserts that there are no lower order terms if $z$ avoids 
 the following list of 16 involution patterns:
 {\tiny\[
 \ba 
 &(1\: 3)(2\: 4)(5\: 8)(6\: 7),\quad(1\: 3)(2\: 5)(4\: 7)(6\: 8),\quad(1\: 3)(2\: 5)(4\: 8)(6\: 7),\quad(1\: 3)(2\: 6)(4\: 8)(5\: 7),\quad(1\: 4)(2\: 3)(5\: 7)(6\: 8),\quad(1\: 4)(2\: 3)(5\: 8)(6\: 7),\quad
 \\&(1\: 5)(2\: 3)(4\: 7)(6\: 8),\quad(1\: 5)(2\: 3)(4\: 8)(6\: 7),\quad(1\: 5)(2\: 4)(3\: 7)(6\: 8),\quad(1\: 5)(2\: 4)(3\: 8)(6\: 7),\quad
 (1\: 6)(2\: 3)(4\: 8)(5\: 7),\quad(1\: 6)(2\: 4)(3\: 8)(5\: 7),\quad
 \\&(1\: 6)(2\: 5)(3\: 8)(4\: 7),\quad(1\: 3)(2\: 4)(5\: 7)(6\: 9)(8\: 10),\quad(1\: 3)(2\: 5)(4\: 6)(7\: 9)(8\: 10),\quand(1\: 3)(2\: 4)(5\: 7)(6\: 8)(9\: 11)(10\: 12).
 \ea
 \]}%
 On inspecting the Rothe diagrams of these patterns, it becomes evident that any $y \in \Ifpf_\infty$ with $\{ (i,j) \in D(y) \mid i \neq j\} = \{ (i,j) \in \D_\lambda \mid i\neq j\}$ for a skew-symmetric partition $\lambda$ must avoid all of them.
 For example, the Rothe diagram of $(1\: 3)(2\: 4)(5\: 8)(6\: 7)$ is
 \[
  \left[\begin{smallmatrix} 
\square&\square&1&\cdot&\cdot&\cdot&\cdot&\cdot \\
\square&\square&\cdot&1&\cdot&\cdot&\cdot&\cdot \\
1&\cdot&\cdot&\cdot&\cdot&\cdot&\cdot & \cdot \\
\cdot&1&\cdot&\cdot&\cdot&\cdot&\cdot&\cdot \\
\cdot&\cdot&\cdot&\cdot&\square&\square&\square&1 \\
\cdot&\cdot&\cdot&\cdot&\square&\square&1&\cdot \\
\cdot&\cdot&\cdot&\cdot&\square &1&\cdot &\cdot \\
\cdot&\cdot&\cdot&\cdot&1&\cdot &\cdot&\cdot
\end{smallmatrix} \right]
\]
so if  $y \in \Ifpf_\infty$ does not avoid this pattern, then
its Rothe diagram will contain two positions $(a,b)$ and $(c,d)$ with $a<b<c<d$ 
while not containing $(a,d)$,
which is impossible if $y$ is dominant.

By Theorem~\ref{sp-reduced-tab-lem}, it remains only to show that $\row(T^\Sp_\lambda) \in \cR^\Sp(z)$, and this follows by combining~\cite[Thm.~3.12 and Lem.~4.30]{HMP6}.
%
\end{proof}

If $\lambda$ is a skew-symmetric partition then define $\SpRF(\lambda) := \SpRF(z^\fpf_\lambda)$,
where $z^\fpf_\lambda \in \Ifpf_\infty$ is dominant with shape $\lambda$.
By Theorem~\ref{sp-reduced-tab-lem}, there is a unique increasing $\Sp$-reduced tableau for $z^\fpf_\lambda$, which has shape $\shalf(\lambda)$, so the following holds by
Propositions~\ref{sp-subcrystal-prop} and \ref{sp-lowest-prop}:

\begin{corollary}
\label{sp-sum-cor}
Suppose $\lambda$ is a skew-symmetric partition
and $\mu = \shalf(\lambda)$. 
The normal $\q_n$-crystal $\SpRF(\lambda)$ is nonempty if and only if $\mu \in \NN^n$,
in which case it is connected with unique highest weight $\mu$.
Hence, each connected normal $\q_n$-crystal
is isomorphic to a unique $ \SpRF(\lambda)$.
 \end{corollary}

 \subsection{Demazure \texorpdfstring{$\q_n$}{qn}-crystals}

Choose an element $z \in \Ifpf_\infty$ and a flag $\phi = (\phi_1 \leq \phi_2 \leq \cdots)$.
Let 
\be\label{bsprf-eq}
\bsprf_n(z,\phi) :=  \bigsqcup_{w \in \cA^\Sp(z)} \brf_n(w,\phi)\subseteq \sprf_n(z)
\ee
be the set of factorizations 
$a=(a^1,a^2,\dotsc,a^n) \in  \sprf_n(z)$
that are bounded by $\phi$ in the sense that every letter $m$ in $a^i$ has $i \leq \phi_m$.
We view this set as a $\q_n$-subcrystal of $\sprf_n(z)$;
by Corollary~\ref{is-dem-cor}, it is also
 a direct sum of Demazure $\gl_n$-crystals.
 Recall that the standard flag is $\phi^S = (1<2<3<\cdots)$ and 
write $\bsprf_n(z) := \bsprf_n(z,\phi^S)$.

 \begin{example}\label{bsprf-ex}
The crystal graph of $\bsprf_3(z)$ for $z= (1\: 5)(2\: 3)(4\: 7)(6\: 8)\in \Ifpf_8$ is 
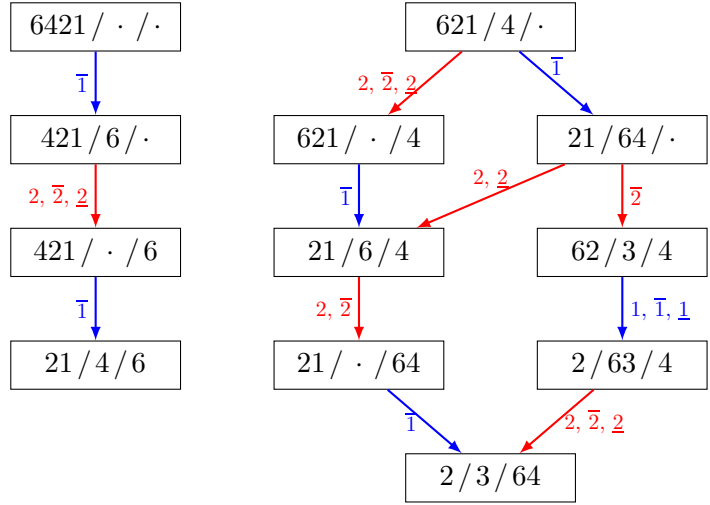
\[
\begin{tikzpicture}[xscale=3.5,yscale=1.5,>=latex]
\node [bnode,text width=20mm] (a1) at (0, 0) {$6421 \gap \cdot \gap \cdot$};
\node [bnode,text width=20mm] (a2) at (0, -1) {$421 \gap 6 \gap \cdot$};
\node [bnode,text width=20mm] (a3) at (0, -2) {$421 \gap \cdot \gap 6$};
\node [bnode,text width=20mm] (a4) at (0, -3) {$21 \gap 4 \gap 6$};
\draw[->,thick,color=blue] (a1) -- (a2) node[midway,left,scale=0.75] {$\overline 1$};
\draw[->,thick,color=red] (a2) -- (a3) node[midway,left,scale=0.75] {$2$, $\overline 2$, $\underline 2$};
\draw[->,thick,color=blue] (a3) -- (a4) node[midway,left,scale=0.75] {$\overline 1$};
\node [bnode,text width=20mm] (b1) at (1.5, 0) {$621 \gap 4 \gap \cdot$};
\node [bnode,text width=20mm] (b2) at (1, -1) {$621 \gap \cdot \gap 4$};
\node [bnode,text width=20mm] (c2) at (2, -1) {$21 \gap 64 \gap \cdot$};
\node [bnode,text width=20mm] (b3) at (1, -2) {$21 \gap 6 \gap 4$};
\node [bnode,text width=20mm] (c3) at (2, -2) {$62 \gap 3 \gap 4$};
\node [bnode,text width=20mm] (b4) at (1, -3) {$21 \gap \cdot \gap 64$};
\node [bnode,text width=20mm] (c4) at (2, -3) {$2 \gap 63 \gap 4$};
\node [bnode,text width=20mm] (b5) at (1.5, -4) {$2 \gap 3 \gap 64$};
\draw[->,thick,color=red] (b1) -- (b2) node[midway,left,scale=0.75] {$2$, $\bar2$, $\underline 2$};
\draw[->,thick,color=blue] (b1) -- (c2) node[midway,above,scale=0.75] {$\bar1$};
\draw[->,thick,color=blue] (b2) -- (b3) node[midway,left,scale=0.75] {$\bar1$};
\draw[->,thick,color=red] (c2) -- (b3) node[midway,above,scale=0.75] {$2$, $\underline 2$};
\draw[->,thick,color=red] (c2) -- (c3) node[midway,right,scale=0.75] {$\bar2$};
\draw[->,thick,color=red] (b3) -- (b4) node[midway,left,scale=0.75] {$2$, $\bar2$};
\draw[->,thick,color=blue] (c3) -- (c4) node[midway,right,scale=0.75] {$1$, $\bar1$, $\underline 1$};
\draw[->,thick,color=blue] (b4) -- (b5) node[midway,left,scale=0.75] {$\overline 1$};
\draw[->,thick,color=red] (c4) -- (b5) node[midway,right,scale=0.75] {$2$, $\bar2$, $\underline 2$};
\end{tikzpicture}
\]
For another example of $\bsprf_n(z)$ see Figure~\ref{fig:Sp_crystal}. 
\end{example}

The characters of the $\q_n$-crystals $\bsprf_n(z)$
are closely related to the involution Schubert polynomials of symplectic type $\fkSS_z$
defined in Section~\ref{schub-sect}. We mention one property of these polynomials
which does not seem to appear in prior literature.
Define an integer $i \in \ZZ$ to be an \defn{fpf-descent} of $z \in \Ifpf_\ZZ$
if $i+1\neq z(i) > z(i+1) \neq i$. If $z \in \Ifpf_\infty$ then all of its fpf-descents are positive.

\begin{proposition}\label{sp-schub-prop}
Suppose $z \in \Ifpf_\infty$. Then $\fkSS_z \in \ZZ[x_1,x_2,\dots,x_n]$  if and only if $z$ has no fpf-descents greater than $n$.
\end{proposition}

\begin{proof}
 If $z=1_\fpf$ then $z$ has no fpf-descents and $ \fkSS_z=1\in \ZZ[x_1,x_2,\dots,x_n]$ for all $n$.
 Assume $z \neq 1_\fpf$ so that  $\fkSS_z$ is homogeneous of positive degree.
Let $i$ be the maximal index such that $x_i$ divides some monomial appearing in $\fkSS_z$.
We have 
$\fkSS_z \in \ZZ[x_1,x_2,\dots,x_n]$
 if and only if $i \leq n$.
As $i$ is also the maximal index with $\partial_i \fkSS_z \neq 0$,
it follows from \eqref{fschub-eq1} that $i \leq n$ if  and only if $z(j) < z(j+1)$ or $z(j)=j+1$ for each $j>n$, meaning $z$ has no fpf-descents greater than $n$.
\end{proof}

By  \eqref{atom-eq-sp}, \eqref{cRsp-eq}, and the Billey--Jockusch--Stanley formula 
 we have $\fkSS_z = \sum_{a \in \bsprf(z)} \xx^{\weight(a)}$ for each $z \in \Ifpf_\infty$.
Proposition~\ref{sp-schub-prop} tells us that if $z$ has no fpf-descents greater than $n$ then 
 \be\label{sp-schub-prop-eq}
 \BSpRF(z) = \bsprf(z) \quand \ch(\BSpRF(z)) = \fkSS_z.
 \ee
If $z$ has at least one fpf-descent greater than $n$,
then $ \BSpRF(z) $ is a proper subset of $ \bsprf(z)$ and 
$ \ch(\BSpRF(z)) $ is the polynomial obtained from $\fkSS_z$ by setting $x_{n+1}=x_{n+2}= \cdots = 0$.

 \begin{example}
The involution Schubert polynomial of $z =  (1\: 5)(2\: 3)(4\: 7)(6\: 8)\in \Ifpf_8$ is 
a polynomial in $x_1,x_2,x_3,x_4,x_5,x_6$ with $36$ terms and coefficients in
$\{1,2,3,4\}$, given by
\[
\fkSS_z = \xx^{011101} + \xx^{011110} + \xx^{011200} + \xx^{012100} + \cdots + 2\xx^{300100} + 2\xx^{301000} + 2\xx^{310000} + \xx^{400000},
\]
while $\ch\bigl( \bsprf_3(z) \bigr) = \xx^{112} + \xx^{121} + \xx^{202} + 3\xx^{211} + \xx^{220} + 2\xx^{301} + 2\xx^{310} + \xx^{400}$,
which is the polynomial obtained from $\fkSS_z$ by setting $x_4=x_5=x_6=0$.
\end{example}

Suppose 
$\alpha   $ is a skew-symmetric weak composition and  $\lambda:=\lambda(\alpha)$.
Assume   that  $ u(\alpha) \in S_n$ and $\shalf(\lambda) \in \NN^n$, so $\sprf_n(\lambda)$ is nonempty. 
Both conditions must hold if $\alpha \in \NN^n$, but this is not necessary (for example, if $\alpha = (5,1,1,1,1)$ and $n = 3$ as in Example~\ref{51111-ex} below).
Let $z^\fpf_\lambda \in \Ifpf_\infty$ be the  dominant element of shape $\lambda$.
Then by Lemma~\ref{sp-ij-lem} we can define
\be\label{bsprf-def}
\bsprf_n(\alpha) := \fkD_{i_1} \fkD_{i_2}\cdots \fkD_{i_k}\brf_n(z^\fpf_\lambda) \subseteq \sprf_n(\lambda),
\ee
where $\fkD_i = \fkD_i^{\sprf_n(\lambda)}$ and $i_1i_2\cdots i_k $ is any reduced word for $u(\alpha)$.
This gives a (nonempty) $\q_n$-crystal that is also a direct sum of Demazure $\gl_n$-crystals.

\begin{example}\label{51111-ex}
We have $z^\fpf_{(5,1,1,1,1)} = (1\: 5)(2\: 4)(3\: 6) \in \Ifpf_6$ and $z^\fpf_{(4,3,3,1)} = (1\: 5)(2\: 4)(3\: 6) \in\Ifpf_6$.
The connected crystal graphs of $\bsprf_3(\alpha)$ for $\alpha=(5,1,1,1,1)$ and $(4,3,3,1)$ are respectively
\[
\begin{tikzpicture}[xscale=3.5,yscale=1.5,>=latex]
\node [bnode,text width=20mm] (a1) at (0, 0) {$4321 \gap \cdot \gap \cdot$};
\node [bnode,text width=20mm] (a2) at (0, -1) {$421 \gap 5 \gap \cdot$};
\node [bnode,text width=20mm] (a3) at (0, -2) {$421 \gap \cdot \gap 5$};
\node [bnode,text width=20mm] (a4) at (0, -3) {$21 \gap 4 \gap 5$};
\draw[->,thick,color=blue] (a1) -- (a2) node[midway,left,scale=0.75] {$\overline 1$};
\draw[->,thick,color=red] (a2) -- (a3) node[midway,left,scale=0.75] {$2$, $\overline 2$, $\underline 2$};
\draw[->,thick,color=blue] (a3) -- (a4) node[midway,left,scale=0.75] {$\overline 1$};
\node [bnode,text width=20mm] (b1) at (1.5, 0) {$421 \gap 3 \gap \cdot$};
\node [bnode,text width=20mm] (b2) at (1, -1) {$421 \gap \cdot \gap 3$};
\node [bnode,text width=20mm] (c2) at (2, -1) {$21 \gap 43 \gap \cdot$};
\node [bnode,text width=20mm] (b3) at (1, -2) {$21 \gap 4 \gap 3$};
\node [bnode,text width=20mm] (c3) at (2, -2) {$42 \gap 3 \gap 4$};
\node [bnode,text width=20mm] (b4) at (1, -3) {$21 \gap \cdot \gap 43$};
\node [bnode,text width=20mm] (c4) at (2, -3) {$2 \gap 43 \gap 4$};
\node [bnode,text width=20mm] (b5) at (1.5, -4) {$2 \gap 3 \gap 43$};
\draw[->,thick,color=red] (b1) -- (b2) node[midway,left,scale=0.75] {$2$, $\bar2$, $\underline 2$};
\draw[->,thick,color=blue] (b1) -- (c2) node[midway,above,scale=0.75] {$\bar1$};
\draw[->,thick,color=blue] (b2) -- (b3) node[midway,left,scale=0.75] {$\bar1$};
\draw[->,thick,color=red] (c2) -- (b3) node[midway,above,scale=0.75] {$2$, $\underline 2$};
\draw[->,thick,color=red] (c2) -- (c3) node[midway,right,scale=0.75] {$\bar2$};
\draw[->,thick,color=red] (b3) -- (b4) node[midway,left,scale=0.75] {$2$, $\bar2$};
\draw[->,thick,color=blue] (c3) -- (c4) node[midway,right,scale=0.75] {$1$, $\bar1$, $\underline 1$};
\draw[->,thick,color=blue] (b4) -- (b5) node[midway,left,scale=0.75] {$\overline 1$};
\draw[->,thick,color=red] (c4) -- (b5) node[midway,right,scale=0.75] {$2$, $\bar2$, $\underline 2$};
\end{tikzpicture}
\]
Observe that $\bsprf_3\bigl( (1\: 5)(2\: 3)(4\: 7)(6\: 8) \bigr) \iso \bsprf_3\bigl( (5,1,1,1,1) \bigr) \oplus\bsprf_3\bigl( (4,3,3,1) \bigr)$.
\end{example}

\begin{example}
Unlike $\gl_n$-crystals $\BRF(\alpha) \subseteq \RF(\lambda(\alpha))$,
the subsets $ \bsprf_n(\alpha)  \subseteq \sprf_n(\lambda(\alpha))$ are not closed by all raising crystal operators $e_i$, at least if we consider $i \in \overline{I} \sqcup \underline{I}$.
The smallest example demonstrating this behavior is $ \bsprf_3(\alpha)$ for $\alpha = (3,1,1)$,
which is made up of the boxed vertices in the $\q_n$-crystal graph of $\sprf_3(\alpha)$ shown below:
 
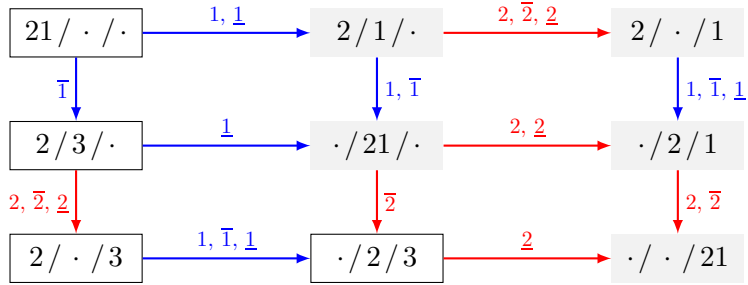
\[
\begin{tikzpicture}[xscale=4,yscale=1.5,>=latex]
\node (a1) at (0, 0) [bnode] {$21 \gap \cdot \gap \cdot$};
\node (a2) at (0, -1) [bnode] {$ {2 \gap 3 \gap \cdot}$};
\node (a3) at (0, -2) [bnode] {$ {2 \gap \cdot \gap 3}$};
\node (a4) at (1, -2) [bnode] {$ {\cdot \gap 2 \gap 3}$};
\node (b1) at (1, 0) [onode]  {$2 \gap 1 \gap \cdot$};
\node (b2) at (1, -1) [onode] {$\cdot \gap 21 \gap \cdot$};
\node (c1) at (2, 0) [onode] {$2 \gap \cdot \gap 1$};
\node (c2) at (2, -1) [onode] {$\cdot \gap 2 \gap 1$};
\node (c3) at (2, -2) [onode] {$\cdot \gap \cdot \gap 21$};
\draw[->,thick,color=blue] (a1) -- node[midway,left,scale=0.75] { $\overline 1$} (a2);
\draw[->,thick,color=blue] (a1) -- (b1) node[midway,above,scale=0.75] {  $1$, $ \underline 1$};
\draw[->,thick,color=red] (a2) -- (a3)  node[midway,left,scale=0.75] { 2, $\overline 2$, $ \underline 2$} ;
\draw[->,thick,color=blue] (a2) -- (b2) node[midway,above,scale=0.75] { $ \underline 1$};
\draw[->,thick,color=blue] (a3) -- (a4) node[midway,above,scale=0.75] {  $1$, $ \bar1$, $\underline 1$};
\draw[->,thick,color=blue] (b1) -- (b2)  node[midway,right,scale=0.75] { $1$, $\overline 1$} ;
\draw[->,thick,color=red] (b1) -- (c1) node[midway,above,scale=0.75] {  $2$, $ \bar2$, $\underline 2$};
\draw[->,thick,color=red] (b2) -- node[midway,right,scale=0.75] { $\overline 2$} (a4);
\draw[->,thick,color=red] (b2) -- (c2)  node[midway,above,scale=0.75] { $2$, $ \underline 2$};
\draw[->,thick,color=blue] (c1) -- (c2)  node[midway,right,scale=0.75] { $1$, $\overline 1$, $ \underline 1$};
\draw[->,thick,color=red] (c2) -- (c3)  node[midway,right,scale=0.75] { $2$, $\overline 2$} ;
\draw[->,thick,color=red] (a4) -- (c3) node[midway,above,scale=0.75] { $ \underline 2$};
\end{tikzpicture}
\]
\end{example}


Among the compositions $\alpha$ allowed in \eqref{bsprf-def}, there is a property distinguishing  the ones in $ \NN^n$:

\begin{proposition}\label{among-the-prop}
Suppose $\alpha$ is a skew-symmetric weak composition with $\shalf(\lambda(\alpha)) \in \NN^n$ and $u(\alpha)\in S_n$.
Then $\pkey_\alpha \in \ZZ[x_1,x_2,\dots,x_n]$ if and only if $\alpha \in \NN^n$.
\end{proposition}

\begin{proof}
Let $\lambda = \lambda(\alpha)$.
If $\alpha \in \NN^n$ then $\lambda=\lambda^\top \in \NN^n$, so we have $\pkey_\lambda \in \ZZ[x_1,x_2,\dots,x_n]$ 
and therefore
$\pkey_\alpha = \pi_{u(\alpha)}\pkey_\lambda \in \ZZ[x_1,x_2,\dots,x_n]$ as $u(\alpha) \in S_n$.

Let $\beta$ be the weak composition with $\beta_i = |\{ (j,k) \in \D_\lambda : j<k=i\}|$ for each $i \in \PP$.
Then~\cite[Eq.~(2.3) and Prop.~2.49]{MS2023} imply that 
$\xx^{u(\alpha) \circ \beta}$ is a monomial appearing in $\pkey_\alpha$.
If $\alpha \notin \NN^n$ then we must have $\lambda \notin \NN^n$ as $u(\alpha) \in S_n$, 
in which case $ \beta \notin\NN^n$ so $u(\alpha) \circ \beta \notin\NN^n$ and
$\pkey_\alpha \notin \ZZ[x_1,x_2,\dotsc,x_n]$.
\end{proof}

We can prove that the crystal graph of $\bsprf_n(\alpha) $ is always connected.
This will require two lemmas.
The first lemma is equivalent to a difficult technical property shown in~\cite{GHPS}.
Here, recall that $[i] =\{1,2,\dots,i\}$ and so an $[i]$-highest weight element of a crystal is one satisfying
$e_1(b) = e_2(b) =\dots =e_i(b) =\zero$.

\begin{lemma}[{\cite[Prop.~2.24]{GHPS}}]
\label{ghps-lem}
Let $\cB$ be a normal $\q_n$-crystal.
Suppose $b \in \cB$ is a $\gl_n$-highest weight element and $i \in [n-1]$ is minimal with $e_{\overline{\imath}}(b)\neq \zero$.
Then  $e_{\overline{\imath}}(b)$ is $[i]$-highest weight.
\end{lemma}

For the next lemma, continue to fix an element  $z \in \Ifpf_\infty$ and a flag $\phi$.

\begin{lemma}\label{sp-connect-lem}
Every $\gl_n$-highest weight element of $\sprf_n(z)$ is also in $\bsprf_n(z,\phi)$.
Additionally, if $a \in \bsprf_n(z,\phi)$ is not $\q_n$-highest weight then $\zero\neq e_i(a) \in \bsprf_n(z,\phi)$ for some $i \in I \sqcup \overline I$.
\end{lemma}

\begin{proof}
Each $\gl_n$-highest weight element of $\sprf_n(z)$ is a $\gl_n$-highest weight
element of $\rf_n(w)$ for some $w \in \cA^\Sp(z)$
so is bounded by $\phi$ by Lemma~\ref{each-highest-lem}.

Suppose $a \in \bsprf_n(z,\phi)$ is not $\q_n$-highest weight. 
Then $a \in \brf_n(w,\phi)$ for some $w \in \cA^\Sp(z)$.
If $a$ is not $\gl_n$-highest weight then $\zero\neq e_i(a) \in \brf_n(w,\phi)\subseteq \bsprf_n(z,\phi)$ for some $i \in [n-1]$ by Lemma~\ref{brf-e-lem}, as needed. Assume $a$ is $\gl_n$-highest weight.
Then there must exist a minimal $i \in [n-1]$ with $e_{\overline{\imath}}(a)\neq \zero$,
and $e_{\overline{\imath}}(a)$ is $[i]$-highest weight by Lemma~\ref{ghps-lem}.
Form $b=(b^1,b^2,\ldots)$ from $e_{\overline{\imath}}(a)$ by retaining the first $i+1$ factors and setting all other factors to be empty.
Since $e_{\overline{\imath}}(a)$ only differs from $a$ in its first $i+1$ factors, it suffices to show that $b$ is bounded by $\phi$.
This holds by Lemma~\ref{each-highest-lem}, since $b$ is a $\gl_{i+1}$-highest weight element of $\rf_{i+1}(v,\phi)$
for some other $v \in S_\infty$.
\end{proof}

We pause to note one consequence of Lemma~\ref{sp-connect-lem}.
Suppose $T$ is a $\Sp$-reduced tableau for $z \in \Ifpf_\infty$ and $\phi$ is a flag.
Let $\bsprf_n(T,\phi) := \sprf_n(T) \cap \bsprf_n(z,\phi)$ and $\bsprf_n(T) := \bsprf_n(T,\phi^S)$.

\begin{corollary}\label{sp-dem-sub-cor}
The map $T \mapsto \bsprf_n(T,\phi)$ is a  bijection from increasing (equivalently, decreasing) $\Sp$-reduced tableaux for $z \in \Ifpf_\infty$ with at most $n$ rows to full subcrystals of $\bsprf_n(z,\phi)$.
\end{corollary}

\begin{proof}
This is clear from Proposition~\ref{sp-subcrystal-prop} and Lemma~\ref{sp-connect-lem}:
\end{proof}

We now can prove one of our main results:

\begin{theorem}\label{sp-crystal-thm}
Suppose $\alpha$ is a skew-symmetric weak composition with $\shalf(\lambda(\alpha)) \in \NN^n$ and $u(\alpha)\in S_n$.
The (nonempty) $\q_n$-crystal $\bsprf_n(\alpha) $ is connected 
and its character is the polynomial obtained from $\pkey_\alpha$
by setting $x_{n+1}=x_{n+2}=\cdots=0$.
\end{theorem}

\begin{proof}
Let $\lambda = \lambda(\alpha)$ and $z=z^\fpf_\lambda$.
To show that $\bsprf_n(\alpha) $ is connected it suffices to show that $\bsprf_n(\lambda) = \bsprf_n(z) $ is connected, and this follows from Corollary~\ref{sp-sum-cor} and Lemma~\ref{sp-connect-lem}.

If $N \geq n$ is sufficiently large then 
\eqref{sp-schub-prop-eq} 
 tells us that $\ch(\bsprf_N(z))=  \fkSS_{z} = \pkey_\lambda$ so Lemma~\ref{sp-ij-lem} implies that
  $\ch(\bsprf_N(\alpha)) = \pi_{u(\alpha)} \pkey_\lambda  =  \pkey_\alpha$. 
But if $N \geq n$ then  $\ch(\bsprf_n(\alpha))$ is always the polynomial obtained from $\ch(\bsprf_N(\alpha)) $ by  $x_{n+1}=x_{n+2}=\cdots=0$.
\end{proof}

Mimicking the $\gl_n$-case, we introduce the following terminology.

\begin{definition}
A \defn{Demazure $\q_n$-crystal} is a $\q_n$-crystal isomorphic to $\bsprf_n(\alpha)$ for some skew-symmetric weak composition $\alpha\in \NN^n$.
\end{definition}

We remark that in contrast to the $\gl_n$-case, this definition is not  directly motivated by the representation theory of $\q_n$ or $U_q(\q_n)$.

This definition does not include  $\bsprf_n(\alpha)$ if $\shalf(\lambda(\alpha)) \in \NN^n$ and $u(\alpha)\in S_n$ but $\alpha \notin \NN^n$.
Although the $\q_n$-crystal is still defined and connected in this case, its character is not the $P$-key polynomial $\kappa_\alpha$,
but is instead the polynomial obtained from  $\kappa_\alpha$ by setting $x_{n+1}=x_{n+2}=\dots=0$.
This function might not be equal to any linear combination of $P$-key polynomials (with an infinite number of variables).

By Theorem~\ref{sp-crystal-thm}, every Demazure $\q_n$-crystal
has a unique highest weight element. This leads immediately to an analogue of Proposition~\ref{unique-embed-prop},
whose straightforward proof is omitted:

\begin{proposition}\label{unique-embed-prop2}
Suppose $\cB$ is a connected normal $\q_n$-crystal with highest weight element $b$.
Let $\cX$ be a Demazure $\q_n$-crystal with highest weight element $u$.
\ben
\item[(a)] There is a unique embedding $\cX \to \cB$ if $\weight(b) = \weight(u)$.
\item[(b)] There are no embeddings $\cX \to \cB$ if $\weight(b) \neq \weight(u)$.
\item[(c)] If $\cX\subseteq \cB$ then $u=b$ and $\cX = \fkD^\cB_{i_1} \fkD^\cB_{i_2}\cdots \fkD^\cB_{i_k} \{b\} $ for some $i_1,i_2,\dotsc,i_k \in [n-1]$.
\item[(d)] Any skew-symmetric $\alpha \in \NN^n$ with $  \bsprf_n(\alpha) \cong \cX$ has 
 $\pkey_\alpha=\ch(\cX)$ and $\shalf(\lambda(\alpha)) = \weight(u)$.
\een
\end{proposition}

The main open problem concerning the $\q_n$-crystals $\bsprf_n(z,\phi)$ is the following conjecture.
 
 \begin{conjecture}\label{sp-demazure-conj}
If $z \in \Ifpf_\infty$ has no fpf-descents greater than $n$ 
and  $\phi$ is any flag,  then $\bsprf_n(z,\phi)$ is a direct sum of Demazure $\q_n$-crystals.
\end{conjecture}
 
 Conjecture~\ref{sp-demazure-conj} is a generalization of the symplectic half of Conjecture~\ref{fkSS-conj},
which follows by taking characters when $\phi=\phi^S$.
It will turn out that Conjecture~\ref{sp-demazure-conj} can be reformulated as the following statement:

\begin{conjecture}\label{sp-demazure-conj2}
 Suppose $T$ is a $\Sp$-reduced tableau for an involution $z \in \Ifpf_\infty$ that has no fpf-descents greater than $n$.
Then there is a $\q_n$-crystal isomorphism $\bsprf_{n}(T) \iso \bsprf_n(\alpha )$
for some skew-symmetric weak composition  $\alpha=\alpha^\Sp(T) \in \NN^n$.
\end{conjecture}

 We have checked this conjecture by computer for all $z \in \Ifpf_8$. 
Figures~\ref{sp-fig1} and~\ref{sp-fig2} contains examples of the skew-symmetric weak compositions $\alpha^\Sp(T)$
 that correspond to various $\Sp$-reduced tableaux.

Suppose $T$ is a $\Sp$-reduced tableau for some $z \in \Ifpf_\infty$.
Let $\mu$ be the strict partition shape of $T$ and let $\lambda$ be the unique skew-symmetric partition with $\shalf(\lambda)=\mu$. Recall the definitions of the permutations $\Delta_n(\phi)$ and $u_n(\alpha)$ from \eqref{Delta-eq}.
Finally assume  $\alpha$ is a skew-symmetric weak composition
and choose some $N \in \PP$  greater than or equal to $\ell(\alpha)$ and every fpf-descent of $z$.

\begin{proposition}\label{sp-alpha-prop}
If there exists a $\q_N$-crystal isomorphism $\bsprf_{N}(T) \iso \bsprf_N(\alpha)$,
then 
 \[\bsprf_{n}(T,\phi) \iso \bsprf_n(\Delta_n(\phi) \circ u_n(\alpha) \circ \lambda(\alpha))\]
 as  $\q_n$-crystals
for all integers $n \geq \ell(\mu)$ and all flags $\phi$.
\end{proposition}

\begin{proof}
Suppose  $\bsprf_{N}(T) \iso \bsprf_N(\alpha)$ so that $\bsprf_{N}(T) $ is a Demazure $\q_N$-crystal.
In view of Corollary~\ref{sp-highest-cor}, Proposition~\ref{unique-embed-prop2}(d) implies that $\shalf(\lambda(\alpha)) =  \mu$ so $\lambda(\alpha) = \lambda$.
Fix $n\geq \ell(\mu)$ and a flag $\phi$, and define $\beta:=\Delta_n(\phi) \circ u_n(\alpha) \circ \lambda(\alpha)$.
Since
$\shalf(\lambda(\beta)) = \shalf(\lambda )=\mu \in \NN^n$ and $u(\beta) \in S_n$,
the $\q_n$-crystal $\bsprf_n(\beta)$ 
is well-defined (though not a Demazure $\q_n$-crystal if $\lambda \notin \NN^n$).

Suppose $\phi=\phi^S$ so that $\Delta_n(\phi) = 1$.
If $n \geq N$ then $\alpha=\beta$ and $ \bsprf_n(z) = \bsprf_N(z) = \bsprf(z)$.
In this case we must also have $\bsprf_n(T) = \bsprf_N(T)$, so
 the $\q_N$-crystal isomorphism $ \bsprf_n(T) \iso \bsprf_n(\alpha)$ defines a $\q_n$-crystal isomorphism
$ \bsprf_n(T) \iso \bsprf_n(\beta)$.
On the other hand, 
if $\ell(\mu)\leq n < N$ then for any reduced word $i_1i_2\cdots i_k \in \cR(u(\alpha))$
we have 
 \[\bsprf_n(T) = \fkR_n \bsprf_N(T) \iso \fkR_n \bsprf_N(\alpha) =  \fkR_n\fkD_{i_1}\fkD_{i_2}\cdots \fkD_{i_k} \brf_N(\lambda)\]
 where $\fkD_i = \fkD_i^{\sprf_N(\alpha)}$ and
 $\fkR_n$ is the operator defined in the proof of Proposition~\ref{alpha-prop}.
 It follows from Lemma~\ref{fkR-lem} that the last expression is equal to $\bsprf_n(u_n(\alpha) \circ \lambda) = \bsprf_n(\beta)$ as needed.

When $\phi$ is an arbitrary flag, 
 $\bsprf_n(z,\phi)$ is a direct sum of $\gl_n$-crystals of the form $\brf_n(w,\phi)$, so we can inductively apply Theorem~\ref{flag-thm} to 
 extend the $\q_n$-crystal isomorphism $\bsprf_{n}(T) \iso \bsprf_n( u_n(\alpha) \circ \lambda(\alpha))$ shown in the previous paragraph to 
$\bsprf_{n}(T,\phi) \iso \bsprf_{n}(\beta)$.
\end{proof}

\begin{proposition}\label{sp-equiv-prop}
Conjectures~\ref{sp-demazure-conj} and~\ref{sp-demazure-conj2} are equivalent.
\end{proposition}

\begin{proof}
If Conjecture~\ref{sp-demazure-conj2} holds then  
Proposition~\ref{sp-alpha-prop} tells us that every full subcrystal of $\bsprf_n(z,\phi)$ is a Demazure $\q_n$-crystal, for any flag $\phi$.
Instead suppose Conjecture~\ref{sp-demazure-conj} holds. Taking $\phi=\phi^S$ shows that each full subcrystal of $\bsprf_n(z)$ is isomorphic to $\bsprf_n(\alpha)$ for some skew-symmetric $\alpha \in \NN^n$.
This implies Conjecture~\ref{sp-demazure-conj2} by Corollary~\ref{sp-dem-sub-cor}.
\end{proof}

\section{Demazure crystals for \texorpdfstring{$Q$}{Q}-key polynomials}
\label{sec:Q_crystals}

There are parallel   ``orthogonal'' versions of most results in the previous section, concerning 
analogues of $\RF(w)$
 associated to all involutions in the symmetric group. These constructions, discussed below, will
lead to another kind of ``shifted'' Demazure crystals, whose characters are $Q$-key polynomials.
Our main results here are Theorem~\ref{o-crystal-thm} and Conjectures~\ref{o-demazure-conj} and~\ref{o-demazure-conj2}.

\subsection{Extended queer crystals}\label{qq-sect}

The results in this section involve  \defn{extended queer crystals} of type $\qq_n$ from~\cite{MT2021}. This is an extension of type $\q_n$ from Section~\ref{q-sect}, and is defined via Section~\ref{crystal-sect} by the following index set, standard crystal, and tensor product.

First, the index set for type $\qq_n$ is the disjoint union of four sets of symbols $I \sqcup \overline{I} \sqcup \underline{I} \sqcup I'$, where $I = [n-1]$, $\overline{I} := \{\overline{1}, \dotsc, \overline{n-1}\}$ and $\underline{I} := \{\underline{1}, \dotsc, \underline{n-1}\}$ are exactly as for $\q_n$-crystals, and $I' := \{1', 2', \dotsc, n'\}$.
\ytableausetup{boxsize=1.2em}%
The \defn{standard $\qq_n$-crystal} is the set $\BB= \left\{ \ytableaushort{{1'}}, \ytableaushort{1}, \ytableaushort{{2'}}, \ytableaushort{2}, \dotsc, \ytableaushort{{n'}}, \ytableaushort{n} \right\}$ with weight function $\weight(\ytableaushort{i}) = \weight(\ytableaushort{{i'}}) := \e_i$, crystal graph
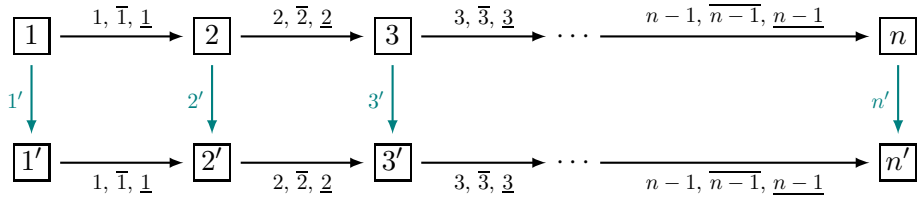
\[
     \begin{tikzpicture}[xscale=2.4, yscale=1.2,>=latex,baseline=(z.base)]
      \node at (0,0.7) (z) {};
      \node at (0,0) (T0) {$\ytableaushort{{1'}}$};
      \node at (1,0) (T1) {$\ytableaushort{{2'}}$};
      \node at (2,0) (T2) {$\ytableaushort{{3'}}$};
      \node at (3,0) (T3) {${\cdots}$};
      \node at (4.8,0) (T4) {$\ytableaushort{{n'}}$};
      \node at (0,1.4) (U0) {$\ytableaushort{1}$};
      \node at (1,1.4) (U1) {$\ytableaushort{2}$};
      \node at (2,1.4) (U2) {$\ytableaushort{3}$};
      \node at (3,1.4) (U3) {${\cdots}$};
      \node at (4.8,1.4) (U4) {$\ytableaushort{n}$};
      \draw[->,thick]  (T0) -- (T1) node[midway,below,scale=0.75] {$ 1$, $\overline 1$, $\underline 1$};
      \draw[->,thick]  (T1) -- (T2) node[midway,below,scale=0.75] {$ 2$, $\overline 2$, $\underline 2$};
      \draw[->,thick]  (T2) -- (T3) node[midway,below,scale=0.75] {$ 3$, $\overline 3$, $\underline 3$};
      \draw[->,thick]  (T3) -- (T4) node[midway,below,scale=0.75] {${n-1}$, $\overline{n-1}$, $\underline{n-1}$};
      \draw[->,thick]  (U0) -- (U1) node[midway,above,scale=0.75] {$ 1$, $\overline 1$, $\underline 1$};
      \draw[->,thick]  (U1) -- (U2) node[midway,above,scale=0.75] {$ 2$, $\overline 2$, $\underline 2$};
      \draw[->,thick]  (U2) -- (U3) node[midway,above,scale=0.75] {$ 3$, $\overline 3$, $\underline 3$};
      \draw[->,thick]  (U3) -- (U4) node[midway,above,scale=0.75] {${n-1}$, $\overline{n-1}$, $\underline{n-1}$};
      \draw[->,thick,color=teal]  (U0) -- (T0) node[midway,left,scale=0.75] {$1'$};
      \draw[->,thick,color=teal]  (U1) -- (T1) node[midway,left,scale=0.75] {$2'$};
      \draw[->,thick,color=teal]  (U2) -- (T2) node[midway,left,scale=0.75] {$3'$};
      \draw[->,thick,color=teal]  (U4) -- (T4) node[midway,left,scale=0.75] {$n'$};
     \end{tikzpicture}
\] 
and statistics $\varepsilon_i, \varphi_i$ defined by~\eqref{string-eqs}.

Finally, we view the set of tensor powers $\BB^{\otimes m}$ for $m\geq 2$ as a $\qq_n$-crystal with the  following crystal operators.
For $i \in I$, the definitions of $e_i$ and $f_i$ on $\BB^{\otimes m}$ are identical to the $\gl_n$-case
and by~\eqref{eq:gl_tensor_product}.
Next, following~\cite[Def.~3.10]{MT2021}, for $ b\in \BB$ and $c \in \BB^{\otimes (m-1)}$ we set
\begin{subequations}
\begin{align}
  e_{\overline 1}(b\otimes c) & :=    \begin{cases} 
 b \otimes e_{\overline 1}(c)
 &\text{if }\weight(b)_1 = \weight(b)_2 = 0,
 \\
 f_{1'}  e_{\overline 1} (b) \otimes e_{1'}(c) 
 &\text{if $\weight(b)_1 = 0$, $f_{1'} e_{\overline 1}(b) \neq \zero$, and $e_{1'}(c)\neq \zero$,}
\\
 e_{1'} e_{\overline 1} (b) \otimes f_{1'}(c) 
 &\text{if $\weight(b)_1 = 0$, $e_{1'} e_{\overline 1}(b)\neq \zero$, and $f_{1'}(c) \neq \zero$,}
 \\
  e_{\overline 1}(b) \otimes c
&\text{otherwise,}
 \end{cases} 
 \allowdisplaybreaks \\
 f_{\overline 1}(b\otimes c) & :=    \begin{cases} 
 b \otimes f_{\overline 1}(c)
 &\text{if }\weight(b)_1 = \weight(b)_2 = 0,
 \\
 f_{\overline 1} f_{1'} (b) \otimes e_{1'}(c) 
 &\text{if $\weight(b)_1 = 1$, $f_{\overline 1} f_{1'}(b) \neq \zero$, and $e_{1'}(c) \neq  \zero$,}
\\
 f_{\overline 1} e_{1'} (b) \otimes f_{1'}(c) 
 &\text{if $\weight(b)_1 = 1$, $ f_{\overline 1} e_{1'} (b)\neq \zero$, and $f_{1'}(c) \neq \zero$,}
 \\
  f_{\overline 1}(b) \otimes c
&\text{otherwise.}
 \end{cases} 
\end{align}
\end{subequations}
Then we set 
\be\label{qq-overline-eq}
 e_{\overline{\imath}} := \sigma_{i-1} \sigma_i e_{\overline{\imath-1}} \sigma_i \sigma_{i-1} 
 \quand  f_{\overline{\imath}} := \sigma_{i-1} \sigma_i f_{\overline{\imath-1}} \sigma_i \sigma_{i-1}\ee for $i \in \{2,3,\dotsc,n-1\}$, and define 
\be\label{qq-underline-eq}
e_{\underline{i}} := \sigma_{w_0} f_{\overline{n-\imath}} \sigma_{w_0}
\quand
f_{\underline{i}} := \sigma_{w_0} e_{\overline{n-\imath}} \sigma_{w_0}
\ee
for $i \in [n-1]$, exactly as in the $\q_n$-case.
 Finally, the operators $e_{1'}$ and $f_{1'}$ are given by
 \be
   e_{1'}(b\otimes c)  :=   \begin{cases} 
     e_{1'}(b) \otimes c&\text{if }\weight(b)_1 \neq 0,
 \\
b \otimes e_{1'}(c) & \text{if }\weight(b)_1 = 0,
 \end{cases}
 \quand
    f_{1'}(b\otimes c)  :=   \begin{cases} 
    f_{1'}(b) \otimes c&\text{if }\weight(b)_1 \neq 0,
 \\
b \otimes f_{1'}(c) & \text{if }\weight(b)_1 = 0,
 \end{cases}
\ee
for $ b\in \BB$ and $c \in \BB^{\otimes (m-1)}$, and for $i \in \{2,3,\dotsc,n\}$ we set
\be\label{e-i-prime-eq}
e_{i'} := \sigma_{i-1} \dotsm \sigma_2 \sigma_1 e_{1'} \sigma_1 \sigma_2 \dotsm \sigma_{i-1}
\quand
f_{i'} := \sigma_{i-1} \dotsm \sigma_2 \sigma_1 f_{1'} \sigma_1 \sigma_2 \dotsm \sigma_{i-1}
\ee
following~\cite[Eq.~(3.6)]{MT2021}.

This data gives rise to categories of \defn{$\qq_n$-crystals} and \defn{normal $\qq_n$-crystals} that are closed under tensor products~\cite[Thm.~3.12]{MT2021}. 
The crystal operators on any normal $\qq_n$-crystal are completely determined by just the operators $e_i$ and $f_i$
indexed by $i \in \{1',\overline{1}, 1,2,\dotsc,n-1\}$. When drawing the graphs of normal crystals, we often include only the arrows of these indices
(see, for example, Figure~\ref{fig:O_crystal}).

As in the $\q_n$-case, each connected normal $\qq_n$-crystal has a unique highest weight element whose weight $\lambda$ uniquely determines its isomorphism class, and these weights $\lambda$ range over all strict partitions in $\NN^n$~\cite{MT2021}.
For each strict partition $\lambda \in \NN^n$, we can choose a connected normal $\qq_n$-crystal  $\cB(\lambda)$ with highest weight element $u_{\lambda}$ of weight $\lambda$. 
This crystal has a unique lowest weight element of weight $w_0 \lambda = (\lambda_n,\dotsc,\lambda_2,\lambda_1)$,
and its character is the \defn{Schur $Q$-polynomial}
$
Q_{\lambda}(x_1, \dotsc, x_n) 
$~\cite[Thm.~1.5]{MT2021}.
Unlike the $\q_n$-case, however, the lowest weight element of $\cB(\lambda)$ is not given by $\sigma_{w_0} u_{\lambda}$ 
(see~\cite[Prop.~7.15]{MT2021})
and $\cB(\lambda)$ is not naturally the crystal basis of a polynomial $U_q(\q_n)$-module.

\subsection{Orthogonal reduced factorizations}

Recall that $I_n \subset I_\infty\subset I_\ZZ$ are the respective sets of involutions in $S_n \subset S_\infty \subset S_\ZZ$.
Fix $z \in I_\ZZ$. Again let $\cA^\O(z)$ be the set of minimal-length elements $w \in S_\ZZ$ with  $z = w^{-1}\circ  w$,
and let  $\Cyc(z) $ be the finite set of pairs $(a,b) \in \ZZ\times \ZZ$ with $a<b = z(a)$.
  Then define
\be\label{cRo-eq}  \cR^\O(z) :=\bigsqcup_{A\subseteq \Cyc(z)} \bigsqcup_{w \in \cA^\O(z)} \cR^+(w,A),\ee
with $\cR^+(w,A)$ as in Section~\ref{rf-sect}. The elements of $\cR^\O(z)$ are certain primed words, all of the same length.
 The next lemma shows that $\cR^\O(z)$ is the same as the set of \defn{primed involution words} for $z \in I_\ZZ$ considered in~\cite{MT2021,Marberg2021b,Marberg2021a}. Recall that we set $s_{i'} = s_{i} = (i \: i+1) \in S_\ZZ$ for $i \in \ZZ$.
 
  \begin{lemma}\label{primed-w-lem}
 Let $z \in I_\ZZ$, $w \in \cA^\O(z)$, and $i=i_1\cdots i_\ell \in \cR^+(w)$.
 For $j \in [\ell]$ let $w_j = s_{i_1}\cdots s_{i_{j-1}}$.
 Then $ \Marked(i) \subseteq \Cyc(z)$ if and only if $s_{i_j}$ commutes with $(w_j)^{-1} \circ w_j$ whenever $i_j \in \ZZ'$.
 \end{lemma}
 
 \begin{proof}
 For each $j \in [\ell]$ define $a_j$ and $b_j$ by the formula~\eqref{abj-eq}
 and set
 $\gamma_j := (a_j,b_j)$.
We say that $j \in [\ell]$ is a \defn{commutation} in the word $i$ if $s_{i_j}$ commutes with $(w_j)^{-1} \circ w_j$.
We have already noted the general property that $ j \mapsto \gamma_j$ is a bijection $[\ell] \to \Inv(w)$.
Because $w \in \cA^\O(z)$,~\cite[Prop.~4.4]{Marberg2021b} asserts that this map restricts to a bijection from the set of commutations in $i$ to $\Cyc(z)$.
 \end{proof}
 
 Like fpf-involution words, primed involution words correspond to maximal chains in a weak order on the orbits of the orthogonal group acting on the complete flag variety; see~\cite{Brion2001,RichSpring}.
 
 \begin{remark}\label{inv-obs}
 We can inductively construct the set $\cR^\O(z)$ as follows.
Consider a primed word $i=i_1i_2\cdots i_\ell$ where each $i_j \in \ZZ\sqcup \ZZ'$.
 Let $z_1 := 1 \in I_\ZZ$ and for $j \in [\ell]$  let 
 \[ z_{j+1} = \begin{cases}  
 s_{i_j}  z_j =  z_j  s_{i_j} &\text{if  $z_j$ preserves $\{ \lceil i_j\rceil , 1+\lceil i_{j}\rceil \}, $}\\
 s_{i_j}  z_j  s_{i_j} &\text{otherwise}.\end{cases}
 \]
 By~\cite[\S2.1]{Marberg2021a}, we have $i \in \cR^\O(z)$ if and only if both of the following conditions hold:
 \bei
 \item for all $j \in [\ell]$ it holds that $z_j(\lceil i_j\rceil ) < z_j(1+\lceil i_j\rceil )$, and
\item  if the letter $i_j $ is primed then $\lceil i_j\rceil$ and $1 + \lceil i_j\rceil$ are both fixed
by $z_j$.
\eei
For efficient procedures to construct the set $\cA^\O(z)$, see~\cite{CJW,HMP2}.
  It follows when $i \in \cR^\O(z)$ that  
  \begin{itemize}
  \item[(a)] one has $\lceil i_1 \rceil \neq \lceil i_3 \rceil$ and if $i_2 \in \ZZ'$ then $\lceil i_1\rceil \neq \lceil i_2 \rceil \pm 1$;
  \item[(b)] if $i_1 \in \ZZ$ then $i_1'i_2i_3\cdots i_\ell$ and $i_2i_1i_3\cdots i_\ell$ are also in $\cR^\O(z)$;
  \item[(c)] if $\lceil i_{j-1} \rceil = \lceil i_{j+1}\rceil$ then $i_j=\lceil i_{j-1} \rceil\pm1 \in \ZZ $ and  $i_{j-1}$ or $i_{j+1}$ is unprimed~\cite[Prop.~8.2]{Marberg2021b}.
  \end{itemize}
  The set $\cR^\O(z)$ is spanned and preserved by the ``primed'' braid relations that interchange
  \be\cdots ab\cdots  \leftrightarrow \cdots ba\cdots \ee if $a,b \in \ZZ\sqcup\ZZ'$ have $|\lceil a\rceil - \lceil b \rceil| > 1$,
  as well as 
  \be\cdots aba \cdots \leftrightarrow \cdots bab \cdots \quand \cdots a'ba \cdots \leftrightarrow \cdots bab' \cdots \ee for $a,b \in \ZZ$
  with $\abs{a-b} =1$, and finally with 
  \be a\cdots \leftrightarrow a'\cdots\quand ab\cdots \leftrightarrow ba \cdots\ee for any $a,b \in \ZZ$ (the last two relations only affect letters at the start of the word)~\cite[Cor.~8.3]{Marberg2021b}.
  \end{remark}

\begin{example}
If $z=(1 \: 4)(2 \: 3) \in I_4$ then $\cR^\O(z)$ has $32$ elements, consisting of all ways of optionally adding primes
to the underlined letters in these words:
\[
\underline3\hs\underline123, \quad \underline1\hs\underline323, \quad \underline123\underline2, \quad \underline231\underline2, \quad \underline213\underline2, \quad \underline1\hs\underline321, \quad \underline321\underline2, \quand \underline3\hs\underline1 21.
\]
\end{example}

 Define $\ORF(z)$ to be the set of tuples of decreasing primed words $a=(a^1,a^2,\ldots)$ 
with $a^i =\emptyset$ for all $i>n$ and with $a^1a^2\cdots \in \cR^\O(z)$.
Also let  $\upORF(z)  = \{ \unprime(a) \mid a \in \ORF(z)\}$.
Then
\[ 
\upORF(z) =  \bigsqcup_{w \in \cA^\O(z)} \RF(w)
\quand
\orf_n(z) =\bigsqcup_{A\subseteq \Cyc(z)}  \bigsqcup_{w \in \cA^\O(z)} \RF^+(w,A).
\]
We refer to the elements of these sets as \defn{orthogonal reduced factorizations}.
 
 These sets can be empty if $n$ is too small; see Theorem~\ref{upORF-thm}
 for a precise characterization.
When nonempty, $\upORF(z)$ and $\ORF(z) $ are both normal $\gl_n$-crystals by Theorem~\ref{up-thm}.
Our goal in this section is to explain the results in~\cite{MT2021,Marberg2019b} that extend these structures to $\q_n$- and $\qq_n$-crystals.

\begin{definition}\label{ef-1-def}
For $a=(a^1,a^2,\ldots) \in \ORF(z)$, define $e_{\overline 1}a$ and $f_{\overline 1}a$ as follows:
\begin{enumerate}
\item[\defn{$e_{\overline{1}}$}:]
Set $e_{\overline{1}}a := \zero$ if 
$a^2=\emptyset$ is empty or $\lceil \max(a^2) \rceil \leq \lceil \max(a^1)\rceil $.
 
Otherwise, form $b = (b^1,b^2,\ldots)$ from $a$ by moving $\max(a^2)$ to the start of $a^1$. Then:
\begin{itemize}
\item if $a^1=\emptyset$ then set $e_{\overline{1}}a := b$;

\item if $a^1\neq \emptyset $ then form $e_{\overline{1}}a$ from $b$ by toggling the primes on the first two letters of $b^1$.
\end{itemize}

\item[\defn{$f_{\overline{1}}$}:]
Set $f_{\overline{1}}a := \zero$ if $a^1=\emptyset$ is empty or $\lceil \max(a^1)\rceil  \leq \lceil \max(a^2) \rceil $.
 
Otherwise, form $b = (b^1,b^2,\ldots)$ from $a$ by moving $\max(a^1)$ to the start of $a^2$. Then:
\begin{itemize}
\item if $b^1=\emptyset$ then $f_{\overline{1}}a := b$;

\item if $b^1\neq \emptyset$ then form $f_{\overline{1}}a$ from $b$ by toggling the primes on $\max(b^1)$ and $\max(b^2)$.
\end{itemize}
\end{enumerate}
\end{definition}


\begin{example}
The $f_{\overline 1}$ operator sends
$
(5'32,41)  \xmapsto{\ f_{\overline 1}\ }  
(3'2,541)
$
and
$
(5'3'2,41)  \xmapsto{\ f_{\overline 1}\ }
 (3'2,5'41)
$
with $e_{\overline 1}$ acting the in the reverse direction.
On the other hand, $e_{\overline 1}(5'32,41) = f_{\overline 1}(3'2,5'41) = \zero$.
\end{example}

The next two theorems are equivalent to results in~\cite{Marberg2019b,MT2021}
about crystal structures on \emph{increasing} factorizations of 
$\cR^\O(z)$.
The way to translate these results to our current setup is again via 
 the $\ast$ operation from Remark~\ref{rem:negation}, 
 which acts on unprimed words by negating every letter.
 
By Remark~\ref{inv-obs},   $\ast$ 
defines a bijection $\cR^\O(z) \to \cR^\O(z^\ast)$, which preserves the number of primed letters in a given word.
On tuples of words, the $\ast$ operation converts the $\q_n$-crystal
  in~\cite[\S3.2]{Marberg2019b}
 into what is described in the following theorem.
Here, we fix $z \in I_\ZZ$ and let 
\[
c_i^\O(z) := \abs{ \{j \in \ZZ \mid i < j \text{ and }\min\{i,z(i)\}\geq z(j)\} }.
\]
These numbers make up the \defn{involution code} of $z$ in~\cite{HMP1,HMP4}.

\begin{theorem}
\label{upORF-thm}
The sets $\upORF(z)$ and $\ORF(z)$ are each nonempty if and only if $c_i^\O(z) \leq n$ for all $i \in \ZZ$. When this holds,
$\upORF(z)$ has a normal $\q_n$-crystal structure 
with crystal operators $e_{i}$ and $f_{i}$
for $i \in \{1,2,\dotsc,n-1\} \sqcup \{\overline 1\}$ given
  as in Definitions~\ref{ef-def} and~\ref{ef-1-def}.
\end{theorem}

\begin{proof}
On adjusting references, the proof is the same as for Theorem~\ref{sprf-thm}.
When nonempty, $\upORF(z)$ is a normal $\q_n$-crystal by~\cite[Cor.~3.33]{Marberg2019b}
via the preceding discussion. 
By~\cite[Rem.~3.10]{Marberg2019b}, we have $\ch(\ORF(z^\ast)) \neq 0$  if and only if $c^\O_i(z) \leq n$ for all $i$.
Comparing~\cite[Rem.~3.10]{Marberg2019b} with~\cite[Cor.~5.10]{Marberg2019a} 
shows that $\ch( \ORF(z^\ast))=\ch( \ORF(z))$, so $\ORF(z) \neq \emptyset$ if and only if $c^\O_i(z) \leq n$ for all $i$.
\end{proof}

The character of $\ORF(z)$ is the \defn{involution Stanley symmetric polynomial} $\widehat F_z(x_1,x_2,\dotsc,x_n)$ studied in~\cite{HMP4}.
This $\q_n$-crystal extends to a $\qq_n$-crystal structure on $\ORF(z)$:

\begin{definition}\label{e0-f0-def}
For $a=(a^1,a^2,\ldots) \in \ORF(z)$, define $e_{1'}a$ and $f_{1'}a$ as follows:
\begin{itemize}
\item[\defn{$e_{{1'}}$}:] Set $e_{1'}a:= \zero$ if $a^1=\emptyset$ is empty or $\max(a^1)\in  \ZZ$ is unprimed.

Otherwise, form $e_{1'}a$ from $a$ by removing the prime on $\max(a^1) \in \ZZ'$.

\item[\defn{$f_{{1'}}$}:] Set $f_{1'}a:= \zero$ if $a^1=\emptyset$ is empty or $\max(a^1)\in  \ZZ'$ is primed.

Otherwise, form $f_{1'}a$ from $a$ by adding a prime to $\max(a^1) \in \ZZ$.

\end{itemize}
\end{definition}

\begin{example}
We have $e_{1'}(5'3'2,41) = (53'2,41)$ and $f_{1'}(5'3'2,41) = \zero$. 
\end{example}

The following statement is equivalent to~\cite[Cor.~7.18]{MT2021} via the $\ast$ operation from Remark~\ref{rem:negation}.

 \begin{theorem}[\cite{MT2021}]
 \label{orf-thm}
When nonempty, $\ORF(z)$ has a normal $\qq_n$-crystal structure with crystal operators $e_{i}$ and $f_{i}$
for $i \in \{1,2,\dotsc,n-1\} \sqcup \{\overline 1\} \sqcup \{1'\} $ given
  as in Definitions~\ref{ef-def},~\ref{ef-1-def}, and~\ref{e0-f0-def}.
\end{theorem}

\begin{figure}[h]
\begin{center}
\input{qq3-decr-crystal.tex}
\end{center}
\caption{The $\qq_3$-crystal graph of $\orf_3(z)$ for the dominant involution  $z=(1 \: 3)(2 \: 4)$ of shape $\lambda=(2,2)$.
The boxed elements make up the set of bounded factorizations $\borf_3(z)$. Solid blue, solid red, dotted green, and dashed blue arrows are $i$-edges for $i = 1$, $2$, $1'$, and $\overline{1}$, respectively.
We have not drawn the $2'$, $\overline 2$-, $\underline 1$-, or $\underline 2$-edges,
as these are determined by the displayed arrows.
}
\label{fig:O_crystal}
\end{figure}

Figure~\ref{fig:O_crystal} shows an example of 
 $\ORF(z)$.
The character of $\ORF(z)$, which is $2^{|\Cyc(z)|}$ times the character of $\upORF(z)$,
coincides with the polynomial denoted
$Q_z(x_1,x_2,\dotsc,x_n)$ in~\cite[\S4.5]{HMP4}.

Below,
recall that we define $e_i$ and $f_i$ for $i \in \overline I \sqcup\underline I \sqcup I' $
by the formulas \eqref{qq-overline-eq}, \eqref{qq-underline-eq}, and \eqref{e-i-prime-eq},
using Definitions~\ref{ef-def},~\ref{ef-1-def}, and~\ref{e0-f0-def} for the relevant base cases.

 \begin{proposition}\label{orf-prop}
Choose an index $i \in I\sqcup \overline I\sqcup \underline I \sqcup I'$. 
Then the following properties hold:
\ben
\item[(a)] If $i \in I\sqcup \overline I\sqcup \underline I$ then $e_i$ and $f_i$ commute with the map $\unprime \colon \ORF(z) \to \upORF(z)$.
\item[(b)] If $i \in I'$, $a \in \ORF(z)$, and $e_i a \neq \zero$, then $\unprime(e_i a)= \unprime(a)$.
\item[(c)] If $i \in I'$, $a \in \ORF(z)$, and $f_i a \neq \zero$, then $\unprime(f_i a) = \unprime(a)$.
\een
\end{proposition}
 
\begin{proof}
 As all inequalities in Definitions~\ref{pair-def},~\ref{ef-def}, and~\ref{ef-1-def} are stated using the ceiling function that removes all primes,
we can see by inspecting the crystal operator definitions that $e_i$ and $f_i$ commute with the map $\unprime$ for all $i \in \{\overline{1}, 1,2,\dotsc,n-1\}$. 
Since $\unprime$ is weight-preserving, it follows from~\eqref{sigma-def} that each $\sigma_i$ for $i \in [n-1]$ also commutes with $\unprime$, so part~(a) follows.

For part~(b), suppose $a \in \ORF(z)$. If $e_{1'} a \neq \zero$ then clearly $\unprime(e_{1'} a) = \unprime(a)$.
If $1' \neq i' \in I'$ and $e_{i'} a \neq \zero$ then, since $e_{i'} = \sigma_i e_{i' - 1} \sigma_i$
we get by induction that
\[
\unprime(e_ia) = \sigma_i (\unprime( e_{i'-1}\sigma_i a)) = \sigma_i \unprime( \sigma_i a) = \sigma_i^2 \unprime(a) = \unprime(a).
\]
This proves part~(b).
Part~(c) follows similarly.
\end{proof}

Our next result is an ``orthogonal'' analogue of Theorem~\ref{reduced-tab-lem} for the $\qq_n$-crystal $\ORF(z)$.
Define \defn{orthogonal Coxeter--Knuth equivalence} to be the transitive closure $\simOCK$
of primed Coxeter--Knuth equivalence $\simCK$ and the symmetric relation on primed words 
with
\be\label{ock1-eq}
a\cdots \simOCK a'\cdots, 
\quad
ab\cdots \simOCK ba\cdots,
\quad 
ab'\cdots \simOCK ba'\cdots,
\quad
a'b\cdots \simOCK b'a\cdots,
\quad
a'b'\cdots \simOCK b'a'\cdots
 \ee
 for all $a,b \in \ZZ$.
 These extra relations can only change the letters at the start of a word;
note that the first three relations imply the last two.
If $z \in I_\ZZ$ then Remark~\ref{inv-obs} implies that
$\cR^\O(z)$ is a disjoint union of orthogonal Coxeter--Knuth equivalence classes.

 The \defn{main diagonal} of a shifted tableau consists of the positions $(i,j)$ with $i=j$.
A shifted tableau $T$ with $\row(T) \in \cR^\O(z)$ is increasing
(equivalently, decreasing) if and only if the tableau $\unprime(T)$
formed by removing the primes from all entries is also increasing (equivalently, decreasing)~\cite[Prop.~2.6]{Marberg2019a}.
This means that if such a tableau is increasing or decreasing then it cannot contain 
both $m'$ and $m$ (for any $m \in \ZZ$) in the same row or column.

\begin{theorem} \label{o-reduced-tab-lem}
Fix a $\simOCK$-equivalence class $\cK\subseteq \cR^\O(z)$ for $z \in I_\ZZ$.
Then $\cK$ 
contains the row reading words of a 
unique shifted tableau $U$ that is increasing with no primed entries on the main diagonal
and a
unique shifted tableau $V$ that is increasing with no primed entries on the main diagonal.
Moreover, these two shifted tableaux have the same shape.
 \end{theorem}

\begin{proof}
The structure of our argument is similar to the proof of Theorem~\ref{sp-reduced-tab-lem}.
Choose a primed word $i=i_1i_2\cdots i_k \in \cK$. 
The \defn{orthogonal Edelman--Greene insertion algorithm} from~\cite{Marberg2021a} gives an increasing shifted tableau 
$P^\O_\EG(i)$ with no diagonal primes and with $i \simOCK \row(P^\O_\EG(i))$ by~\cite[Prop.~3.21]{Marberg2021a}. 
This is the unique increasing shifted tableau with no diagonal primes and with row reading word in $\cK$,
since if $T$ is such a tableau then
\ben
\item[(a)] $\cK$ also contains the column reading word $\col(T)$ by~\cite[Lem.~2.7]{Marberg2021a};
\item[(b)] we have $T = P^\O_\EG(\col(T))=P^\O_\EG(i)$  by (a) and~\cite[Cors.~3.25 and 3.26]{Marberg2021a}.
\een
The desired increasing shifted tableau is therefore $U := P^\O_\EG(i)$.

The $\ast$ operation from Remark~\ref{rem:negation}
 interchanges the sets of increasing  and decreasing shifted tableaux 
with row reading words in $\cR^\O(z) \cup \cR^\O(z^\ast)$.
As $\ast$ also preserves $\simOCK$-equivalence, we conclude from the previous 
paragraph that 
the desired decreasing shifted tableau is $V := P^\O_\EG(i^\ast)^\ast$.

 Let $\mu$ and $\nu$ be the shapes of $U$ and $V$, which are also the shapes of $P^\O_\EG(i)$ and $P^\O_\EG(i^\ast)$. 
 Recall the definition of $F_{S,k}(\x)$ from the proof of Theorem~\ref{sp-reduced-tab-lem}.
 By~\cite[Thm.~7.10]{MT2021},
 the   subsets of $\ORF(z)$ and $\ORF(z^\ast)$
 consisting of all factorizations $a=(a^1,a^2,\dotsc,a^n)$
 with $P^\O_\EG(a^1 a^2 \cdots a^n) = P^\O_\EG(i)$
 and $P^\O_\EG(a^1 a^2 \cdots a^n) = P^\O_\EG(i^\ast)$, respectively,
 are full $\qq_n$-subcrystals.

The formal power series  $\sum_{i \in \cK} F_{\Des(i),k}(\x)$ and $\sum_{i \in \cK^\ast} F_{\Des(i),k}(\x)$ are the 
limits of the
characters of these subcrystals as $n\to\infty$ .
 By~\cite[Thm.~7.10]{MT2021} and the remarks after~\cite[Prop.~6.13]{MT2021}, 
 these power series  coincide with the Schur $Q$-functions $Q_\mu(\x)$ and $ Q_\nu(\x)$.

Since  $i=i_1i_2\cdots i_k\in \cK$, the word $\unprime(i)$ has no adjacent repeated letters,
so $\Des(i^\ast) = [k - 1] \setminus \Des(i)$.
As 
the automorphism $F_{S,k}(\x) \mapsto F_{[k-1]\setminus S,k}(\x)$
fixes all Schur $Q$-functions~\cite[\S{III}.8, Ex.~3(a)]{Macdonald},
it follows as in the proof of Theorem~\ref{sp-reduced-tab-lem}
that $Q_\nu(\x) = Q_\mu(\x)$ so $\nu=\mu$.
\end{proof}

We define a \defn{$\O$-reduced tableau} for $z \in I_\infty$
to be a shifted tableau $T$ with no primes on the main diagonal that is increasing or decreasing with $\row(T) \in \cR^\O(z)$.
Given an $\O$-reduced tableau $T$ for $w$, let
$\orf_n(T) \subseteq \orf_n(z)$ be the subset of factorizations $a=(a^1,a^2,\ldots)$ 
with
$a^1a^2\cdots \simOCK \row(T)$.

\begin{proposition}
\label{o-subcrystal-prop}
The map $T \mapsto \orf_n(T)$ is a bijection from 
increasing (equivalently, decreasing)
 $\O$-reduced tableaux for $z \in I_\ZZ$ with at most $n$ rows 
to full $\qq_n$-subcrystals of $\orf_n(z)$.
If $T$ is an $\O$-reduced tableau then $\orf_n(T)$ is empty if and only if $T$ has more than $n$ rows.
\end{proposition}

\begin{proof}
Our argument is similar to the proof of Proposition~\ref{sp-subcrystal-prop}.
Fix $z \in I_\ZZ$. Via the $\ast$ operation, which preserves  $\simOCK$-equivalence, 
the results 
\cite[Cor.~3.25]{Marberg2021a} and 
\cite[Thm.~7.10]{MT2021}
imply that 
each subset $\ORF(T)\subseteq \ORF(z)$ is a full subcrystal or empty. By Theorem~\ref{o-reduced-tab-lem} every full subcrystal arises as $\ORF(T)$ for some 
 $\O$-reduced tableau $T$.

 If there exists a decreasing $\O$-reduced tableau $T$ with at most $n$ rows,
then $\ORF(T)$ is nonempty since it contains the sequence $a=(a^1,a^2,\dotsc,a^n)$,
in which $a^i$ is row $n + 1 - i$ of $ T$.
It remains to check that $\ORF(T)$ is empty if $T$ has more than $n$ rows.
For this, suppose $a \in \ORF(T)$. 
Then the \defn{orthogonal EG correspondence} described in~\cite[\S3.2]{Marberg2021a} maps $a^\ast$ to a pair of shifted tableaux $(P,Q)$ of the same shape,
where $P = P^\O_\EG\bigl( (a^1a^2\cdots a^n)^\ast \bigr)$ and where $Q$ is semistandard with all entries in $\{1'<1<\dots<n'<n\}$.
From the proof of Theorem~\ref{o-reduced-tab-lem}, we see that $P$ and $Q$ have the same shape as $T$.
Yet there are no semistandard shifted tableaux $Q$ with more than $n$ rows and all entries in $\{1'<1<\cdots<n'<n\}$, so $\ORF(T)$ must be empty if $T$ has more than $n$ rows.
\end{proof}

As in the $\Sp$-case, decreasing $\O$-reduced tableaux correspond to lowest weight elements.

\begin{proposition}\label{o-lowest-prop}
 If $T$ is a decreasing $\O$-reduced tableau for $z \in I_\ZZ$ with at most $n$ rows, 
 then the unique $\qq_n$-lowest weight element of $\orf_n(T)$ is the sequence $a=(a^1,a^2,\dotsc,a^n)$,
where $a^i$ is formed from row $n + 1 - i$ of $ T$ by adding a prime to its first entry.
\end{proposition}

\begin{proof}
Suppose $T$ has shape $\mu$ and recall that $T = P^{\O}_{\EG}\bigl( (a^1a^2\cdots a^n)^\ast \bigr)^\ast$.
Composing $\ast $ with the map in~\cite[Thm.~7.10]{MT2021}
gives an isomorphism from $\orf_n(T)$ to a connected normal $\qq_n$-crystal of semistandard shifted tableaux of shape $\mu$. The latter crystal has exactly $2^{\ell(\mu)}$ elements of weight $(\mu_n,\dotsc,\mu_2,\mu_1)$,
which are all $\q_n$-lowest weight by~\cite[Lem.~6.11 and Thm.~6.14]{MT2021}.
These elements include the crystal's unique $\qq_n$-lowest weight element by~\cite[Thm.~6.20]{MT2021}.

Since the indicated factorization $a$ has  $\weight(a)=(\mu_n,\dotsc,\mu_2,\mu_1)$,
it must be  $\q_n$-lowest weight, and to show that it is  $\qq_n$-lowest weight 
it suffices to check that $f_i a=\zero$ for all $i \in I'$.
We have $f_{1'} a =\zero$ since the first factor of $a$ is either empty or starts with a primed letter.

To show that $f_i a=\zero$ for the other indices $i \in I'$, we will leverage the following general observation. Suppose $b \in \rf^+_n(w)$ is a decreasing factorization of 
a permutation $w \in S_\ZZ$ and $f_i b = \zero$ for some $i \in [n-2]$.
Then we also have $f_i e_{i+1} b = \zero$, since if $e_{i+1} b = (c^1,c^2,\dotsc,c^n) \neq \zero$ then the set of letters in 
$\unprime(c^{i+1})$ is a superset of $\unprime(b^{i+1})$, so the fact that there are no unpaired letters in $b^i$ for the $(b^i,b^{i+1})$-pairing
implies that there are also no unpaired letters in $c^i=b^i$ for the $(c^i,c^{i+1})$-pairing.

Now choose $i \in [n-1]$ and let $\widetilde a := \sigma_1\sigma_2\cdots \sigma_{n-i}(a)$.
We need to show that $f_{(n+1-i)'}a = \zero$, or equivalently that $f_{1'}\widetilde a=\zero$.
Since $a$ is $\gl_n$-lowest weight, the previous paragraph implies that $\widetilde a =e^{m_1}_1e_2^{m_2}\cdots e^{m_{n-i}}_{n-i}(a)$ for some nonnegative integers $m_j\geq 0$.
We have $m_1=m_2=\cdots=m_{n-i}=0$ if $\ell(\mu)<i<n$ since the first $n-\ell(\mu)$ factors of $a$ are empty.
In this case the first factor of $\widetilde a$ is also empty so  $f_{1'}\widetilde a = \zero$.

If $1 \leq i \leq \ell(\mu)$, then one can check using Definition~\ref{ef-def} (and the fact that the diagonal entries of $T$ are unprimed, strictly decreasing, and each differ by at least two) that $m_1,m_2,\dotsc,m_{n-i}>0$ are all positive and the first factor of $\widetilde a$ is nonempty with first letter given by the primed number $\max(a^{n+1-i}) \in \ZZ'$.
In this case we again have $f_{1'}\widetilde a = \zero$  as needed.
\end{proof}

\begin{corollary}\label{o-highest-cor}
If $T$ is a $\O$-reduced tableau with at most $n$ rows, then the unique highest weight of the connected normal $\qq_n$-crystal $\orf_n(T)$ is the partition shape of $T$.
\end{corollary}

\begin{proof}
By Theorem~\ref{o-reduced-tab-lem} we may assume that $T$ is decreasing. 
If this holds and $T$ has shape $\mu \in \NN^n$, then Proposition~\ref{o-lowest-prop} shows that the unique lowest weight of  $\orf_n(T)$ is $(\mu_n,\dotsc,\mu_2,\mu_1)$.
Reversing this tuple gives the unique highest weight by~\cite[Prop.~7.15]{MT2021}.
\end{proof}

\begin{example}
The unique increasing and decreasing $\O$-reduced tableaux for  $z=(1\: 3)(2\: 4) \in I_4$ are $\ytabsmall{1 & 2 \\ \none & 3}$ and $\ytabsmall{3 & 2 \\ \none & 1}$.
Compare these with Figure~\ref{fig:O_crystal} and Proposition~\ref{o-lowest-prop}.
\end{example}

Also just like the $\Sp$-case, the increasing $\O$-reduced tableaux indexing the full $\qq_n$-subcrystals of $\orf_n(z)$ do not identify the $\qq_n$-crystal's highest weight elements, leaving this open problem:

\begin{problem}
Describe the highest weight elements of the $\qq_n$-crystals $\orf_n(z)$ for $z \in I_\ZZ$.
\end{problem}

If $\lambda$ is a partition then let $\half(\lambda)$ be the strict partition
whose nonzero parts are $\{ \lambda_i - (i-1) \mid i \in \PP\} \cap \PP$.
For example, if
\[ \lambda = (4,3,3,1)= \ytabsmall{ \ & \ & \ & \ \\ \ & \ & \ \\ \ & \ & \ \\ \ }
\quad\text{then}\quad  \half(\lambda)=(4,2,1).\] 
The map $\lambda \mapsto \half(\lambda)$ is a bijection from symmetric partitions to strict partitions.

\begin{proposition}\label{o-dom-prop}
Suppose $\lambda$ is a symmetric partition and $\mu = \half(\lambda)$.
Let $z = w_\lambda \in I_\infty$ and define $T^\O_\lambda$ to be the shifted tableau of shape $\mu$ 
with entry $i+j-1$ in each $(i,j) \in \SD_\mu$.
Then $\cR^\O(z)$ is a single $\simOCK$-equivalence class and $T^\O_\lambda$
is the unique increasing $\O$-reduced tableau for $z$.
\end{proposition}

\begin{proof}
We saw in the proof of Theorem~\ref{o-reduced-tab-lem}
that 
the increasing $\O$-reduced tableaux for an arbitrary $z\in I_\ZZ$ are the outputs $P^\O_\EG(i)$ 
of the orthogonal Edelman--Greene insertion algorithm applied to $i \in \cR^\O(z)$.
Moreover, the number of distinct outputs is the number of $\simOCK$-equivalence classes in $ \cR^\O(z)$.
By~\cite[Lem.~3.17]{Marberg2021a}, however, we have $P^\O_\EG(i)=T^\O_\lambda$
for every $i\in \cR^\O(z)$ when $z=w_\lambda$.
\end{proof}

If $\lambda$ is a symmetric partition then define $\ORF(\lambda) := \ORF(w_\lambda)$,
where $w_\lambda \in I_\infty$ is the unique dominant element of $S_\infty$ with shape $\lambda$.
By Theorem~\ref{o-reduced-tab-lem}, there is a unique increasing $\O$-reduced tableau for $w_\lambda$, which has shape $\half(\lambda)$, so the following holds by Propositions~\ref{o-subcrystal-prop} and \ref{o-lowest-prop}.

\begin{corollary}
\label{o-sum-cor}
Suppose $\lambda$ is a symmetric partition and $\mu = \half(\lambda)$. 
The normal $\qq_n$-crystal $\ORF(\lambda)$ is nonempty if and only if $\mu \in \NN^n$, 
in which case it is connected with unique highest weight $\mu$.
Thus, each connected normal $\qq_n$-crystal is isomorphic to a unique $ \ORF(\lambda)$.
\end{corollary}

Also let $\upORF(\lambda) := \upORF(w_\lambda)$ for each symmetric partition $\lambda$.

\begin{corollary}
Suppose $\lambda$ is a symmetric partition and $\gamma$ is the unique skew-symmetric partition with $\half(\lambda) = \shalf(\gamma)$. Then $\upORF(\lambda) \iso \SpRF(\gamma)$ as $\q_n$-crystals.
\end{corollary}

\begin{proof}
It suffices 
to show that the normal $\q_n$-crystal $\upORF(\lambda)$ is connected with highest weight $\half(\lambda)$.
This holds by Corollary~\ref{o-sum-cor}, as 
 $\unprime \colon \ORF(z) \to \upORF(z)$ sends $\qq_n$-highest weights to $\q_n$-highest weights and  directed paths in the crystal graph to directed paths by Proposition~\ref{orf-prop}.
\end{proof}

\begin{corollary}
If $\cB$ is a normal $\qq_n$-crystal then for all $i' \in I'$ and $j \in I\sqcup \overline I$ the crystal operators $e_{i'}$ and $f_{i'}$ preserve the statistics $\varepsilon_j$ and $\varphi_j$.
\end{corollary}

\begin{proof}
It is enough to check this when $\cB = \ORF(\lambda)$; then the result holds by  Proposition~\ref{orf-prop}.
\end{proof}

\subsection{Demazure \texorpdfstring{$\qq_n$}{qqn}-crystals}

Choose an element $z \in I_\infty$ and a flag $\phi = (\phi_1 \leq \phi_2 \leq \cdots)$.
We say that $a=(a^1,a^2,\dotsc,a^n) \in \RF^+(w)$ for $w \in S_\infty$ is \defn{bounded} by $\phi$ 
if every letter $m$ in $a^i$ has $i \leq \phi_{\lceil m \rceil}$.  
Let $\BRF^+(w,A,\phi)$ be the set of such $\phi$-bounded elements $a\in\RF^+(w,A)$ and define
\be\label{borf-def}
\borf_n(z,\phi) :=  \bigsqcup_{A\subseteq \Cyc(z)} \bigsqcup_{w \in \cA^\O(z)} \brf^+_n(w,A,\phi) \subseteq \ORF(z).
\ee
We view this set as a $\qq_n$-subcrystal of $\orf_n(z)$; by  Theorem~\ref{up-thm} and Corollary~\ref{is-dem-cor}, it is also a direct sum of Demazure $\gl_n$-crystals.
As usual, let 
$\borf_n(z) := \borf_n(z,\phi^S)$. 

\begin{example}\label{borf-ex}
The crystal graph of $\borf_2(z)$ for $z= (1\: 4)(3\: 6)\in I_6$ is 
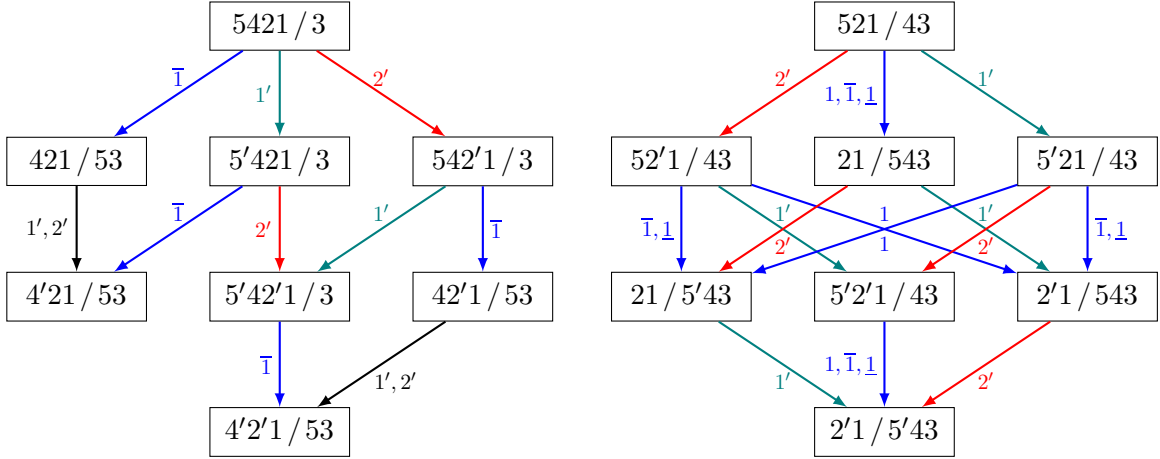
\[
\begin{tikzpicture}[xscale=2.7,yscale=1.8,>=latex]
\node [bnode,text width=16mm] (b1) at (1, 0) {$5421 \gap 3$};
\node [bnode,text width=16mm] (a2) at (0, -1) {$421 \gap 53$};
\node [bnode,text width=16mm] (b2) at (1, -1) {$5'421 \gap 3$};
\node [bnode,text width=16mm] (c2) at (2, -1) {$542'1 \gap 3$};
\node [bnode,text width=16mm] (a3) at (0, -2) {$4'21 \gap 53$};
\node [bnode,text width=16mm] (b3) at (1, -2) {$5'42'1 \gap 3$};
\node [bnode,text width=16mm] (c3) at (2, -2) {$42'1 \gap 53$};
\node [bnode,text width=16mm] (b4) at (1, -3) {$4'2'1 \gap 53$};
\draw[->,thick,color=blue] (b1) -- (a2) node[midway,above,scale=0.75] {$\overline 1$};
\draw[->,thick,color=teal] (b1) -- (b2) node[midway,left,scale=0.75] {$1'$};
\draw[->,thick,color=red] (b1) -- (c2) node[midway,above,scale=0.75] {$2'$};
\draw[->,thick,color=black] (a2) -- (a3) node[midway,left,scale=0.75] {$1',2'$};
\draw[->,thick,color=blue] (b2) -- (a3) node[midway,above,scale=0.75] {$\overline 1$};
\draw[->,thick,color=red] (b2) -- (b3) node[midway,left,scale=0.75] {$2'$};
\draw[->,thick,color=teal] (c2) -- (b3) node[midway,above,scale=0.75] {$1'$};
\draw[->,thick,color=blue] (c2) -- (c3) node[midway,right,scale=0.75] {$\overline 1$};
\draw[->,thick,color=blue] (b3) -- (b4) node[midway,left,scale=0.75] {$\overline 1$};
\draw[->,thick,color=black] (c3) -- (b4) node[midway,below,scale=0.75] {\ \quad$1',2'$};
\end{tikzpicture}
\qquad
\begin{tikzpicture}[xscale=2.7,yscale=1.8,>=latex]
\node [bnode,text width=16mm] (b1) at (1, 0) {$521 \gap 43$};
\node [bnode,text width=16mm] (a2) at (0, -1) {$52'1 \gap 43$};
\node [bnode,text width=16mm] (b2) at (1, -1) {$21 \gap 543$};
\node [bnode,text width=16mm] (c2) at (2, -1) {$5'21 \gap 43$};
\node [bnode,text width=16mm] (a3) at (0, -2) {$21 \gap 5'43$};
\node [bnode,text width=16mm] (b3) at (1, -2) {$5'2'1 \gap 43$};
\node [bnode,text width=16mm] (c3) at (2, -2) {$2'1 \gap 543$};
\node [bnode,text width=16mm] (b4) at (1, -3) {$2'1 \gap 5'43$};
\draw[->,thick,color=red] (b1) -- (a2) node[midway,above,scale=0.75] {$2'$};
\draw[->,thick,color=blue] (b1) -- (b2) node[midway,left,scale=0.75] {$1,\overline 1, \underline 1$};
\draw[->,thick,color=teal] (b1) -- (c2) node[midway,above,scale=0.75] {$1'$};
\draw[->,thick,color=blue] (a2) -- (a3) node[midway,left,scale=0.75] {$\overline 1,\underline 1$};
\draw[->,thick,color=teal] (a2) -- (b3) node[midway,above,scale=0.75] {$1'$};
\draw[->,thick,color=blue] (a2) -- (c3) node[midway,above,scale=0.75] {$1$};
\draw[->,thick,color=red] (b2) -- (a3) node[midway,below,scale=0.75] {$2'$};
\draw[->,thick,color=teal] (b2) -- (c3) node[midway,above,scale=0.75] {$1'$};
\draw[->,thick,color=blue] (c2) -- (a3) node[midway,below,scale=0.75] {$1$};
\draw[->,thick,color=red] (c2) -- (b3) node[midway,below,scale=0.75] {$2'$};
\draw[->,thick,color=blue] (c2) -- (c3) node[midway,right,scale=0.75] {$\overline 1, \underline 1$};
\draw[->,thick,color=teal] (a3) -- (b4) node[midway,below,scale=0.75] {$1'$};
\draw[->,thick,color=blue] (b3) -- (b4) node[midway,left,scale=0.75] {$1,\bar1,\underline 1$};
\draw[->,thick,color=red] (c3) -- (b4) node[midway,below,scale=0.75] {$2'$};
\end{tikzpicture}
\]
For another example of $\BORF(z)$ see Figure~\ref{fig:O_crystal}.
\end{example}

The characters of the $\qq_n$-crystals $\borf_n(z)$
are closely related to the involution Schubert polynomials of orthogonal type $\fkSO_z$
defined in Section~\ref{schub-sect}.
There is a version of Proposition~\ref{sp-schub-prop} for these polynomials.
Recall that $i \in \ZZ$ is a \defn{descent} of $z$ if $z(i)>z(i+1)$.

\begin{proposition}\label{o-schub-prop}
Suppose $z \in I_\infty$. Then $\fkSO_z \in \ZZ[x_1,x_2,\dots,x_n]$  if and only if $z$ has no descents greater than $n$.
\end{proposition}

\begin{proof}
If $z=1$ then $z$ has no descents and $ \fkSO_z=1\in \ZZ[x_1,\dots,x_n]$ for all $n$.
Assume $z \neq 1$ so that  $\fkSO_z$ is homogeneous of positive degree.
Let $i$ be the maximal index such that $x_i$ divides a monomial appearing in $\fkSO_z$.
 Then 
   $\fkSO_z\in \ZZ[x_1,\dots,x_n]$
 if and only if  $i\leq n$.
As $i$ is also the  maximal index  with $\partial_i \fkSO_z \neq 0$, 
we conclude by~\eqref{ischub-eq1}  that $i\leq n$ if  and only if $z(j) < z(j+1)$ for all $j>n$.
\end{proof}

By  \eqref{atom-eq-o}, \eqref{cRo-eq},  and the Billey--Jockusch--Stanley formula 
 we have $\fkSO_z = \sum_{a \in \borf(z)} \xx^{\weight(a)}$ for each $z \in I_\infty$.
Proposition~\ref{o-schub-prop} tells us that if $z$ has no descents greater than $n$ then 
 \be\label{o-bjs-eq}
 \BORF(z) = \borf(z) \quand \ch\bigl( \BORF(z) \bigr) = \fkSO_z.
 \ee
If $z$ has at least one descent greater than $n$,
then $ \BORF(z) $ is a proper subset of $ \borf(z)$ and 
$ \ch\bigl( \BORF(z) \bigr) $ is the polynomial obtained from $\fkSO_z$ by setting $x_{n+1}=x_{n+2}=\cdots=0$.

 \begin{example}
The involution Schubert polynomial of $z =(1\: 4)(3\: 6) \in I_6$ is a polynomial  in $x_1,x_2,x_3,x_4,x_5$
with $26$ terms  and coefficients in $\{4,8,16\}$,
given by
\[ \fkSO_z = 4 \xx^{11201} + 4\xx^{11210} + 4\xx^{11300} +4\xx^{12101} \dots+ 16\xx^{31100} + 8 \xx^{32000} + 4\xx^{40100} + 4\xx^{41000}
\]
while $\ch(\borf_2(z)) = 4 \xx^{23} + 8\xx^{32} + 4\xx^{41}= 4 x_1^2x_2^3 + 8x_1^3x_2^2 + 4x_1^4x_2 $,
which is the polynomial obtained from $\fkSO_z$ by setting $x_3=x_4=x_5=0$.
\end{example}

Suppose $\alpha$ is a symmetric weak composition with $\lambda := \lambda(\alpha)$.
Assume that $u(\alpha) \in S_n$ and $\half(\lambda) \in \NN^n$, so that $\ORF(\lambda)$ is nonempty.
Both conditions must hold if $\alpha \in \NN^n$, but this is not necessary.
If $w_\lambda \in I_\infty$ is  dominant of shape $\lambda$,
then by Lemma~\ref{sp-ij-lem} we can define
\be\borf_n(\alpha) := \fkD_{i_1} \fkD_{i_2}\cdots \fkD_{i_k}\brf_n(w_\lambda) \subseteq \orf_n(\lambda)\ee
where $\fkD_i = \fkD_i^{\orf_n(\lambda)}$ and $i_1i_2\cdots i_k $ is any reduced word for $u(\alpha)$.
This gives a (nonempty) $\qq_n$-crystal that is also a direct sum of Demazure $\gl_n$-crystals.

  \begin{example}
  We have $w_{(4,2,1,1)} = (1\: 5)(2\: 3)$ and $w_{(3,3,2)} = (1\: 4)(2\: 5)$. 
  The connected crystal graphs of $\borf_2(\alpha)$ for $\alpha = (4,2,1,1)$ and $(3,3,2)$ are respectively
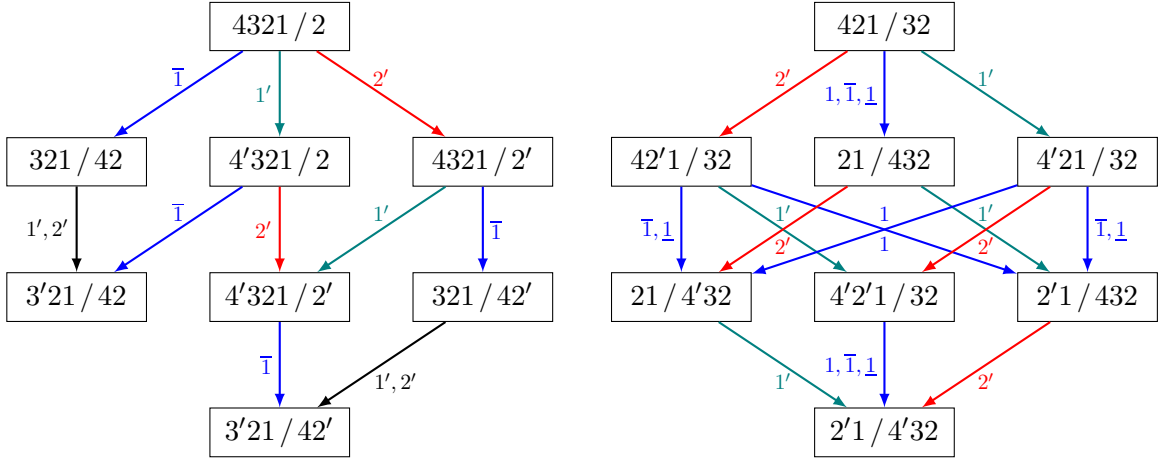
\[
\begin{tikzpicture}[xscale=2.7,yscale=1.8,>=latex]
\node [bnode,text width=16mm] (b1) at (1, 0) {$4321 \gap 2$};
\node [bnode,text width=16mm] (a2) at (0, -1) {$321 \gap 42$};
\node [bnode,text width=16mm] (b2) at (1, -1) {$4'321 \gap 2$};
\node [bnode,text width=16mm] (c2) at (2, -1) {$4321 \gap 2'$};
\node [bnode,text width=16mm] (a3) at (0, -2) {$3'21 \gap 42$};
\node [bnode,text width=16mm] (b3) at (1, -2) {$4'321 \gap 2'$};
\node [bnode,text width=16mm] (c3) at (2, -2) {$321 \gap 42'$};
\node [bnode,text width=16mm] (b4) at (1, -3) {$3'21 \gap 42'$};
\draw[->,thick,color=blue] (b1) -- (a2) node[midway,above,scale=0.75] {$\overline 1$};
\draw[->,thick,color=teal] (b1) -- (b2) node[midway,left,scale=0.75] {$1'$};
\draw[->,thick,color=red] (b1) -- (c2) node[midway,above,scale=0.75] {$2'$};
\draw[->,thick,color=black] (a2) -- (a3) node[midway,left,scale=0.75] {$1',2'$};
\draw[->,thick,color=blue] (b2) -- (a3) node[midway,above,scale=0.75] {$\overline 1$};
\draw[->,thick,color=red] (b2) -- (b3) node[midway,left,scale=0.75] {$2'$};
\draw[->,thick,color=teal] (c2) -- (b3) node[midway,above,scale=0.75] {$1'$};
\draw[->,thick,color=blue] (c2) -- (c3) node[midway,right,scale=0.75] {$\overline 1$};
\draw[->,thick,color=blue] (b3) -- (b4) node[midway,left,scale=0.75] {$\overline 1$};
\draw[->,thick,color=black] (c3) -- (b4) node[midway,below,scale=0.75] {\ \quad$1',2'$};
\end{tikzpicture}
\qquad
\begin{tikzpicture}[xscale=2.7,yscale=1.8,>=latex]
\node [bnode,text width=16mm] (b1) at (1, 0) {$421 \gap 32$};
\node [bnode,text width=16mm] (a2) at (0, -1) {$42'1 \gap 32$};
\node [bnode,text width=16mm] (b2) at (1, -1) {$21 \gap 432$};
\node [bnode,text width=16mm] (c2) at (2, -1) {$4'21 \gap 32$};
\node [bnode,text width=16mm] (a3) at (0, -2) {$21 \gap 4'32$};
\node [bnode,text width=16mm] (b3) at (1, -2) {$4'2'1 \gap 32$};
\node [bnode,text width=16mm] (c3) at (2, -2) {$2'1 \gap 432$};
\node [bnode,text width=16mm] (b4) at (1, -3) {$2'1 \gap 4'32$};
\draw[->,thick,color=red] (b1) -- (a2) node[midway,above,scale=0.75] {$2'$};
\draw[->,thick,color=blue] (b1) -- (b2) node[midway,left,scale=0.75] {$1,\overline 1, \underline 1$};
\draw[->,thick,color=teal] (b1) -- (c2) node[midway,above,scale=0.75] {$1'$};
\draw[->,thick,color=blue] (a2) -- (a3) node[midway,left,scale=0.75] {$\overline 1,\underline 1$};
\draw[->,thick,color=teal] (a2) -- (b3) node[midway,above,scale=0.75] {$1'$};
\draw[->,thick,color=blue] (a2) -- (c3) node[midway,above,scale=0.75] {$1$};
\draw[->,thick,color=red] (b2) -- (a3) node[midway,below,scale=0.75] {$2'$};
\draw[->,thick,color=teal] (b2) -- (c3) node[midway,above,scale=0.75] {$1'$};
\draw[->,thick,color=blue] (c2) -- (a3) node[midway,below,scale=0.75] {$1$};
\draw[->,thick,color=red] (c2) -- (b3) node[midway,below,scale=0.75] {$2'$};
\draw[->,thick,color=blue] (c2) -- (c3) node[midway,right,scale=0.75] {$\overline 1, \underline 1$};
\draw[->,thick,color=teal] (a3) -- (b4) node[midway,below,scale=0.75] {$1'$};
\draw[->,thick,color=blue] (b3) -- (b4) node[midway,left,scale=0.75] {$1,\bar1,\underline 1$};
\draw[->,thick,color=red] (c3) -- (b4) node[midway,below,scale=0.75] {$2'$};
\end{tikzpicture}
\]
  Observe that $\borf_2\bigl( (1\: 4)(3\: 6) \bigr) \iso \borf_2\bigl( (4,2,1,1) \bigr) \oplus \borf_2\bigl( (3,3,2) \bigr)$.
   \end{example}
   
 \begin{example}
 Like the symplectic case, the subsets $\borf_n(\alpha)\subseteq \borf_n(\lambda(\alpha))$
 are not closed by all raising crystal operators $e_i$ if we allow $i \in \overline I$. 
 The smallest example demonstrating this behavior is $ \borf_3(\alpha)$ for $\alpha=(3,1,1)=\lambda(\alpha)$.
The dominant involution of this shape is $w_{(3,1,1)} = (1 \: 4) \in I_4$ so we have $(1,2,3) \in \borf_3\bigl( (3,1,1) \bigr)$ but 
$e_{\overline 2}(1,2,3) =  (2,31,\emptyset)\notin \borf_3\bigl( (3,1,1) \bigr)$.
\end{example}

The following result is similar to Proposition~\ref{among-the-prop}:

\begin{proposition}\label{among-the-prop2}
Suppose $\alpha$ is a symmetric weak composition with $\half(\lambda(\alpha)) \in \NN^n$ and $u(\alpha)\in S_n$.
Then $\qkey_\alpha \in \ZZ[x_1,x_2,\dots,x_n]$ if and only if $\alpha \in \NN^n$.
\end{proposition}

\begin{proof}
Let $\lambda = \lambda(\alpha)$.
If $\alpha \in \NN^n$ then $\lambda=\lambda^\top \in \NN^n$, so as $u(\alpha) \in S_n$
both $\qkey_\lambda$ and $\qkey_\alpha = \pi_{u(\alpha)}\qkey_\lambda $ belong to $\ZZ[x_1,x_2,\dots,x_n]$.
Let $\gamma$ be the weak composition with $\gamma_i = |\{ (j,k) \in \D_\lambda : j\leq k=i\}|$ for each $i \in \PP$;
 note the slight difference with $\beta$ in the proof of Proposition~\ref{among-the-prop}.
Then~\cite[Eq.~(2.3) and Prop.~2.49]{MS2023} imply that 
$\xx^{u(\alpha) \circ \gamma}$ is a monomial appearing in $\qkey_\alpha$.
If $\alpha \notin \NN^n$ then we must have $\lambda \notin \NN^n$ as $u(\alpha) \in S_n$, 
in which case $ \gamma \notin\NN^n$ so $u(\alpha) \circ \gamma \notin\NN^n$ and
$\qkey_\alpha \notin \ZZ[x_1,x_2,\dotsc,x_n]$.
\end{proof}

In the next lemma, continue to fix an element $z \in I_\infty$ and a flag $\phi$.

\begin{lemma}\label{o-connect-lem}
Every $\gl_n$-highest weight element of $\orf_n(z)$ is in $\borf_n(z,\phi)$,
and 
if $a \in \borf_n(z,\phi)$ is not $\qq_n$-highest weight then 
$\zero \neq e_i a \in \borf_n(z,\phi)$ for some $i \in I  \sqcup \overline I \sqcup I'$.
Finally, if $i \in I'$ then $e_i$ and $f_i$ restrict to maps $\borf_n(z,\phi)\to\borf_n(z,\phi)\sqcup\{\zero\}$.
\end{lemma}

\begin{proof}
If $a \in \orf_n(z)$ is $\gl_n$-highest weight then $\unprime(a) \in \upORF(z)$ is also $\gl_n$-highest weight
by Proposition~\ref{orf-prop}(a). In this case, the latter element is bounded by $\phi$ by Lemma~\ref{each-highest-lem} 
since it belongs to $\BRF(w)$
for some $w \in \cA^\O(z)$, so $a$ is also bounded.

Suppose $a^+ \in \borf_n(z,\phi)$ is not $\qq_n$-highest weight and let $a := \unprime(a^+)$.
We have $a \in \brf_n(w,\phi)$ for some $w \in \cA^\O(z)$.
If $e_i a^+ \neq \zero$ for some $i \in I'$ then 
$\unprime(e_i a^+) = \unprime(a^+) = a \in \brf_n(w,\phi)$ by Proposition~\ref{orf-prop}(b), so $e_ia^+ \in \borf_n(z)$. 
If $e_i a^+ \neq \zero$ for some $i \in I\sqcup \overline I$ then  $\zero\neq \unprime(e_i a^+)  =e_i a \in \brf_n(w,\phi)$ by Lemma~\ref{brf-e-lem} and Proposition~\ref{orf-prop}(a),
so again $e_i a^+ \in \borf_n(z,\phi)$.

Assume $e_j(a^+) = \zero$ for all $j \in I\sqcup I'$ and let $i \in [n-1] = I$ be minimal with
 $e_{\overline{\imath}}(a^+)\neq \zero$.
 Then $a$ is a $\gl_n$-highest weight element of the normal $\q_n$-crystal $\upORF(z)$
 and $i $ is also the minimal index in $[n-1]$ with  $e_{\overline{\imath}} a \neq \zero$ by Proposition~\ref{orf-prop}(a),
 so $e_{\overline{\imath}} a$ is $[i]$-highest weight by Lemma~\ref{ghps-lem}.
From this observation, 
we can deduce that $e_{\overline{\imath}} a$ is bounded by $\phi$ in exactly the same way as in
the proof of Lemma~\ref{sp-connect-lem}.
Then, as $\unprime(e_{\overline{\imath}} a^+) = e_{\overline{\imath}} a$, we conclude that $e_{\overline{\imath}} a^+$ is also $\phi$-bounded.

The last assertion in the lemma holds by parts (b) and (c) of Proposition~\ref{orf-prop}.
\end{proof}

 This lemma leads to an analogue of Corollary~\ref{sp-dem-sub-cor}.
If $T$ is an $\O$-reduced tableau for $z \in I_\infty$ and $\phi$ is a flag, then  
let $\borf_n(T,\phi) := \orf_n(T) \cap \borf_n(z,\phi)$
and
$\borf_n(T) := \borf_n(T,\phi^S)$.

\begin{corollary}\label{o-dem-sub-cor}
The map $T \mapsto \borf_n(T,\phi)$ is a  bijection from increasing (equivalently, decreasing) $\O$-reduced tableaux for $z \in I_\infty$ with at most $n$ rows 
to full subcrystals of $\borf_n(z,\phi)$.
\end{corollary}

\begin{proof}
This is clear from Proposition~\ref{o-subcrystal-prop} and Lemma~\ref{o-connect-lem}.
\end{proof}

We now have the  $\qq_n$-version of Theorem~\ref{sp-crystal-thm}.
 
\begin{theorem}\label{o-crystal-thm}
Suppose $\alpha $ is a symmetric weak composition with $\half(\lambda(\alpha)) \in \NN^n$
and $u(\alpha) \in S_n$.
Then the (nonempty) $\qq_n$-crystal $\borf_n(\alpha) $ is connected 
and its character is the polynomial obtained from $\qkey_\alpha$
by setting $x_{n+1}=x_{n+2}=\cdots=0$.
\end{theorem}

\begin{proof}
The proof is similar to that of Theorem~\ref{sp-crystal-thm}.
Let $\lambda = \lambda(\alpha)$ and $z=w_\lambda \in I_\infty$.
To show that $\borf_n(\alpha) $ is connected it suffices to show that $\borf_n(\lambda) = \borf_n(z) $ is connected, and this follows from Corollary~\ref{o-sum-cor} and Lemma~\ref{o-connect-lem}.
If $N \geq n$ is sufficiently large then 
\eqref{o-bjs-eq} 
 tells us that $\ch(\borf_N(z))=  \fkSO_{z} = \qkey_\lambda$ so Lemma~\ref{sp-ij-lem} implies that
  $\ch(\borf_N(\alpha)) = \pi_{u(\alpha)} \qkey_\lambda  =  \qkey_\alpha$. 
This suffices as if $N \geq n$ then  $\ch(\borf_n(\alpha))$ is always obtained from $\ch(\borf_N(\alpha)) $ by  $x_{n+1}=x_{n+2}=\cdots=0$.
\end{proof}

In analogy with the previous cases, we make the following definition.

\begin{definition}\label{qq-dem-def}
A \defn{Demazure $\qq_n$-crystal} is a $\qq_n$-crystal isomorphic to $\borf_n(\alpha)$ for some symmetric weak composition $\alpha \in \NN^n$.
\end{definition}

As in the $\q_n$-case, this definition does not include  $\borf_n(\alpha)$ 
 if $\half(\lambda(\alpha)) \in \NN^n$ and $u(\alpha)\in S_n$ but $\alpha \notin \NN^n$.
Although the $\qq_n$-crystal is still defined and connected in this case, its character is not the $Q$-key polynomial $\kappa_\alpha$.
Instead, the character is obtained from  $\kappa_\alpha$ by setting $x_{n+1}=x_{n+2}=\dots=0$.
This function might not be equal to any linear combination of $Q$-key polynomials.

By Theorem~\ref{o-crystal-thm}, every Demazure $\qq_n$-crystal
has a unique highest weight element. This leads to another analogue of Proposition~\ref{unique-embed-prop}, which follows 
as an exercise from Definition~\ref{qq-dem-def}: 

\begin{proposition}\label{unique-embed-prop3}
Suppose $\cB$ is a connected normal $\qq_n$-crystal with highest weight element $b$.
Let $\cX$ be a Demazure $\qq_n$-crystal with highest weight element $u$.
\ben
\item[(a)] There is a unique embedding $\cX \to \cB$ if $\weight(b) = \weight(u)$.
\item[(b)] There are no embeddings $\cX \to \cB$ if $\weight(b) \neq \weight(u)$.
\item[(c)] If $\cX\subseteq \cB$ then $u=b$ and $\cX = \fkD^\cB_{i_1} \fkD^\cB_{i_2}\cdots \fkD^\cB_{i_k} \{b\} $ for some $i_1,i_2,\dotsc,i_k \in [n-1]$.
\item[(d)] Any symmetric $\alpha \in \NN^n$ with $  \borf_n(\alpha) \cong \cX$ has 
 $\qkey_\alpha=\ch(\cX)$ and $\half(\lambda(\alpha)) = \weight(u)$.
 \een
\end{proposition}

The main open problem concerning the $\qq_n$-crystals $\borf_n(z,\phi)$ is the following conjecture.

\begin{conjecture}\label{o-demazure-conj}
If $z \in I_\infty$ is an involution with no descents greater than $n$  
and  $\phi$ is any flag,  then $\borf_n(z,\phi)$ is a direct sum of Demazure $\qq_n$-crystals.
\end{conjecture}
 
This is a generalization of the orthogonal half of Conjecture~\ref{fkSS-conj}, which follows by taking characters when $\phi=\phi^S$.

\begin{example} \label{ex:specific_decomposition}
Let $z = (1 \: 4)(2 \: 5)(6 \: 8)$. We have checked  that there is a $\qq_n$-crystal isomorphism 
\[ 
\borf_n(z) \iso \borf_n(343001)\oplus \borf_n(3340001) \oplus \borf_n(442002)^{\oplus 2} \oplus \borf_n(3520011)^{\oplus 2}
\]
for $n= 7$ (and hence, by \eqref{o-bjs-eq}, for all $n\geq 7$). On taking characters, this implies that
\[
\fkSO_{z} 
= \qkey_{343001} + \qkey_{3340001} + 2\qkey_{442002} + 2 \qkey_{3520011}.
\]
Interestingly, this is another $\NN$-linear expansion  of $\fkSO_{z} $ into $Q$-key polynomials:
\[\fkSO_{z} = \qkey_{334001} + \qkey_{3430001} + 2\qkey_{442002} + 2 \qkey_{3520011}.\]
However for $n\geq 7$ we have
\[ 
\borf_n(z) \not\iso \borf_n(334001)\oplus \borf_n(3430001) \oplus \borf_n(442002)^{\oplus 2} \oplus \borf_n(3520011)^{\oplus 2}
\]
since the $6$ polynomials $\qkey_{343001}$, $\qkey_{3340001}$, $\qkey_{334001}$, $\qkey_{3430001}$, $\qkey_{442002}$, and $\qkey_{3520011}$ are distinct.
\end{example}

Like  Conjecture~\ref{sp-demazure-conj}, it will turn out that Conjecture~\ref{o-demazure-conj} is equivalent to a different statement, given below, which only involves bounded reduced factorizations for the standard flag:

\begin{conjecture}\label{o-demazure-conj2}
 Suppose $T$ is an $\O$-reduced tableau for an involution $z \in I_\infty$ that has no descents greater than $n$.
Then there is a $\qq_n$-crystal isomorphism $\borf_{n}(T) \iso \borf_n(\alpha )$
for some symmetric weak composition  $\alpha=\alpha^\O(T) \in \NN^n$.
\end{conjecture}

We have checked this conjecture by computer for all $z \in I_7$. 
Figure~\ref{o-fig}  shows some examples of the symmetric weak compositions $\alpha^\O(T) $
that correspond to various $\O$-reduced tableaux.

To show the equivalence of Conjectures~\ref{o-demazure-conj} and \ref{o-demazure-conj2}, we need 
a result similar to Proposition~\ref{sp-alpha-prop}.
Suppose $T$ is a $\O$-reduced tableau for some $z \in I_\infty$.
Let $\mu$ be the strict partition shape of $T$ and let $\lambda$ be the unique symmetric partition with $\half(\lambda)=\mu$. Recall the definitions of the permutations $\Delta_n(\phi)$ and $u_n(\alpha)$ from \eqref{Delta-eq}.
Finally assume  $\alpha$ is a symmetric weak composition
and choose a positive integer $N\in \PP$ that is greater than or equal to $\ell(\alpha)$ and every descent of $z$.

\begin{proposition}\label{o-alpha-prop}
If there exists a $\qq_n$-crystal isomorphism $\borf_{N}(T) \iso \borf_N(\alpha)$, 
then 
\[\borf_{n}(T,\phi) \iso \borf_n(\Delta_n(\phi) \circ u_n(\alpha) \circ \lambda(\alpha))\]
as $\qq_n$-crystals
for all integers $n \geq \ell(\mu)$ and all flags $\phi$.
\end{proposition}

\begin{proof}
Suppose  $\borf_{N}(T) \iso \borf_N(\alpha)$ so that $\borf_{N}(T) $ is a Demazure $\qq_N$-crystal.
In view of Corollary~\ref{o-highest-cor}, Proposition~\ref{unique-embed-prop3}(d) implies that $\half(\lambda(\alpha)) =  \mu$ so $\lambda(\alpha) = \lambda$.
Fix $n\geq \ell(\mu)$ and a flag $\phi$, and define $\beta:=\Delta_n(\phi) \circ u_n(\alpha) \circ \lambda(\alpha)$.
Since
$\half(\lambda(\beta)) = \half(\lambda )=\mu \in \NN^n$ and $u(\beta) \in S_n$,
the $\qq_n$-crystal $\borf_n(\beta)$ 
is well-defined (though not a Demazure $\qq_n$-crystal if $\lambda \notin \NN^n$).
The rest of the proof is the same as for Proposition~\ref{sp-alpha-prop}; just change all ``$\Sp$'' superscripts to ``$\O$''.
\end{proof}

\begin{proposition}\label{o-equiv-prop}
Conjectures~\ref{o-demazure-conj} and~\ref{o-demazure-conj2} are equivalent.
\end{proposition}

\begin{proof}
The idea is the same as for Proposition~\ref{sp-equiv-prop}.
If Conjecture~\ref{o-demazure-conj2} holds then Corollary~\ref{o-dem-sub-cor} and Proposition~\ref{o-alpha-prop}
imply that every full subcrystal of $\borf_n(z,\phi)$ is a Demazure $\qq_n$-crystal. 
If Conjecture~\ref{o-demazure-conj} holds then taking $\phi=\phi^S$
shows that each full subcrystal
of $\borf_n(z)$
 is isomorphic to $\borf_n(\alpha)$ for some symmetric $\alpha \in \NN^n$, which implies Conjecture~\ref{o-demazure-conj2}
by Corollary~\ref{o-dem-sub-cor}.
\end{proof}

When $T$ is an $\O$-reduced tableau with just one row, we can prove a formula for $ \alpha^\O(T)$.

\begin{theorem}\label{thm:single_row_o_tableau}
Suppose $T$ is an $\O$-reduced tableau with only one row for an element  $z \in I_\infty$ that has no descents greater than $n$.
Let $i_1 < i_2 <i_3< \cdots <i_k$ be the entries of $\unprime(T)$ and define $\alpha^\O(T) = k \e_{i_1} + \e_{i_2} +\e_{i_3}+ \cdots + \e_{i_k}.$
Then $\borf_{n}(T) \iso \borf_n(\alpha^\O(T) )$ as $\qq_n$-crystals.
\end{theorem}

\begin{example}
If $T = \ytab{2 &3 & {5'} & {7'} & 8}$ then $\alpha^{\O}(T) = 05101011$.
\end{example}

\begin{proof}
It follows by inspecting the definition of $\simOCK$ that 
 every $a=(a^1,a^2,\dotsc,a^n) \in \ORF(T)$
has the property that  $\unprime(a^1a^2\cdots a^n)$ is a permutation of $i_1i_2\cdots i_k$.
We can therefore define an operator
  $\psi$ that acts on $a \in \cB$
by replacing each letter $i_j$ or $i_j'$ in each component by $j$, and then adding a prime to the first letter of the first nonempty component
if this letter is primed in $a$.
For example, $\psi$ would send $(\emptyset, 6'5,\emptyset,\emptyset,81') \mapsto (\emptyset, 3'2,\emptyset,\emptyset,41)$
and $(\emptyset, 65,\emptyset,\emptyset,81') \mapsto (\emptyset, 32,\emptyset,\emptyset,41)$
if   $\{i_1<i_2<\dots<i_k\} = \{1,5,6,8\}$.

Let $\lambda :=1^{k-1}k = (k,1,1,\dotsc,1) \in \NN^k$.
We claim that $\psi$ is a strict $\qq_n$-crystal morphism $\ORF(T)\to \ORF(\lambda)$.
If this holds, then $\psi$ must be an isomorphism 
as $\ORF(T)$ and $ \ORF(\lambda)$ are both connected normal $\qq_n$-crystals 
with  highest weight $\half(\lambda) = (k)$.
Given that  $\psi$ is an isomorphism, 
the image of $\BORF(T)$ under $\psi$
is evidently the set of reduced factorizations in $\ORF(\lambda) = \ORF(w_\lambda)$ that are 
bounded by any flag $\phi$ with $\phi_j = i_j$ for $j \in [k]$.
For such a flag, we have $\BORF(T) \cong \BORF(w_\lambda, \phi)$ as $\qq_n$-crystals,
and Proposition~\ref{alpha-prop} implies that $\BORF(w_\lambda, \phi) \cong \BORF(\Delta(\phi) \circ _n\lambda)$.
We obtain the desired result after this step  on noting that  $\Delta(\phi) \circ_n \lambda= k \e_{i_1} + \e_{i_2} +\e_{i_3}+ \cdots + \e_{i_k}$
whenever $i_k \leq n$, which must hold for $z$ to have no descents greater than $n$. 

It remains to prove our claim that $\psi$ is a strict $\qq_n$-crystal morphism $\ORF(T)\to \ORF(\lambda)$.
By Proposition~\ref{o-lowest-prop} and Proposition~\ref{o-dom-prop}, the unique lowest weight elements of $\ORF(T)$ and  $\ORF(\lambda)$ 
have the form 
 $a=(a^1,\emptyset,\dotsc,\emptyset)$
 and
  $b=(b^1,\emptyset,\dotsc,\emptyset)$, respectively,
where $a^1 = i_k'i_{k-1}i_{k-2}\cdots i_3i_2i_1$ and $b^1 = k'(k-1)(k-2)\cdots 321$.
It follows that $\psi$ at least maps $a \mapsto b$.

 The crystal graph of $\ORF(T)$ is weakly connected by the arrows labeled by $i \in \{1',\overline 1\} \sqcup [n-1]$.
To prove our claim, it therefore suffices to check that for any element $a=(a^1,a^2,\dotsc,a^n)  \in \ORF(T)$, 
we have
$\psi(e_i a) = e_i \psi(a)$ and $\psi(f_i a) = f_i \psi(a)$
for just these indices $i$.
 Since  $\unprime(a^1a^2\cdots a^n)$ is a permutation of $i_1i_2\cdots i_k$ and since $\psi$
 is the obvious standardization map if we ignore primes,
it follows immediately
 from Definitions~\ref{ef-def},~\ref{ef-1-def}, and~\ref{e0-f0-def} 
 that we at least have $\unprime(\psi(e_i a)) = \unprime(e_i \psi(a))$ and $\unprime(\psi(f_i a)) = \unprime(f_i \psi(a))$.
 If we take primes into consideration, then what we want to show 
 is still clear from the definitions 
 outside two exceptional cases.

The first exception is when $i \in [n-1]$ 
 and $b := e_i(a) \neq \zero$ but the words  $a^1a^2\cdots a^n$ and $b^1b^2\cdots b^n$ start with different letters.
 In this case the first nonempty component of $a$ must be $a^i$, and the component $b^i$ must have at least two letters. 
Then
 $\psi(e_i a)$ and $ e_i \psi(a)$ are identical if we ignore the primes on the first two letters of their $i$th components,
 but which of these letters is primed in principle could differ for the two factorizations.
 
 To deal with this exceptional case, we argue that it actually can never occur if $T$ has only one row.
 Recall from the proof of Theorem~\ref{o-reduced-tab-lem} that $T$ has the same shape as $\PO(a^1a^2\cdots a^n)$
 where $\PO$ denotes the orthogonal Edelman--Greene insertion algorithm from~\cite{Marberg2021a}.
 The case under consideration can only arise if $\max(a^i) < \max(a^{i+1})$
(as otherwise $a^1a^2\cdots a^n$ and $b^1b^2\cdots b^n$ would start with the same letter)
 and $a^{i+1}$ has a second letter that is between these two values
 (as otherwise 
 $\max(a^{i+1})$
 would not be unpaired in the $(a^i,a^{i+1})$-pairing relevant for computing $e_ia$). 
 In this situation, however, the shifted tableau $\PO(a^1a^2\cdots a^n)$ must have at least two rows,
 since inserting the second letter of $a^{i+1}$ according to procedure in~\cite[Def.~3.1]{Marberg2021a} 
 will ``bump'' $\max(a^{i+1})$ from   a non-diagonal position in the first row to the second row.
 
 The second exception is when $i \in [n-1]$ 
 and $b := f_i(a) \neq \zero$ but  $a^1a^2\cdots a^n$ and $b^1b^2\cdots b^n$ start with different letters.
This situation also can never occur if $T$ has only one row, 
as interchanging $a$ and $b$ would put us back in the first exceptional case.
We conclude that  $\psi(e_i a) = e_i \psi(a)$ and $\psi(f_i a) = f_i \psi(a)$
for all  $i \in \{1',\overline 1\} \sqcup [n-1]$ and $a=(a^1,a^2,\dotsc,a^n)  \in \ORF(T)$,
as needed.
\end{proof}

It may be possible to use a similar argument to calculate $\alpha^\Sp(T)$ when $T$ has only one row.
However, the relevant formula for $\alpha^\Sp(T)$ is not as easy to guess; see the examples in Figure~\ref{sp-fig1}.

We mention one other conjecture.
Recall from Proposition~\ref{rf-nonempty-prop} that the (Lehmer) code of $w \in S_\infty$ is  the weak composition $c(w) = (c_1,c_2,\ldots)$ with $c_i = \abs{\{ j  \in \ZZ \mid i<j\text{, }w(i)>w(j)\}}$. 
If $w$ is \defn{vexillary} is the sense of being a $2143$-avoiding permutation, then $\fkS_w = \kappa_{c(w)}$ by~\cite[Thm.~22]{ReinerShimozono} and consequently $\BRF(w) \iso \BRF(c(w))$ whenever $n$ is greater than or equal to all descents of $w$.

If $z \in I_\infty$ is vexillary, then its code is a symmetric weak composition by the remarks before~\cite[Prob.~3.36]{MS2023}.
For such involutions, our computations suggest that $\fkSO_z = \qkey_{c(z)}$ \cite[Conj.~2.39]{MS2023}. 
Further calculations support a crystal-theoretic generalization of this conjecture:
 
\begin{conjecture}\label{z-vex-conj}
If $z \in I_\infty$ is vexillary with largest descent at most $n$,
 then there is an isomorphism of $\qq_n$-crystals $\BORF(z) \iso \BORF(c(z))$.
 \end{conjecture}

As discussed in~\cite[\S2.5]{MS2023}, the crystals $\BRF(w)$ and  $\BORF(z)$ can only be connected for all sufficiently large $n$ when $w$ and $z$ are vexillary.
There is a similar \defn{fpf-vexillary} property that characterizes when $\BSpRF(y)$ is connected for $y \in \Ifpf_\infty$~\cite[Prop.~2.37]{MS2023}.
However, we do not currently have a prediction for how to express $\alpha$ in terms of $y$ when $\BSpRF(y) \iso \BSpRF(\alpha)$.

\def\LGr{\mathsf{LGr}}

 \section{Concluding remarks}
 
As mentioned in the introduction, key polynomials are characters of $B_n$-modules $\Gamma(X_w, L_{-\lambda})$ given by restricting the Borel--Weil construction to a Schubert variety $X_w \subseteq B_n^- \backslash \GL_n$.
However, it is an open problem to find an analogous geometric description of $P$- and $Q$-key polynomials.

There is an interpretation of the Schur $Q$-function $Q_\lambda(\xx)$ as the character of the polynomial representation indexed by $\lambda$ for the queer Lie superalgebra $\q_n$~\cite{Kac78} (see also~\cite[\S2.3.4]{CW12}). 
Thus, a reasonable goal on the way to resolving our open problem would be to build a representation of $\q_n$ or its corresponding Lie supergroup $\mathsf{Q}_n := \mathsf{Q}_n(\CC)$ on $\Gamma(Y_w, L_{-\lambda})$ for some space $Y$ with subvarieties $Y_w$ naturally indexed (or constructed from) permutations $w$ in (some subset of) $S_n$.

Instead of working with $\O_n$-orbits of $\flagvar_n$,   consider the closures of $B_n^-$-orbits $Y_w := \overline{B_n^- w \GL_n / \O_n}$ in the symmetric space $Y = \SymO_n := \GL_n / \O_n$.
Note that $\SymO_n$ is the set of $n \times n$ symmetric $\CC$-matrices and $\SymO_n = Y_{w_0}$.
However, $\SymO_n$ is an affine variety rather than projective variety, and so it (and the expected subvarieties $Y_w$) might not have useful cohomology theories.

One way to fix this could be to intersect $Y_w$ with the $n \times n$ unitary matrices $\U_n$ since $Y \cap \U_n = \U_n / \O_n(\RR)$ is known to be diffeomorphic to the Lagrangian Grassmannian $\LGr(\RR^{2n})$~\cite{Arnold67}, which is a projective variety.
This gives us a candidate for a nice decomposition of $\LGr(\RR^{2n})$ to understand its cohomology ring.
There is another homogeneous space description of $\LGr_n(\RR^{2n}) \iso \overline{\Sp}_n / \U_n$, where $\overline{\Sp}_n = \Sp_{2n} \cap \U_{2n}$ is the compact real form of $\Sp_{2n}$. 
The complex Lagragian Grassmannian $\LGr(\CC^{2n})$ has also been studied and has a known relation to Schur $Q$-functions (e.g.,~\cite{IN09}
).
It has its own homogeneous space construction as $\LGr(\CC^{2n}) \iso \Sp_{2n} / P$, where $P$ is the maximal parabolic subgroup associated to the long simple root.

This suggests that there might be (a real form of) a queer Lie supergroup $\mathsf{Q}_n$ acting on the set of global sections $\Gamma(\LGr(F^{2n}), L_{-\lambda})$ for $F = \RR$ or $\CC$   and some appropriate line bundle $L_{-\lambda}$.
Then by restricting to the orbit closure $Y_w$ indexed by $w$, we would obtain a representation of some Lie sub-(super)group of $\mathsf{Q}_n$ analogous to $B_n$ whose corresponding character would be equal to the $Q$-key polynomial $\qkey_{w\lambda}$.

Another potential geometric interpretation of $P$- and $Q$-key polynomials could lie with some generalization of Hessenberg varieties~\cite{CMPS92}.
Specifically, let $\lambda = (k)$ be a single row of length $k$, and let $\fkS_w(\xx; t_1, t_2, \ldots)$ denote the \defn{double Schubert polynomial} with equivariant parameters $t_i$ (see, e.g.,~\cite{Knutson22}).
Then we have
\[
\fkS_{w_{\lambda}}(\xx; -x_1, -x_2, \ldots) = \fkSO_z = \qkey_{\lambda}
\quand
\fkS_{w_{\lambda}}(\xx; -x_2, -x_3, \ldots) = \fkSS_z = \pkey_{\lambda},
\]
where $z$ is the unique dominant (fpf) involution of shape $\lambda$.
Since Hessenberg varieties are degenercy loci, we can obtain polynomial representatives for the corresponding cohomology classes by specializations of double Schubert polynomials~\cite[\S3]{AT10}.
These are specializations of the factorial Schur function $s_{\lambda}(x_1 | t_1, t_2, \ldots)$ (see, e.g.,~\cite[$6^{\text{th}}$ variation]{Macdonald92}) or a flagged version for general $\xx$.
Moreover, a Borel--Weil(--Bott) construction for Hessenberg varieties was considered in~\cite{AFZ20}.

Due to the positive shift by the row index (as opposed to the negative shift to compute the content of a box used in the factorial Schur function construction), these formulas break down for $\lambda$ with more than one part.
We can get around this by using the half shapes and the corresponding \defn{factorial Schur $P$- and $Q$-functions} (see, e.g.,~\cite{IN09}), denoted here by $P_{\mu}(\xx|t_1, t_2, \ldots)$ and $Q_{\mu}(\xx|t_1, t_2, \ldots)$ and specializing the equivariant parameters $t_i = -x_i$.
However, we cannot simply restrict the number of variables to the length of the half shape as before as we need to only have a single product of binomials.
Thus, we impose a flagging $\phi$ on the shifted tableaux, where row $i$ only contains entries at most $\phi_i$, and we denote the resulting polynomials by $P_{\mu,\phi}(\xx|t_1, t_2, \ldots)$ and $Q_{\mu,\phi}(\xx|t_1, t_2, \ldots)$.
These do not appear to have been considered before in the literature, 
and are slightly different from the flagged factorial Schur $P$- and $Q$-functions studied in~\cite{Matsumura19},
However, they specialize to the $P$- and $Q$-key polynomials of any symmetric partition $\lambda$ by the formulas
\be
\label{eq:factorial_PQ_specialization}
\pkey_{\lambda} = P_{\shalf(\lambda), \phi^S}(\xx | {-\xx})
\quand
\qkey_{\lambda} = Q_{\half(\lambda), \phi^S}(\xx| {-\xx}).
\ee

Degeneracy loci constructions are given in~\cite{Wyser13} for the $\Sp_{2n}$/$\O_n$-orbit closures in the complete flag variety.
In the $\O_n$ case, these constructions were used in~\cite{MP2019a} to obtain formulas for $K$-theoretic generalizations of 
involution Schubert polynomials indexed by vexillary involutions.
It may be possible to find an analogous construction of Hessenberg varieties (within the $\Sp_{2n}$ and $\O_n$ flag varieties) as degeneracy loci that will yield $\pkey_{\lambda}$ and $\qkey_{\lambda}$ as a consequence.
By translating the divided difference operators back to a geometric construction, one would be able to obtain a fully geometric description of the $P$- and $Q$-key polynomials.

A now classical fact due to Wachs~\cite[Thm.~2.4]{Wachs85} is that every flagged Schur function is the Schubert polynomial $\partial_w \xx^{\lambda}$ for some $\lambda$ and some word $w$ analogous to the one given in~\eqref{Delta-eq}.
Likewise, Reiner--Shimozono showed that every flagged Schur function is a key polynomial~\cite[Thm.~23]{ReinerShimozono}.
From~\cite[Thm.~5.3]{Matsumura19}, we have a positive monomial expansion of $P_{\lambda}(\xx|{-t_1}, -t_2, \ldots)$ and $Q_{\lambda}(\xx|{-t_1}, -t_2, \ldots)$.
Therefore, a natural question would be to see if $P_{\mu,\phi}(\xx|{-\xx})$ (respectively $Q_{\mu,\phi}(\xx|{-\xx})$), for any flag $\phi$, is a $P$-key (respectively $Q$-key) polynomial (or if not, a positive expansion of such polynomials) or an involution Schubert polynomial.
One way to attempt to answer this would be to see if the restriction of the natural $\q_n$-crystal (respectively $\qq_n$-crystal) is a Demazure $\q_n$-crystal (respectively $\qq_n$-crystal).

\begin{figure}[h]
\begin{center}
\begin{tabular}{l|l|l|l}
\ytableausetup{smalltableaux}
\begin{tabular}[t]{ll}
$T$ & $\alpha^{\Sp}(T)$ \\ \hline \\[-6pt]
$\begin{ytableau}2\end{ytableau}$ & {\footnotesize$22$}
\\[-6pt] \\
$\begin{ytableau}4\end{ytableau}$ & {\footnotesize$0022$}
\\[-6pt] \\
$\begin{ytableau}6\end{ytableau}$ & {\footnotesize$000022$}
\\[-6pt] \\
$\begin{ytableau}2 & 1\end{ytableau}$ & {\footnotesize$311$}
\\[-6pt] \\
$\begin{ytableau}4 & 2\end{ytableau}$ & {\footnotesize$1301$}
\\[-6pt] \\
$\begin{ytableau}6 & 2\end{ytableau}$ & {\footnotesize$130001$}
\\[-6pt] \\
$\begin{ytableau}4 & 3\end{ytableau}$ & {\footnotesize$00311$}
\\[-6pt] \\
$\begin{ytableau}6 & 4\end{ytableau}$ & {\footnotesize$001301$}
\\[-6pt] \\
$\begin{ytableau}6 & 5\end{ytableau}$ & {\footnotesize$0000311$}
\\[-6pt] \\
$\begin{ytableau}4 & 2 & 1\end{ytableau}$ & {\footnotesize$4111$}
\\[-6pt] \\
$\begin{ytableau}6 & 2 & 1\end{ytableau}$ & {\footnotesize$411001$}
\\[-6pt] \\
$\begin{ytableau}4 & 3 & 2\end{ytableau}$ & {\footnotesize$14011$}
\\[-6pt] \\
$\begin{ytableau}6 & 4 & 2\end{ytableau}$ & {\footnotesize$140101$}
\\[-6pt] \\
$\begin{ytableau}6 & 5 & 2\end{ytableau}$ & {\footnotesize$1400011$}
\\[-6pt] \\
$\begin{ytableau}6 & 4 & 3\end{ytableau}$ & {\footnotesize$004111$}
\\[-6pt] \\
$\begin{ytableau}6 & 5 & 4\end{ytableau}$ & {\footnotesize$0014011$}
\\[-6pt] \\
$\begin{ytableau}4 & 3 & 2 & 1\end{ytableau}$ & {\footnotesize$51111$}
\\[-6pt] \\
$\begin{ytableau}6 & 4 & 2 & 1\end{ytableau}$ & {\footnotesize$511101$}
\\[-6pt] \\
$\begin{ytableau}6 & 5 & 2 & 1\end{ytableau}$ & {\footnotesize$5110011$}
\\[-6pt] \\
$\begin{ytableau}6 & 4 & 3 & 2\end{ytableau}$ & {\footnotesize$150111$}
\\[-6pt] \\
$\begin{ytableau}6 & 5 & 4 & 2\end{ytableau}$ & {\footnotesize$1501011$}
\\[-6pt] \\
$\begin{ytableau}6 & 5 & 4 & 3\end{ytableau}$ & {\footnotesize$0051111$}
\\[-6pt] \\
$\begin{ytableau}6 & 4 & 3 & 2 & 1\end{ytableau}$ & {\footnotesize$611111$}
\end{tabular}
&
\ytableausetup{smalltableaux}
\begin{tabular}[t]{ll}
$T$ & $\alpha^{\Sp}(T)$ \\ \hline \\[-6pt]
$\begin{ytableau}6 & 5 & 4 & 2 & 1\end{ytableau}$ & {\footnotesize$6111011$}
\\[-6pt] \\
$\begin{ytableau}6 & 5 & 4 & 3 & 2\end{ytableau}$ & {\footnotesize$1601111$}
\\[-6pt] \\
$\begin{ytableau}6 & 5 & 4 & 3 & 2 & 1\end{ytableau}$ & {\footnotesize$7111111$}
\\[-6pt] \\
$\begin{ytableau}4 & 3 \\
\none & 2\end{ytableau}$ & {\footnotesize$333$}
\\[-6pt] \\
$\begin{ytableau}6 & 3 \\
\none & 2\end{ytableau}$ & {\footnotesize$333$}
\\[-6pt] \\
$\begin{ytableau}6 & 4 \\
\none & 2\end{ytableau}$ & {\footnotesize$3303$}
\\[-6pt] \\
$\begin{ytableau}6 & 5 \\
\none & 2\end{ytableau}$ & {\footnotesize$33003$}
\\[-6pt] \\
$\begin{ytableau}6 & 5 \\
\none & 4\end{ytableau}$ & {\footnotesize$00333$}
\\[-6pt] \\
$\begin{ytableau}4 & 3 & 1 \\
\none & 2 & \none\end{ytableau}$ & {\footnotesize$4331$}
\\[-6pt] \\
$\begin{ytableau}6 & 4 & 1 \\
\none & 2 & \none\end{ytableau}$ & {\footnotesize$4133$}
\\[-6pt] \\
$\begin{ytableau}6 & 5 & 1 \\
\none & 2 & \none\end{ytableau}$ & {\footnotesize$41303$}
\\[-6pt] \\
$\begin{ytableau}4 & 3 & 2 \\
\none & 2 & \none\end{ytableau}$ & {\footnotesize$34301$}
\\[-6pt] \\
$\begin{ytableau}6 & 4 & 3 \\
\none & 2 & \none\end{ytableau}$ & {\footnotesize$334001$}
\\[-6pt] \\
$\begin{ytableau}6 & 5 & 3 \\
\none & 2 & \none\end{ytableau}$ & {\footnotesize$3340001$}
\\[-6pt] \\
$\begin{ytableau}6 & 5 & 4 \\
\none & 2 & \none\end{ytableau}$ & {\footnotesize$3304001$}
\\[-6pt] \\
$\begin{ytableau}6 & 5 & 2 \\
\none & 4 & \none\end{ytableau}$ & {\footnotesize$14033$}
\\[-6pt] \\
$\begin{ytableau}6 & 5 & 3 \\
\none & 4 & \none\end{ytableau}$ & {\footnotesize$004331$}
\\[-6pt] \\
$\begin{ytableau}6 & 5 & 4 \\
\none & 4 & \none\end{ytableau}$ & {\footnotesize$0034301$}
\end{tabular}
&
\ytableausetup{smalltableaux}
\begin{tabular}[t]{ll}
$T$ & $\alpha^{\Sp}(T)$ \\ \hline \\[-6pt]
$\begin{ytableau}4 & 3 & 2 \\
\none & 2 & 1\end{ytableau}$ & {\footnotesize$4422$}
\\[-6pt] \\
$\begin{ytableau}6 & 4 & 3 \\
\none & 2 & 1\end{ytableau}$ & {\footnotesize$4242$}
\\[-6pt] \\
$\begin{ytableau}6 & 5 & 4 \\
\none & 2 & 1\end{ytableau}$ & {\footnotesize$4224$}
\\[-6pt] \\
$\begin{ytableau}6 & 5 & 3 \\
\none & 4 & 2\end{ytableau}$ & {\footnotesize$24402$}
\\[-6pt] \\
$\begin{ytableau}6 & 5 & 4 \\
\none & 4 & 2\end{ytableau}$ & {\footnotesize$24042$}
\\[-6pt] \\
$\begin{ytableau}6 & 5 & 4 \\
\none & 4 & 3\end{ytableau}$ & {\footnotesize$004422$}
\\[-6pt] \\
$\begin{ytableau}4 & 3 & 2 & 1 \\
\none & 2 & \none & \none\end{ytableau}$ & {\footnotesize$53311$}
\\[-6pt] \\
$\begin{ytableau}6 & 4 & 3 & 1 \\
\none & 2 & \none & \none\end{ytableau}$ & {\footnotesize$533101$}
\\[-6pt] \\
$\begin{ytableau}6 & 5 & 4 & 1 \\
\none & 2 & \none & \none\end{ytableau}$ & {\footnotesize$5133001$}
\\[-6pt] \\
$\begin{ytableau}6 & 5 & 2 & 1 \\
\none & 4 & \none & \none\end{ytableau}$ & {\footnotesize$51133$}
\\[-6pt] \\
$\begin{ytableau}6 & 4 & 3 & 2 \\
\none & 2 & \none & \none\end{ytableau}$ & {\footnotesize$353011$}
\\[-6pt] \\
$\begin{ytableau}6 & 5 & 4 & 3 \\
\none & 2 & \none & \none\end{ytableau}$ & {\footnotesize$3350011$}
\\[-6pt] \\
$\begin{ytableau}6 & 5 & 3 & 2 \\
\none & 4 & \none & \none\end{ytableau}$ & {\footnotesize$150331$}
\\[-6pt] \\
$\begin{ytableau}6 & 5 & 4 & 2 \\
\none & 4 & \none & \none\end{ytableau}$ & {\footnotesize$1503301$}
\\[-6pt] \\
$\begin{ytableau}6 & 5 & 4 & 3 \\
\none & 4 & \none & \none\end{ytableau}$ & {\footnotesize$0053311$}
\\[-6pt] \\
$\begin{ytableau}4 & 3 & 2 & 1 \\
\none & 2 & 1 & \none\end{ytableau}$ & {\footnotesize$54221$}
\\[-6pt] \\
$\begin{ytableau}6 & 5 & 3 & 1 \\
\none & 4 & 2 & \none\end{ytableau}$ & {\footnotesize$52412$}
\end{tabular}
&
\ytableausetup{smalltableaux}
\begin{tabular}[t]{ll}
$T$ & $\alpha^{\Sp}(T)$ \\ \hline \\[-6pt]
$\begin{ytableau}6 & 5 & 4 & 1 \\
\none & 4 & 2 & \none\end{ytableau}$ & {\footnotesize$51242$}
\\[-6pt] \\
$\begin{ytableau}6 & 4 & 3 & 2 \\
\none & 2 & 1 & \none\end{ytableau}$ & {\footnotesize$452201$}
\\[-6pt] \\
$\begin{ytableau}6 & 5 & 4 & 3 \\
\none & 2 & 1 & \none\end{ytableau}$ & {\footnotesize$4252001$}
\\[-6pt] \\
$\begin{ytableau}6 & 5 & 3 & 2 \\
\none & 4 & 2 & \none\end{ytableau}$ & {\footnotesize$254021$}
\\[-6pt] \\
$\begin{ytableau}6 & 5 & 4 & 3 \\
\none & 4 & 2 & \none\end{ytableau}$ & {\footnotesize$2450201$}
\\[-6pt] \\
$\begin{ytableau}6 & 5 & 4 & 2 \\
\none & 4 & 3 & \none\end{ytableau}$ & {\footnotesize$150422$}
\\[-6pt] \\
$\begin{ytableau}6 & 5 & 4 & 3 \\
\none & 4 & 3 & \none\end{ytableau}$ & {\footnotesize$0054221$}
\\[-6pt] \\
$\begin{ytableau}6 & 5 & 3 & 2 \\
\none & 4 & 2 & 1\end{ytableau}$ & {\footnotesize$55222$}
\\[-6pt] \\
$\begin{ytableau}6 & 5 & 4 & 3 \\
\none & 4 & 2 & 1\end{ytableau}$ & {\footnotesize$52522$}
\\[-6pt] \\
$\begin{ytableau}6 & 5 & 4 & 3 \\
\none & 4 & 3 & 2\end{ytableau}$ & {\footnotesize$255022$}
\\[-6pt] \\
$\begin{ytableau}6 & 4 & 3 & 2 & 1 \\
\none & 2 & \none & \none & \none\end{ytableau}$ & {\footnotesize$633111$}
\\[-6pt] \\
$\begin{ytableau}6 & 5 & 4 & 3 & 1 \\
\none & 2 & \none & \none & \none\end{ytableau}$ & {\footnotesize$6331011$}
\\[-6pt] \\
$\begin{ytableau}6 & 5 & 3 & 2 & 1 \\
\none & 4 & \none & \none & \none\end{ytableau}$ & {\footnotesize$611331$}
\\[-6pt] \\
$\begin{ytableau}6 & 5 & 4 & 2 & 1 \\
\none & 4 & \none & \none & \none\end{ytableau}$ & {\footnotesize$6113301$}
\\[-6pt] \\
$\begin{ytableau}6 & 5 & 4 & 3 & 2 \\
\none & 2 & \none & \none & \none\end{ytableau}$ & {\footnotesize$3630111$}
\\[-6pt] \\
$\begin{ytableau}6 & 5 & 4 & 3 & 2 \\
\none & 4 & \none & \none & \none\end{ytableau}$ & {\footnotesize$1603311$}
\\[-6pt] \\
$\begin{ytableau}6 & 4 & 3 & 2 & 1 \\
\none & 2 & 1 & \none & \none\end{ytableau}$ & {\footnotesize$642211$}
\end{tabular}
\end{tabular}
\end{center}
\caption{Some $\Sp$-reduced tableaux with the weak compositions predicted in Conjecture~\ref{sp-demazure-conj2}.}\label{sp-fig1}
\end{figure}
\begin{figure}[h]
\begin{center}
\begin{tabular}{l|l|l}
\ytableausetup{smalltableaux}
\begin{tabular}[t]{ll}
$T$ & $\alpha^{\Sp}(T)$ \\ \hline \\[-6pt]
$\begin{ytableau}6 & 5 & 3 & 2 & 1 \\
\none & 4 & 2 & \none & \none\end{ytableau}$ & {\footnotesize$624121$}
\\[-6pt] \\
$\begin{ytableau}6 & 5 & 4 & 3 & 1 \\
\none & 4 & 2 & \none & \none\end{ytableau}$ & {\footnotesize$6241201$}
\\[-6pt] \\
$\begin{ytableau}6 & 5 & 4 & 2 & 1 \\
\none & 4 & 3 & \none & \none\end{ytableau}$ & {\footnotesize$611422$}
\\[-6pt] \\
$\begin{ytableau}6 & 5 & 4 & 3 & 2 \\
\none & 2 & 1 & \none & \none\end{ytableau}$ & {\footnotesize$4622011$}
\\[-6pt] \\
$\begin{ytableau}6 & 5 & 4 & 3 & 2 \\
\none & 4 & 2 & \none & \none\end{ytableau}$ & {\footnotesize$2640211$}
\\[-6pt] \\
$\begin{ytableau}6 & 5 & 4 & 3 & 2 \\
\none & 4 & 3 & \none & \none\end{ytableau}$ & {\footnotesize$1604221$}
\\[-6pt] \\
$\begin{ytableau}6 & 5 & 3 & 2 & 1 \\
\none & 4 & 2 & 1 & \none\end{ytableau}$ & {\footnotesize$652221$}
\\[-6pt] \\
$\begin{ytableau}6 & 5 & 4 & 3 & 1 \\
\none & 4 & 3 & 2 & \none\end{ytableau}$ & {\footnotesize$625122$}
\\[-6pt] \\
$\begin{ytableau}6 & 5 & 4 & 3 & 2 \\
\none & 4 & 2 & 1 & \none\end{ytableau}$ & {\footnotesize$5622201$}
\\[-6pt] \\
$\begin{ytableau}6 & 5 & 4 & 3 & 2 \\
\none & 4 & 3 & 2 & \none\end{ytableau}$ & {\footnotesize$2650221$}
\\[-6pt] \\
$\begin{ytableau}6 & 5 & 4 & 3 & 2 \\
\none & 4 & 3 & 2 & 1\end{ytableau}$ & {\footnotesize$662222$}
\\[-6pt] \\
$\begin{ytableau}6 & 5 & 4 & 3 & 2 & 1 \\
\none & 2 & \none & \none & \none & \none\end{ytableau}$ & {\footnotesize$7331111$}
\\[-6pt] \\
$\begin{ytableau}6 & 5 & 4 & 3 & 2 & 1 \\
\none & 4 & \none & \none & \none & \none\end{ytableau}$ & {\footnotesize$7113311$}
\\[-6pt] \\
$\begin{ytableau}6 & 5 & 4 & 3 & 2 & 1 \\
\none & 2 & 1 & \none & \none & \none\end{ytableau}$ & {\footnotesize$7422111$}
\\[-6pt] \\
$\begin{ytableau}6 & 5 & 4 & 3 & 2 & 1 \\
\none & 4 & 2 & \none & \none & \none\end{ytableau}$ & {\footnotesize$7241211$}
\\[-6pt] \\
$\begin{ytableau}6 & 5 & 4 & 3 & 2 & 1 \\
\none & 4 & 3 & \none & \none & \none\end{ytableau}$ & {\footnotesize$7114221$}
\\[-6pt] \\
$\begin{ytableau}6 & 5 & 4 & 3 & 2 & 1 \\
\none & 4 & 2 & 1 & \none & \none\end{ytableau}$ & {\footnotesize$7522211$}
\end{tabular}
&
\ytableausetup{smalltableaux}
\begin{tabular}[t]{ll}
$T$ & $\alpha^{\Sp}(T)$ \\ \hline \\[-6pt]
$\begin{ytableau}6 & 5 & 4 & 3 & 2 & 1 \\
\none & 4 & 3 & 2 & \none & \none\end{ytableau}$ & {\footnotesize$7251221$}
\\[-6pt] \\
$\begin{ytableau}6 & 5 & 4 & 3 & 2 & 1 \\
\none & 4 & 3 & 2 & 1 & \none\end{ytableau}$ & {\footnotesize$7622221$}
\\[-6pt] \\
$\begin{ytableau}6 & 5 & 4 \\
\none & 4 & 3 \\
\none & \none & 2\end{ytableau}$ & {\footnotesize$4444$}
\\[-6pt] \\
$\begin{ytableau}6 & 5 & 4 & 1 \\
\none & 4 & 3 & \none \\
\none & \none & 2 & \none\end{ytableau}$ & {\footnotesize$54441$}
\\[-6pt] \\
$\begin{ytableau}6 & 5 & 4 & 2 \\
\none & 4 & 3 & \none \\
\none & \none & 2 & \none\end{ytableau}$ & {\footnotesize$454401$}
\\[-6pt] \\
$\begin{ytableau}6 & 5 & 4 & 3 \\
\none & 4 & 3 & \none \\
\none & \none & 2 & \none\end{ytableau}$ & {\footnotesize$4454001$}
\\[-6pt] \\
$\begin{ytableau}6 & 5 & 4 & 2 \\
\none & 4 & 3 & 1 \\
\none & \none & 2 & \none\end{ytableau}$ & {\footnotesize$55442$}
\\[-6pt] \\
$\begin{ytableau}6 & 5 & 4 & 3 \\
\none & 4 & 3 & 1 \\
\none & \none & 2 & \none\end{ytableau}$ & {\footnotesize$54542$}
\\[-6pt] \\
$\begin{ytableau}6 & 5 & 4 & 3 \\
\none & 4 & 3 & 2 \\
\none & \none & 2 & \none\end{ytableau}$ & {\footnotesize$455402$}
\\[-6pt] \\
$\begin{ytableau}6 & 5 & 4 & 3 \\
\none & 4 & 3 & 2 \\
\none & \none & 2 & 1\end{ytableau}$ & {\footnotesize$55533$}
\\[-6pt] \\
$\begin{ytableau}6 & 5 & 4 & 2 & 1 \\
\none & 4 & 3 & \none & \none \\
\none & \none & 2 & \none & \none\end{ytableau}$ & {\footnotesize$644411$}
\\[-6pt] \\
$\begin{ytableau}6 & 5 & 4 & 3 & 1 \\
\none & 4 & 3 & \none & \none \\
\none & \none & 2 & \none & \none\end{ytableau}$ & {\footnotesize$6444101$}
\\[-6pt] \\
$\begin{ytableau}6 & 5 & 4 & 3 & 2 \\
\none & 4 & 3 & \none & \none \\
\none & \none & 2 & \none & \none\end{ytableau}$ & {\footnotesize$4644011$}
\\[-6pt] \\
$\begin{ytableau}6 & 5 & 4 & 2 & 1 \\
\none & 4 & 3 & 1 & \none \\
\none & \none & 2 & \none & \none\end{ytableau}$ & {\footnotesize$654421$}
\end{tabular}
&
\ytableausetup{smalltableaux}
\begin{tabular}[t]{ll}
$T$ & $\alpha^{\Sp}(T)$ \\ \hline \\[-6pt]
$\begin{ytableau}6 & 5 & 4 & 3 & 1 \\
\none & 4 & 3 & 2 & \none \\
\none & \none & 2 & \none & \none\end{ytableau}$ & {\footnotesize$645412$}
\\[-6pt] \\
$\begin{ytableau}6 & 5 & 4 & 3 & 2 \\
\none & 4 & 3 & 1 & \none \\
\none & \none & 2 & \none & \none\end{ytableau}$ & {\footnotesize$5644201$}
\\[-6pt] \\
$\begin{ytableau}6 & 5 & 4 & 3 & 2 \\
\none & 4 & 3 & 2 & \none \\
\none & \none & 2 & \none & \none\end{ytableau}$ & {\footnotesize$4654021$}
\\[-6pt] \\
$\begin{ytableau}6 & 5 & 4 & 3 & 1 \\
\none & 4 & 3 & 2 & \none \\
\none & \none & 2 & 1 & \none\end{ytableau}$ & {\footnotesize$655331$}
\\[-6pt] \\
$\begin{ytableau}6 & 5 & 4 & 3 & 2 \\
\none & 4 & 3 & 2 & \none \\
\none & \none & 2 & 1 & \none\end{ytableau}$ & {\footnotesize$5653301$}
\\[-6pt] \\
$\begin{ytableau}6 & 5 & 4 & 3 & 2 \\
\none & 4 & 3 & 2 & 1 \\
\none & \none & 2 & \none & \none\end{ytableau}$ & {\footnotesize$664422$}
\\[-6pt] \\
$\begin{ytableau}6 & 5 & 4 & 3 & 2 \\
\none & 4 & 3 & 2 & 1 \\
\none & \none & 2 & 1 & \none\end{ytableau}$ & {\footnotesize$665332$}
\\[-6pt] \\
$\begin{ytableau}6 & 5 & 4 & 3 & 2 & 1 \\
\none & 4 & 3 & \none & \none & \none \\
\none & \none & 2 & \none & \none & \none\end{ytableau}$ & {\footnotesize$7444111$}
\\[-6pt] \\
$\begin{ytableau}6 & 5 & 4 & 3 & 2 & 1 \\
\none & 4 & 3 & 1 & \none & \none \\
\none & \none & 2 & \none & \none & \none\end{ytableau}$ & {\footnotesize$7544211$}
\\[-6pt] \\
$\begin{ytableau}6 & 5 & 4 & 3 & 2 & 1 \\
\none & 4 & 3 & 2 & \none & \none \\
\none & \none & 2 & \none & \none & \none\end{ytableau}$ & {\footnotesize$7454121$}
\\[-6pt] \\
$\begin{ytableau}6 & 5 & 4 & 3 & 2 & 1 \\
\none & 4 & 3 & 2 & \none & \none \\
\none & \none & 2 & 1 & \none & \none\end{ytableau}$ & {\footnotesize$7553311$}
\\[-6pt] \\
$\begin{ytableau}6 & 5 & 4 & 3 & 2 & 1 \\
\none & 4 & 3 & 2 & 1 & \none \\
\none & \none & 2 & \none & \none & \none\end{ytableau}$ & {\footnotesize$7644221$}
\\[-6pt] \\
$\begin{ytableau}6 & 5 & 4 & 3 & 2 & 1 \\
\none & 4 & 3 & 2 & 1 & \none \\
\none & \none & 2 & 1 & \none & \none\end{ytableau}$ & {\footnotesize$7653321$}
\end{tabular}
\end{tabular}
\end{center}
\caption{Some $\Sp$-reduced tableaux with the weak compositions predicted in Conjecture~\ref{sp-demazure-conj2}.}\label{sp-fig2}
\end{figure}

\begin{figure}[h]
\begin{center}
\begin{tabular}{l|l|l|l}
\ytableausetup{smalltableaux}
\begin{tabular}[t]{ll}
$T$ & $\alpha^{\O}(T)$ \\ \hline \\[-6pt]
$\begin{ytableau}3 & 2 \\
\none & 1\end{ytableau}$ & {\footnotesize$22$}
\\[-6pt] \\
$\begin{ytableau}4 & 2 \\
\none & 1\end{ytableau}$ & {\footnotesize$22$}
\\[-6pt] \\
$\begin{ytableau}5 & 2 \\
\none & 1\end{ytableau}$ & {\footnotesize$22$}
\\[-6pt] \\
$\begin{ytableau}4 & 3 \\
\none & 1\end{ytableau}$ & {\footnotesize$202$}
\\[-6pt] \\
$\begin{ytableau}5 & 3' \\
\none & 1\end{ytableau}$ & {\footnotesize$202$}
\\[-6pt] \\
$\begin{ytableau}5 & 3 \\
\none & 1\end{ytableau}$ & {\footnotesize$202$}
\\[-6pt] \\
$\begin{ytableau}5 & 4 \\
\none & 1\end{ytableau}$ & {\footnotesize$2002$}
\\[-6pt] \\
$\begin{ytableau}4 & 3 \\
\none & 2\end{ytableau}$ & {\footnotesize$022$}
\\[-6pt] \\
$\begin{ytableau}5 & 3 \\
\none & 2\end{ytableau}$ & {\footnotesize$022$}
\\[-6pt] \\
$\begin{ytableau}5 & 4 \\
\none & 2\end{ytableau}$ & {\footnotesize$0202$}
\\[-6pt] \\
$\begin{ytableau}5 & 4 \\
\none & 3\end{ytableau}$ & {\footnotesize$0022$}
\\[-6pt] \\
$\begin{ytableau}3 & 2 & 1 \\
\none & 1 & \none\end{ytableau}$ & {\footnotesize$321$}
\\[-6pt] \\
$\begin{ytableau}4 & 3 & 1 \\
\none & 2 & \none\end{ytableau}$ & {\footnotesize$312$}
\\[-6pt] \\
$\begin{ytableau}5 & 3 & 1 \\
\none & 2 & \none\end{ytableau}$ & {\footnotesize$312$}
\\[-6pt] \\
$\begin{ytableau}5 & 4 & 1 \\
\none & 2 & \none\end{ytableau}$ & {\footnotesize$3102$}
\\[-6pt] \\
$\begin{ytableau}5 & 4 & 1' \\
\none & 3 & \none\end{ytableau}$ & {\footnotesize$3012$}
\\[-6pt] \\
$\begin{ytableau}5 & 4 & 1 \\
\none & 3 & \none\end{ytableau}$ & {\footnotesize$3012$}
\\[-6pt] \\
$\begin{ytableau}4 & 3 & 2 \\
\none & 1 & \none\end{ytableau}$ & {\footnotesize$2301$}
\\[-6pt] \\
$\begin{ytableau}5 & 4 & 2 \\
\none & 1 & \none\end{ytableau}$ & {\footnotesize$23001$}
\end{tabular}
&
\ytableausetup{smalltableaux}
\begin{tabular}[t]{ll}
$T$ & $\alpha^{\O}(T)$ \\ \hline \\[-6pt]
$\begin{ytableau}5 & 3 & 2 \\
\none & 1 & \none\end{ytableau}$ & {\footnotesize$23001$}
\\[-6pt] \\
$\begin{ytableau}5 & 3' & 2 \\
\none & 1 & \none\end{ytableau}$ & {\footnotesize$23001$}
\\[-6pt] \\
$\begin{ytableau}5 & 4 & 3 \\
\none & 1 & \none\end{ytableau}$ & {\footnotesize$20301$}
\\[-6pt] \\
$\begin{ytableau}4 & 3 & 2 \\
\none & 2 & \none\end{ytableau}$ & {\footnotesize$0321$}
\\[-6pt] \\
$\begin{ytableau}5 & 4 & 2 \\
\none & 3 & \none\end{ytableau}$ & {\footnotesize$0312$}
\\[-6pt] \\
$\begin{ytableau}5 & 4 & 3 \\
\none & 2 & \none\end{ytableau}$ & {\footnotesize$02301$}
\\[-6pt] \\
$\begin{ytableau}5 & 4 & 3 \\
\none & 3 & \none\end{ytableau}$ & {\footnotesize$00321$}
\\[-6pt] \\
$\begin{ytableau}4 & 3 & 2 \\
\none & 2 & 1\end{ytableau}$ & {\footnotesize$332$}
\\[-6pt] \\
$\begin{ytableau}5 & 3 & 2' \\
\none & 2 & 1\end{ytableau}$ & {\footnotesize$332$}
\\[-6pt] \\
$\begin{ytableau}5 & 3 & 2 \\
\none & 2 & 1\end{ytableau}$ & {\footnotesize$332$}
\\[-6pt] \\
$\begin{ytableau}5 & 4 & 2 \\
\none & 3 & 1'\end{ytableau}$ & {\footnotesize$3302$}
\\[-6pt] \\
$\begin{ytableau}5 & 4 & 2 \\
\none & 3 & 1\end{ytableau}$ & {\footnotesize$3302$}
\\[-6pt] \\
$\begin{ytableau}5 & 4 & 3 \\
\none & 2 & 1\end{ytableau}$ & {\footnotesize$323$}
\\[-6pt] \\
$\begin{ytableau}5 & 4 & 3 \\
\none & 3 & 1\end{ytableau}$ & {\footnotesize$3032$}
\\[-6pt] \\
$\begin{ytableau}5 & 4 & 3 \\
\none & 3 & 1'\end{ytableau}$ & {\footnotesize$3032$}
\\[-6pt] \\
$\begin{ytableau}5 & 4 & 3 \\
\none & 3 & 2\end{ytableau}$ & {\footnotesize$0332$}
\\[-6pt] \\
$\begin{ytableau}4 & 3 & 2 & 1 \\
\none & 1 & \none & \none\end{ytableau}$ & {\footnotesize$4211$}
\\[-6pt] \\
$\begin{ytableau}5 & 3' & 2 & 1 \\
\none & 1 & \none & \none\end{ytableau}$ & {\footnotesize$42101$}
\\[-6pt] \\
$\begin{ytableau}5 & 3 & 2 & 1 \\
\none & 1 & \none & \none\end{ytableau}$ & {\footnotesize$42101$}
\end{tabular}
&
\ytableausetup{smalltableaux}
\begin{tabular}[t]{ll}
$T$ & $\alpha^{\O}(T)$ \\ \hline \\[-6pt]
$\begin{ytableau}4 & 3 & 2 & 1 \\
\none & 2 & \none & \none\end{ytableau}$ & {\footnotesize$4121$}
\\[-6pt] \\
$\begin{ytableau}5 & 4 & 3 & 1 \\
\none & 2 & \none & \none\end{ytableau}$ & {\footnotesize$41201$}
\\[-6pt] \\
$\begin{ytableau}5 & 4 & 2 & 1 \\
\none & 3 & \none & \none\end{ytableau}$ & {\footnotesize$4112$}
\\[-6pt] \\
$\begin{ytableau}5 & 4 & 3 & 1 \\
\none & 3 & \none & \none\end{ytableau}$ & {\footnotesize$40121$}
\\[-6pt] \\
$\begin{ytableau}5 & 4 & 3 & 1' \\
\none & 3 & \none & \none\end{ytableau}$ & {\footnotesize$40121$}
\\[-6pt] \\
$\begin{ytableau}5 & 4 & 3 & 2 \\
\none & 1 & \none & \none\end{ytableau}$ & {\footnotesize$24011$}
\\[-6pt] \\
$\begin{ytableau}5 & 4 & 3 & 2 \\
\none & 2 & \none & \none\end{ytableau}$ & {\footnotesize$04211$}
\\[-6pt] \\
$\begin{ytableau}5 & 4 & 3 & 2 \\
\none & 3 & \none & \none\end{ytableau}$ & {\footnotesize$04121$}
\\[-6pt] \\
$\begin{ytableau}4 & 3 & 2 & 1 \\
\none & 2 & 1 & \none\end{ytableau}$ & {\footnotesize$4321$}
\\[-6pt] \\
$\begin{ytableau}5 & 4 & 2 & 1 \\
\none & 3 & 1 & \none\end{ytableau}$ & {\footnotesize$4312$}
\\[-6pt] \\
$\begin{ytableau}5 & 4 & 2 & 1 \\
\none & 3 & 1' & \none\end{ytableau}$ & {\footnotesize$4312$}
\\[-6pt] \\
$\begin{ytableau}5 & 4 & 3 & 1 \\
\none & 3 & 2 & \none\end{ytableau}$ & {\footnotesize$4132$}
\\[-6pt] \\
$\begin{ytableau}5 & 4 & 3 & 2 \\
\none & 2 & 1 & \none\end{ytableau}$ & {\footnotesize$34201$}
\\[-6pt] \\
$\begin{ytableau}5 & 4 & 3 & 2 \\
\none & 3 & 1' & \none\end{ytableau}$ & {\footnotesize$34021$}
\\[-6pt] \\
$\begin{ytableau}5 & 4 & 3 & 2 \\
\none & 3 & 1 & \none\end{ytableau}$ & {\footnotesize$34021$}
\\[-6pt] \\
$\begin{ytableau}5 & 4 & 3 & 2 \\
\none & 3 & 2 & \none\end{ytableau}$ & {\footnotesize$04321$}
\\[-6pt] \\
$\begin{ytableau}5 & 4 & 3 & 2 \\
\none & 3 & 2 & 1\end{ytableau}$ & {\footnotesize$4422$}
\\[-6pt] \\
$\begin{ytableau}5 & 4 & 3 & 2 & 1 \\
\none & 1 & \none & \none & \none\end{ytableau}$ & {\footnotesize$52111$}
\\[-6pt] \\
$\begin{ytableau}5 & 4 & 3 & 2 & 1 \\
\none & 2 & \none & \none & \none\end{ytableau}$ & {\footnotesize$51211$}
\end{tabular}
&
\ytableausetup{smalltableaux}
\begin{tabular}[t]{ll}
$T$ & $\alpha^{\O}(T)$ \\ \hline \\[-6pt]
$\begin{ytableau}5 & 4 & 3 & 2 & 1 \\
\none & 3 & \none & \none & \none\end{ytableau}$ & {\footnotesize$51121$}
\\[-6pt] \\
$\begin{ytableau}5 & 4 & 3 & 2 & 1 \\
\none & 2 & 1 & \none & \none\end{ytableau}$ & {\footnotesize$53211$}
\\[-6pt] \\
$\begin{ytableau}5 & 4 & 3 & 2 & 1 \\
\none & 3 & 1' & \none & \none\end{ytableau}$ & {\footnotesize$53121$}
\\[-6pt] \\
$\begin{ytableau}5 & 4 & 3 & 2 & 1 \\
\none & 3 & 1 & \none & \none\end{ytableau}$ & {\footnotesize$53121$}
\\[-6pt] \\
$\begin{ytableau}5 & 4 & 3 & 2 & 1 \\
\none & 3 & 2 & \none & \none\end{ytableau}$ & {\footnotesize$51321$}
\\[-6pt] \\
$\begin{ytableau}5 & 4 & 3 & 2 & 1 \\
\none & 3 & 2 & 1 & \none\end{ytableau}$ & {\footnotesize$54221$}
\\[-6pt] \\
$\begin{ytableau}5 & 4 & 3 \\
\none & 3 & 2 \\
\none & \none & 1\end{ytableau}$ & {\footnotesize$333$}
\\[-6pt] \\
$\begin{ytableau}5 & 4 & 3 & 1 \\
\none & 3 & 2 & \none \\
\none & \none & 1 & \none\end{ytableau}$ & {\footnotesize$4331$}
\\[-6pt] \\
$\begin{ytableau}5 & 4 & 3 & 2 \\
\none & 3 & 2 & \none \\
\none & \none & 1 & \none\end{ytableau}$ & {\footnotesize$34301$}
\\[-6pt] \\
$\begin{ytableau}5 & 4 & 3 & 2 \\
\none & 3 & 2 & 1 \\
\none & \none & 1 & \none\end{ytableau}$ & {\footnotesize$4432$}
\\[-6pt] \\
$\begin{ytableau}5 & 4 & 3 & 2 & 1 \\
\none & 3 & 2 & \none & \none \\
\none & \none & 1 & \none & \none\end{ytableau}$ & {\footnotesize$53311$}
\\[-6pt] \\
$\begin{ytableau}5 & 4 & 3 & 2 & 1 \\
\none & 3 & 2 & 1 & \none \\
\none & \none & 1 & \none & \none\end{ytableau}$ & {\footnotesize$54321$}
\end{tabular}
\end{tabular}
\end{center}
\caption{Some $\O$-reduced tableaux with the weak compositions predicted in Conjecture~\ref{o-demazure-conj2}.}\label{o-fig}
\end{figure}

%
%
%
%
%
%
%
%
%
%

 \printbibliography

\end{document}